\documentclass{amsart}
\usepackage{graphicx} 
\usepackage[utf8]{inputenc}
\usepackage{amsmath}
\usepackage{xcolor}
\usepackage{amssymb}
\usepackage{amsmath,stackengine}
\usepackage{amsfonts}
\usepackage{amstext}
\usepackage{amsthm}
\usepackage[english]{babel}
\usepackage[autostyle]{csquotes}
\usepackage[super]{nth}
\usepackage{hyperref}
     \usepackage[backend=bibtex]{biblatex}
\newtheorem{theorem}{Theorem}[section]
\newtheorem{definition}[theorem]{Definition}
\newtheorem{lemma}[theorem]{Lemma}
\newtheorem{remark}[theorem]{Remark}
\newtheorem{corollary}[theorem]{Corollary}
\newtheorem{proposition}[theorem]{Proposition}

\newcommand{\R}{\mathbb{R}}

\newcommand{\slap}{(-\Delta)^{s}}
\newcommand{\rn}{\mathbb{R}^{n}}
\newcommand{\g}[1]{\mathcal{G}_{#1}}
\newcommand{\gk}[1]{\mathbb{G}_{#1}}
\newcommand{\Ds}{D_{s}}

\newcommand{\ul}{u_{\ell}}

\hypersetup{
    colorlinks=true,
    linkcolor= blue,
    citecolor=cyan,
    pdfpagemode=FullScreen,
    }

\usepackage{geometry}
{%
\setcounter{enumi}{0}

\begin{enumerate}}%
{\end{enumerate} }  
\makeatletter
\newcommand{\myitem}[1]{%
\medskip \item[#1]\protected@edef\@currentlabel{#1}%
} 
\title{Large Solutions for Fractional Laplacian on Infinite Cylindrical Domains}
\author[I. Chowdhury]{Indranil Chowdhury}
\address{\parbox{.8\linewidth}
{{\textbf{I. Chowdhury}}\medskip \\
Indian Institute of Technology - Kanpur, India \medskip}}
\curraddr{}
\email{indranil@iitk.ac.in}

\author[N.  Dattatreya]{N N Dattatreya}
\address{\parbox{.8\linewidth}
{{\textbf{N.  Dattatreya}}\medskip \\
Indian Institute of Technology - Kanpur, India \medskip}}
\curraddr{}
\email{dattatreya21@iitk.ac.in}
\date{}

\smallskip

\begin{document}
\begin{abstract}

We investigate large solutions of linear and semi-linear equations involving the \emph{fractional Laplacian} on domains that are becoming unbounded in some, but not all, directions. Solutions that blow up on the boundary of a domain are commonly called large solutions. For local operators, such behaviour arises only in the presence of lower-order nonlinear terms.  However, for nonlocal operators, the existence of such solutions is more subtle. In contrast to the local case, the nonlocal nature of the operator gives rise to different boundary blow-up phenomena even in the case of linear equations. In this article, we study the existence and qualitative behaviour of boundary blow-up solutions to fractional linear and semi-linear equations on finite cylindrical domains and employ these properties to construct \emph{large solutions} on \emph{infinite cylinders}.

\end{abstract}
\maketitle

\section{Introduction}
The primary objective of this work is to study large solutions of linear and semi-linear equations involving the fractional Laplacian, on the domain that is unbounded in some, but not all, directions. The solutions that exhibit blow-up phenomena near the boundary of domains are called \textit{large} solutions.
Let $n\geq 2$, $1\leq m<n$, and $\Omega_{\infty}:=B^{m}_{1}(0)\times \R^{n-m} \subset \R^{n}$ be the infinite cylinder. For $0<s<1$, we consider
\begin{equation}\label{semi-linear prob on infinite cylinder}
    \begin{cases}
         \slap u=-f(x,u)\quad &\text{in } \Omega_{\infty},\\
         u=g &\text{in }\R^{n}\setminus\overline{\Omega}_{\infty}, \\
         E_{\Omega_{\infty}}(u)=h &\text{on }\partial\Omega_{\infty}.
    \end{cases}
 \end{equation}
Here,  $E_{\Omega_{\infty}}$ is a singular boundary trace-like operator defined in \eqref{dfn E}. 
We recall that 
\begin{equation*}
    \slap u(x)=\mathcal{A}(n,s) \text{ p.v} \int_{\R^{n}}\frac{u(x)-u(y)}{|x-y|^{n+2s}}\ dy,\quad \mathcal{A}(n,s):=\frac{\Gamma(n/2+s)}{\pi^{n/2}\Gamma(2-s)}s(1-s).
\end{equation*}
Unlike local differential operators, the value of $\slap u(x)$ depends on the behaviour of $u$ throughout $\rn$, reflecting the nonlocal nature of the operator.

The \textit{large} solutions are also known as \textit{boundary blow-up} solutions. Study of these solutions in the fractional setting has received significant attention over the past decade, due to their rich behaviour and the variety of blow-up occurrences that may arise compared with the classical local theory.

The study of large solutions for local operators, such as the Laplacian, dates back more than a century, see \cite{biberbach}. Let $\Omega\subset \rn$. Consider the equation 
\begin{equation*}
    \begin{cases}
        \Delta u=f(u) \quad &\text{in }\Omega\\
        u(x)\to +\infty  &\text{as }d(x,\partial\Omega)\to 0.
    \end{cases}
\end{equation*}
The existence of a solution to the above equation is guaranteed only when the nonlinear term satisfies the well-known \textit{Keller-Osserman condition}, derived independently in \cite{Keller1957} and \cite{Osserman1957}. The study in \cite{Keller1957} is motivated by problems in electrodynamics, whereas \cite{Osserman1957} is motivated by problems in geometry. For  the connection between conformal geometry and boundary blow-up solutions, see the seminal work of Loewner and Nirenberg \cite{LoewnerNirenberg1974}. The literature regarding large solutions for various operators in the local setting is extensive, addressing contexts such as existence, boundary behaviour, uniqueness, and gradient estimates; we refer to \cite{bandlemarcus4, bandlemarcus3, Lieberman, diaz1993,DumontDupaigneOlivierRuadulescu2007} and the references therein. A related development connecting large solutions to stochastic control can be found in \cite{lions}.


\medskip 

The nonlocal setting is remarkably different. 
One even exhibits a function solving $(-\triangle)^{s} u =0$ in the unit ball $B_{1}(0) \subset \mathbb{R}^{n}$, that blows up on the boundary $\partial B_{1}(0)$ from inside and vanishes identically outside $\overline{B_{1}(0)}$. We refer to \cite{Abatangelo2015LargesHarmonic, Bogdan-book} for further discussions.
Unlike local operators, the existence of a large solution requires careful understanding even in the case of a linear equation, including the boundary regularity. For further reading in these directions, we refer to \cite{ROSOTON2014275, Grubb2015, RosSerra2016, AbatangeloVazquez2023} and the references therein.

The first work on boundary blow-up solutions involving nonlocal operators is by Felmer and Quaas in \cite{FelmerQuass2012}. They studied well-posedness of large solutions for a semi-linear equation with power type nonlinearity, in the viscosity solution framework, see also \cite{ChenFelmerQuass2015} for related results.
The existence of $L^{1}$ weak solutions to linear and semi-linear equations involving the fractional Laplacian is carried out in \cite{Abatangelo2015LargesHarmonic}. In \cite{Abatangelo2015LargesHarmonic}, the problem is posed with a singular trace-like operator $E_{\Omega}$ on the boundary and is given by
\begin{equation}\label{eqn main semilinear omega}
\begin{cases}
\slap u=-f(x,u) \quad &\text{in }\Omega\\
u(x)=g(x) &\text{in }\R^{n}\setminus\overline{\Omega}\\
E_{\Omega}u(\theta)=h(\theta) &\text{on }\partial\Omega.
\end{cases}
\end{equation} 
This boundary condition was previously studied in \cite{Grubb2014}. It is also important to notice that in \cite{BogdanKulczyckiMateusz2008} a representation of s-harmonic function is obtained via Martin kernel, which amounts to the trace operator $E_{\Omega}$. For the linear case with $-f(x,t)=f(x)$, the author \cite{Abatangelo2015LargesHarmonic} studied well-posedness of \eqref{eqn main semilinear omega} along with its representation via Green's, Poisson, and Martin kernels. These large solutions can be generated independently by each of the three data sets appearing in the problem. More precisely:

\begin{enumerate}
    \item[--] The positive boundary trace data $h$: The corresponding solution is necessarily large, regardless of values of $f$ and $g$. Even if they both are zero.
    \item[--] When $f(x)=\delta_{\Omega}^{-\beta}$, $2s<\beta<1+s$: The linear equation admits a large solution for non-negative complement and boundary data; even when $g=0$ and $h=0$.
    \item[--] When $g$ blows up near the boundary: The corresponding solution of the linear equation is large even if $f=0$ and $h=0$.
\end{enumerate}
 Further, for the semi-linear equation, the existence of a large solution for a suitable class of $f(x,t)$ is established when $h>0$.

In continuation of \cite{Abatangelo2015LargesHarmonic}, the weak dual notion of solution to a linear equation, for a general class of nonlocal operators, is studied in \cite{AbatangeloVazquez2023} for the case $g=0$. Moreover, the authors show that boundary blow-up arises naturally as a limiting phenomenon of the corresponding interior theory, justifying that the observations in \cite{Abatangelo2015LargesHarmonic}, for the fractional Laplacian, are a byproduct of nonlocal phenomena. Therefore, these types of large solutions are not related to Keller-Osserman type condition on nonlinear force terms which is a necessary and sufficient condition to get a large solution for local operators. 
Boundary blow-up solutions arising from a Keller-Osserman type condition, referred to as `very large' solutions, are studied in \cite{ChenFelmerQuass2015, Abaratgelo2017,Chowdhury2026}. We do not study the latter type of large solution in this article. The work relating to representation of large solution by Green's, Poisson and Martin kernels, to measure data can be found in \cite{BogdanKulczyckiMateusz2008, BogdanJarohsKania2020, HansenBogdan2025}, where the domains considered are more general in nature.  A further interesting work regarding the Neumann boundary value problem for large solutions for nonlocal equations can be found in \cite{Ros-OtonWeidner2026}, where the normal derivative of the third condition present in the above equation is considered. Indeed, on bounded domains, the operator $E_{\Omega}$ behaves like the limit of $\delta_{\Omega}^{1-s} u$ at the boundary of $\Omega$, see Lemma \ref{lemma equivalent E} below. 

\medskip

Motivated by these developments in bounded domain, we investigate \emph{boundary blow-up} solutions on \emph{cylindrical domains} that are \emph{unbounded} in certain directions.  The aim of studying PDEs on a sequence of finite cylindrical domains is to establish that the resulting sequence of solutions approximates a solution of a similar equation on an infinite cylinder. These studies have been carried out for various local and nonlocal operators in different contexts such as convergence of solutions, eigenvalues, and energies of given equations. In this direction, we refer to \cite{Chipotbook, ChipotRougirel2008, ChipotYeressian2008, ChowdhuryRoy2017} and the references therein. Work related to large solutions in this setting is very limited.
For instance, \cite{bandle2013large} studies large solutions for a semi-linear equation involving the Laplacian, and \cite{IndroDatta2025} explores large solutions of a second-order quasilinear equation  including the 
p-Laplacian. In both works, the solutions on finite cylinders are shown to converge to a solution on the infinite cylinder. Furthermore, the limiting solution does not depend on the variable along which the cylinder is unbounded. It coincides with the solution of the problem defined on the cross-section of the infinite cylinder.

\medskip

In this article, we consider ``weak dual'' notion of solution defined in Definition \ref{Dfn weak dual solution with f,g,h} and Definition \ref{dfn weak dual for semilinear} for linear and semi-linear equations, respectively. This notion of solution was employed in \cite{AbatangeloVazquez2023} for studying a linear equation with data $(f,0,h)$, and in \cite{BonforteFigalliVazquez2018} for investigating a semi-linear equation with data $(u^{p},0,0)$, $0<p<1$.

In both the linear and semi-linear equations, we establish that a large solution can be constructed on the infinite cylinder $\Omega_{\infty}$ as the pointwise limit of a sequence of solutions defined on a class of finite cylinders, where these solutions blow up only on the lateral part of the boundary of these finite cylinders. For the linear case, we investigate large solutions generated by different sets of data. These include large solutions generated by the operator with data $(0,0,h)$ where $h > 0$; large solutions generated by the forcing term given by the data $(\delta^{-\beta}_{\Omega_{\infty}}, 0, 0)$ for $2s < \beta < 1+s$; and large solutions generated by the complement data $(0, g, 0)$. In the second case, unlike in the bounded domain setting, the forcing term does not belong to \(L^1(\Omega_\infty, \delta^s)\) (cf. \ref{F1}). Nevertheless, we show that the weak dual formulation remains well defined through dimension reduction.  
For a semi-linear equation, we first notice that the nonlinearity alone cannot produce large solutions whenever $h\equiv 0$, i.e., there is no large solution to the problem \eqref{semi-linear prob on infinite cylinder} as well as to the problem \eqref{eqn main semilinear omega}  with data $(f(x,t), 0, 0)$. We further provide the existence of large solutions to \eqref{eqn main semilinear omega}, for data $(f(x,t),g,h)$ when $h>0$, regardless of the choice of $g$, on a bounded domain $\Omega$. This result is then used in the construction of a large solution on the infinite cylinder for a similar data.

\medskip

 Let us define the distance from the boundary as $\delta_{\Omega} (x) := \text{dist}(x,\partial\Omega)$. For a generic domain $\Omega \subset \R^n$ we take the following assumptions on the complement data and the boundary data:

 \begin{enumerate}
    \myitem{$\textbf{(G)}$} \label{G}\textit{ $g:\R^{n}\setminus\overline{\Omega} \to \R$ is a measurable function such that $g\in L^{1}(\R^{n}\setminus\overline{\Omega},\delta_{\Omega}^{-s})$.}
\end{enumerate}
\begin{enumerate}
    \myitem{$\textbf{(H)}$}\label{H} \textit{ $h\in L^{1}(\partial\Omega)$ is a function  such that $h\geq 0$.} 
\end{enumerate}

\subsection{Linear equation}
We consider the following linear equation on a domain $\Omega$ whose complement has non-zero Lebesgue measure,
\begin{equation}\label{eqn main linear omega}
\begin{cases}
\slap u=f(x) \quad &\text{in }\Omega\\
u(x)=g(x) &\text{in }\R^{n}\setminus\overline{\Omega}\\
E_{\Omega}u(\theta)=h(\theta) &\text{on }\partial\Omega.
\end{cases}
\end{equation}
For a generic domain $\Omega \subset \R^n$ we assume that the force function satisfies
\begin{enumerate}
    \myitem{$\textbf{(F1)}$}\label{F1} 
    $f:\Omega\to \R$ is in $L^{1}(\Omega,\delta^{s})\cap L^{q}_{loc}(\Omega)$ for $q>n/2s$.
\end{enumerate}
\medskip

Firstly,  we state the following representation results of solutions:
\begin{theorem}
   Let $\Omega$ be a domain with $C^{1,1}$ boundary and $|\rn\setminus \Omega|>0$. Assume that $f\in L^{1}(\Omega,\delta^{s}_{\Omega})\cap L^{q}_{loc}(\Omega)$ for $q>n/2s$, $g$ satisfies \ref{G}, and $h\in L^{1}(\partial\Omega)$. The function defined by
\begin{equation*}\label{representation}
 u(x)=\begin{cases}
        \int_{\Omega}\gk{\Omega}(x,y)f(y)\ dy-\int_{\R^{n}\setminus\overline{\Omega}}\slap \gk{\Omega}(x,y)g(y)\ dy+\int_{\partial\Omega}\Ds \gk{\Omega}(x,\theta)h(\theta)\ d\mathcal{H}(\theta) \quad &\text{for } x\in \Omega\\
        g(x) &\text{for }x\in \overline{\Omega}^{c}.
    \end{cases}
\end{equation*}
is the unique solution of the problem \eqref{eqn main linear omega}.
\end{theorem}
 The result when $g=0$ is proved in \cite{AbatangeloVazquez2023}. The proof of the above theorem is presented in Section \ref{subsection_3.1} (see Theorem \ref{thrm linear existence with g} and Theorem \ref{thrm rep of u on any domain}). The result for s-harmonic functions on a general domain that has an associated Green's kernel, is obtained in \cite{BogdanKulczyckiMateusz2008}. We consider $C^{1,1}$ domains in order to use some estimates on the Green's kernel given in \cite{ChenSong1998}.

 Let us assume \ref{F1},\ref{G} and \ref{H} on $\Omega=\Omega_{\infty}$,  and consider 
\begin{equation}\label{linear eqn on infinite cylinder}
\begin{cases}
     \slap u=f(x) \quad &\text{in } \Omega_{\infty} \\
     u=g &\text{in }\R^{n}\setminus\overline{\Omega}_{\infty}\\
    E_{\Omega_{\infty}}u=h &\text{on }\partial\Omega_{\infty}.
\end{cases}
\end{equation}
We further consider $\Omega_{\ell}$ to be the $C^{2}$ perturbation of the $B_{1}^{m}(0)\times [-\ell,\ell]^{n-m}$ defined in \eqref{dfn of omegaell} and define 
\begin{equation}\label{dfn of g ell and h ell}
  f_{\ell}(x)=f|_{\Omega_{\ell}} (x),\quad  g_{\ell}(y):=\begin{cases}
        g(y) \quad &\text{in } \R^{n}\setminus\overline{\Omega}_{\infty}\\
        0 &\text{in }\Omega_{\infty}\setminus\Omega_{\ell},
    \end{cases} 
    \quad\text{and}\quad 
    h_{\ell}(\theta):=\begin{cases}
        h(\theta)\quad &\text{on }\partial\Omega_{\ell}\cap\partial\Omega_{\infty}\\
        0 &\text{on }\partial\Omega_{\ell}\setminus(\partial\Omega_{\ell}\cap\partial\Omega_{\infty}).
    \end{cases}
\end{equation}
For each $\ell>0$, we consider the problem on the finite cylinders $\Omega_{\ell}$:
\begin{equation}\label{solution on omega ell}
    \begin{cases}
         \slap \ul =f_{\ell}(x) \quad &\text{in }\Omega_{\ell}\\
        \ul=g_{\ell} &\text{in }\R^{n}\setminus\overline{\Omega}_{\ell}\\
        E_{\Omega_{\ell}}\ul(x)= h_{\ell}(\theta) &\text{in }  \partial\Omega_{\ell}.
    \end{cases}
\end{equation}
Then, we prove that $\{\ul\}_{\ell}$ is locally uniformly bounded and increasing sequence. Further, the following result holds
\begin{theorem}\label{thm:linear-infite}
    Assume \ref{F1}, \ref{G} and \ref{H} for $\Omega=\Omega_{\infty}$. Additionally, we assume that $f,g\geq 0$, then there exists a function $u\in L_{loc}^{1}(\Omega_{\infty})$ such that $u_{\ell}\to u$ point wise and $u$ solves \eqref{linear eqn on infinite cylinder}.
\end{theorem}
The proof of the above result is given in Section \ref{subsection_3.2}. As it is a linear equation, the proof is given separately for  $(f,0,0)$, $(0,g,0)$ and $(0,0,h)$.

 Further, motivated by the results in \cite{Abatangelo2015LargesHarmonic} for bounded domain, we study the large solutions for infinite cylinder involving $(\delta^{-\beta}_{\Omega_{\infty}}(x),0,0)$ as data. We cannot directly use Theorem \ref{thm:linear-infite} as the forcing term does not satisfy \ref{F1}. In this direction, we prove two results regarding dimension reduction from $\Omega_{\infty}$ to $B_{1}^{m}(0)\subset \R^{m}$; one is for the Green's kernel and the other is for the weak dual solution, see Theorem \ref{thrm-dim-reduction-for-Greens} and Theorem \ref{them-dim-reduction-for-sol}.


The following three results establish different types of large solutions on an infinite cylinder:

\begin{theorem}[Driven by the force function] 
Let $m>2s$, $2s<\beta<1+s$ and $f(x)=1/\delta^{\beta}_{\Omega_{\infty}}(x)$ for all $x\in \Omega_{\infty}$. If $u_{\ell}^{f}$ denotes the solution of \eqref{solution on omega ell} with the data $(1/\delta_{\Omega_\infty}^{\beta}|_{\Omega_{\ell}},0,0)$, then there exists a function $u^{f}$ such that $\ul^{f}\to u^{f}$ point-wise. Moreover, $u^{f}$ is a solution of \eqref{linear eqn on infinite cylinder} with the data $(f,0,0)$. There exist two positive constants $c, C$ such that 
      \begin{equation*}
          c\delta_{\Omega_{\infty}}^{-\beta+2s}(x)\leq u^{f}(x)\leq C \delta_{\Omega_{\infty}}^{-\beta+2s}.
      \end{equation*}
\end{theorem}

\begin{theorem}[Driven by the operator]
Assume \ref{H} on $\Omega=\Omega_{\infty}$ with $h\in C(\partial\Omega_{\infty})$ and $h>0$.  If $u_{\ell}^{h}$ denotes the solution of \eqref{solution on omega ell} with the data $(0,0,h_{\ell})$, then $u_{\ell}^{h}(x)\to \infty$ as $x\to \theta \in \partial\Omega_{\infty}\cap \partial\Omega_{\ell}$. Further, there exists a function $u^{h}$ such that $u_{\ell}^{h}(x)\to u^{h}(x)$ point-wise, which is a solution of \eqref{linear eqn on infinite cylinder} with data $(0,0,h)$. Moreover, $u^{h}(x)\to \infty$ as $x\to \theta\in \partial\Omega_{\infty}$.
\end{theorem}
\begin{theorem}[Driven by the complement data] 
Assume \ref{G} on $\Omega=\Omega_{\infty}$ and 
      \begin{equation*}
          \lim_{\substack{x \to x_{0} \\ x\in\overline{\Omega}^{c}_{\infty}}} g(x)=+\infty, \quad \text{for all } x_{0} \in \partial\Omega_{\infty}.
      \end{equation*}
      If $u^{g}_{\ell}$ is solution of \eqref{solution on omega ell} with the data $(0,g_\ell,0)$, then $u_{\ell}^{g}(x)\to \infty$ as $x\to x_{0}\in \partial\Omega_{\infty}\cap \partial\Omega_{\ell}$. Further,  there exists a function $u^{g}$ such that $u_{\ell}^{g}\to u^{g}$ point-wise, which is solution to \eqref{linear eqn on infinite cylinder} with the data $(0,g,0)$.
\end{theorem}

The proof of above three results are given in section \ref{subsection_3.2}.
\subsection{Semi-linear equation}
Consider \eqref{eqn main semilinear omega}, where
\begin{enumerate}
    \myitem{$\textbf{(F2)}$}\label{F2} \textit{$f:\Omega\times \R\to\R$ be such that \begin{itemize}
        \item For any $x\in \Omega$, $f(x,0)=0$, $f(x,t)\geq 0$ for all $t>0$, $f\in C(\Omega\times \R)$ and $f(x,t)$ is increasing in the $t$ variable.
        \item There exist $a_{1},a_{2}\geq 0$ and $p\in [0,\frac{1+s}{1-s})$ such that 
        \begin{equation*}
            f(x,t)\leq a_{1}+a_{2}t^{p}\quad \text{for all } t>0 \text{ and }x\in \Omega, 
        \end{equation*}
    \end{itemize}}
\end{enumerate}
We first prove a comparison-type results on bounded domains in Lemma \ref{lemma restriction of a solution to subdomains} and Lemma \ref{lemma comparison principle for semilinear eqn}. Then prove the existence of a weak dual solution to the semi-linear equation \eqref{eqn main semilinear omega} on a bounded domain with $C^{1,1}$ boundary in Theorem \ref{theorem existence} adopting the sub and super solution method proved in Theorem \ref{thrm sub and super solution method}. 
In \cite{BogdanJarohsKania2020}, a semi-linear equation with $f(x,t)=mF(x,t)$ is considered in general domains, and the existence is obtained for small parameters $m$.

Unlike the linear case, the semi-linear equation does not admit any large solution to the data $(f(x,\cdot),0,0)$. The proof when $f(x,\cdot)$ is monotonic is due to the uniqueness since $f(x,0)=0$. However a similar result is true if we drop the assumption of monotonicity. We establish this in Theorem \ref{thrm no large solution}.

In the semi-linear case, the concern is to compare solutions on two nested cylinders. To this end, we establish a result regarding the restriction of a solution to a subdomain (see Lemma \ref{lemma restriction of a solution to subdomains}), analogous to \cite[Corollary 1 ($L^{1}$ weak solution)]{Abatangelo2015LargesHarmonic} in the weak dual case. In doing so, we generalize \cite[Corollary 1]{Abatangelo2015LargesHarmonic} to subdomains that are not required to be pre-compact. Moreover, the result also holds for linear equations.

Moving on to the infinite cylinder setting, consider the equation \eqref{semi-linear prob on infinite cylinder}, and for each $\ell>0$,
 \begin{equation}\label{semi-linear prob on finite cylinder}
    \begin{cases}
         \slap \ul=-f(x,\ul)\quad &\text{in } \Omega_{\ell}\\
         \ul=g_{\ell} &\text{in }\R^{n}\setminus\overline{\Omega}_{\ell}\\
         E_{\Omega_{\ell}}(\ul)=h_{\ell} &\text{on }\partial\Omega_{\ell},
    \end{cases}
 \end{equation}
 where, $g_{\ell}$ and $h_{\ell}$ are defined in \eqref{dfn of g ell and h ell}.
The sequence $\{u_{\ell}\}_{\ell}$ of solutions is locally uniformly bounded (Lemma \ref{local uniform bound for semi-linear eqn on bounded domain}) and is increasing point-wise, which follows by the use of comparison principle.
 Then we prove the following result
 \begin{theorem}\label{semilinear theorem}
     Assume \ref{F2}, \ref{G} and \ref{H} on $\Omega=\Omega_{\infty}$. Assume additionally that $h\in C(\partial\Omega_{\infty})\cap L^{\infty}(\partial\Omega_{\infty})$ and $g\geq 0$. Then, there exists a function $u\in L_{loc}^{1}(\Omega)$ and $u$ is the unique solution of \eqref{semi-linear prob on infinite cylinder}.
 \end{theorem}
The proof of the above theorem is presented in Section \ref{subsection_4.2} and Section \ref{subsection_4.3}.
 \begin{remark}
      By our definition, $\slap$ is a positive operator; therefore, the negative sign on the right-hand side of the semi-linear equation is justified. This convention is consistent with the local problems.
 \end{remark}
 \subsection{Outline}
In Section \ref{section prilims}, we introduce fundamental definitions and key results concerning Fractional Green's kernel, Green's potential, the $s$-normal derivative, and the boundary trace operator $E_{\Omega}$. Additionally, we construct upper barriers for the Green's kernel and Poisson kernel on unbounded domains. Subsequently, we provide basic definitions concerning linear and semi-linear equations and state some results for the same on bounded domains. Further, we establish some results regarding Green's kernel and Green's potential on a sequence of finite cylinders.

In Section \ref{section linear}, We establish the existence and uniqueness theorem for linear equations on domains that are not necessarily bounded. Followed by a discussion on comparison principle. Further, we discuss the dimension reduction between $\Omega_{\infty}$ and its cross-section $B_{1}^{m}(0)$, and establish that the integral on the unbounded component $\R^{n-m}$ in $\Omega_{\infty}$, of the Green's kernel associated with the infinite cylinder coincides with the Green's kernel associated with the cross-section. At the end if this section, we prove the results regarding linear equations on the infinite cylinder.

In Section  \ref{section semilinear}, we first establish results for solutions of the semi-linear equation on bounded domains, including a comparison principle. In this context, we state the existence theorem on bounded domain whose proof is presented in the Appendix. Subsequently, we derive the existence and uniqueness result for the semi-linear equation on the infinite cylinder.

In Section \ref{section appendix}, we present the proof of sub-super solution method for establishing the existence of a bounded solution to a semi-linear equation on bounded domains as well as the proof of existence of a large solution to semi-linear equations on bounded domains.

\section{Preliminaries}\label{section prilims}
This section introduces the basic concepts underlying our analysis.

\subsection{Notation}
 The following notations are used throughout the article.
\begin{enumerate}
    \item Let $\Omega\subset \rn$. We denote by $\delta_{\Omega}(x)=$dist$(x,\partial\Omega)$, the distance of the point $x$ from the boundary of $\Omega$.
    \item $\gk{\Omega}(x,y)$ denotes the fractional Green's kernel on $\Omega$.
    \item $\g{\Omega}(\cdot)$ denotes the Green's potential of $\slap$.
    \item Arbitrary positive constants are denoted by $c$ and $C$.
    \item For any two functions $f$ and $h$, $f\asymp h$  means there exist two positive constants $c,C$ such that $ch\leq f\leq Ch$.
\item $L^{1}(\Omega,\delta^{s}_{\Omega
})$ denotes the weighted $L^{1}$ space with weight $\delta_{\Omega
}^{s}$.
\item We write $f\in \delta_{\Omega}^{s}C(\overline{\Omega})$, to indicate that $f$ can be represented by $f=\delta_{\Omega}^{s}g$ where $g\in C(\overline{\Omega})$.
\item The notation $\delta_{x}$ denotes the Dirac delta at $x$.
\item The notation $B^{n}$ and $S^{n-1}$ denote the unit ball and the unit sphere in $\rn$ respectively.

\end{enumerate}
\subsection{Fractional Green's Kernel and Potential}  
We refer to \cite{Lankoff1972} and the references therein for a comprehensive discussion of the fractional Green's kernel. For completeness, we introduce the following definitions and discuss several results relevant to our setting.
\begin{definition}
   The Green's kernel associated with a domain $\Omega$ is a collection of functions $\gk{\Omega}:\Omega\times\rn \to\R$ given by 
    \begin{equation*}
        \gk{\Omega}(x,y)=\Gamma_{s}(x-y)-H_{\Omega}(x,y),
    \end{equation*}
    where $\Gamma_{s}$ is the fundamental solution associated with $\slap$, which is given by 
    \begin{equation*}
        \Gamma_{s}(x):= \frac{ a(n,s)  }{|x|^{n-2s}}\quad \text{ where,}\quad a(n,s):=\frac{\Gamma(n/2-s)}{\pi^{n/2}\Gamma(2+s)}s(1+s),
    \end{equation*}
    and $H_{\Omega}(x,\cdot)$ is the s-harmonic function solving (pointwise)
    \begin{equation*}\label{dfn of H}
        \begin{cases}
            \slap H_{\Omega}(x,\cdot)=0 \quad &\text{on }\Omega\\
            H_{\Omega}(x,y)=\Gamma(x-y) &\text{in } \R^{n}\setminus\Omega.
        \end{cases}
    \end{equation*}
\end{definition}

The following result can also be found in \cite{Abatangelo2015LargesHarmonic}; the proof is provided for the completeness. 
\begin{proposition}
    Let $\Omega$ be a bounded domain and $x\in \Omega$, then $H_{\Omega}(x,\cdot)\in C^{2s+\epsilon}(\Omega)\cap C(\overline{\Omega})$.
\end{proposition}
\begin{proof}
    Let $r=r(x)>0$ such that $\overline{B_{r}(x)}\subset \Omega$ and $\kappa(=\kappa_{x})\in C^{\infty}(\rn)$ with $0\leq \kappa\leq 1$, $\kappa=1$ in $\R^{n}\setminus\overline{\Omega}$ and $\kappa=0$ in $\overline{B_{r}(x)}$. Define 
    \begin{equation}\label{dfn small gamma s}
       \gamma_{s}(x,y)=\kappa_y\Gamma_{s}(x-y).
    \end{equation}
    Note that $\gamma_{s}(x,\cdot)\in C^{\infty}(\rn) \cap L^{\infty}(\rn)$, $\gamma_{s}(x,y)=\Gamma_{s}(x-y)$ in $\R^{n}\setminus\overline{\Omega}$ and $\gamma_{s}(x,y)=0$ for $y\in B_{r}(x)$. Further by \cite[Proposition 2.7]{Silvestre2007}, $\slap \gamma_{s}\in L^{\infty}(\rn)\cap C^{\infty}(\rn)$. Set 
    \begin{equation}\label{dfn of small h}
        h(x,y)=H_{\Omega}(x,y)-\gamma_{s}(x,y),
    \end{equation} which solves 
    \begin{equation*}
    \begin{cases}
      \slap h(x,\cdot)= -\slap \gamma_{s}(x,\cdot) \quad &\text{in } \Omega\\
      h(x,y)=0 &\text{in }\Omega^{c}.
    \end{cases}
    \end{equation*}
    Then by \cite[Proposition 1.1]{ROSOTON2014275}, $h(x,\cdot)\in C^{s}(\rn)$ and by \cite[Proposition 2.8]{Silvestre2007} $h(x,\cdot)\in C^{2s+\epsilon}(\Omega)$. Hence, the result follows from \eqref{dfn of small h}.
\end{proof}
Further, we derive a local uniform bound for $H_{\Omega}(\cdot,\cdot)$. 
   \begin{lemma}\label{lemma uniform bound for H}
        Let $\Omega$ be a bounded domain, and $K$ be a compact subset of  $\Omega$. For any $x\in K$ and $y\in \Omega$, we have
        \begin{equation*}
            |H_{\Omega}(x,y)|\leq C(n,s,K)\left(1+\delta_{\Omega}^{s}(y)\right).
        \end{equation*}
\end{lemma}
\begin{proof}
   Choose $r=r(K)>0$ such that $B_{r}(x)\subset \Omega$ for all $x\in K$, and consider $\gamma_{s}(x,\cdot)$ as in \eqref{dfn small gamma s}. By \eqref{dfn of small h}, $H_{\Omega}(x,y)=h(x,y)+\gamma_{s}(x,y)$. By \cite[Lemma 2.7]{ROSOTON2014275}, for a fixed $x\in \Omega$ 
    \begin{equation*}
        |h(x,y)|\leq C ||\slap \gamma_{s}(x,\cdot)||_{L^{\infty}(\Omega)}\delta^{s}(y).
    \end{equation*}
    To show that $H_{\Omega}$ is bounded locally uniformly with respect to the first variable, we compute the bounds with respect to $x$, for $\gamma_{s}$ and $\slap \gamma_{s}$. For $x\in K$ and $y\in \Omega\setminus B_{r}(x)$
    \begin{equation*}
        \gamma_{s}(x,y)\leq \Gamma_{s}(x-y)=\frac{a(n,s)}{|x-y|^{n-2s}}\leq \frac{a(n,s)}{r^{n-2s}}.
    \end{equation*}
 For $x\in K$ and $y\in B_{r}(x)$, we note that $\gamma_{s}(x,y)=0$. Then for $x\in K$ and $y\in \Omega\setminus B_{r}(x)$ we get
    \begin{equation*}
    \begin{split}
        |\slap \gamma_{s}(x,y)| 
          &\leq \mathcal{A}(n,s) \Big\{\int_{B_{1}(y)}\frac{\left|\gamma_s(x,y)-\gamma_{s}(x,z) - (y-z)\cdot D\gamma_s(x,y)\right|}{|y-z|^{n+2s-2}}\ dz \\
          & \hspace{3cm}+\int_{\rn\setminus B_{1}(y)}\frac{\left|\gamma_s(x,y)-\gamma_{s}(x,z)\right|}{|y-z|^{n+2s}}\ dz\Big\}\\
        &\leq \mathcal{A}(n,s) \left\{\int_{B_{1}(y)}\frac{\|D^{2} \gamma_{s}(x,\cdot)\|_{L^\infty(B_1(y))}}{|y-z|^{n+2s-2}}\ dz+ \frac{2a(n,s)}{r^{n-2s}}\int_{\rn\setminus B_{1}(y)}\frac{1}{|y-z|^{n+2s}}\ dz\right\}\\
        &\leq \mathcal{A}(n,s)\omega_{n-1}\left\{\sup_{K\times \overline{B_{1}(y)}}\left|D^{2} \gamma_{s}\right| +\frac{2a(n,s)}{r^{n-2s}}\right\}\leq C(n,s,K). 
    \end{split}
    \end{equation*}
    Further, for $x\in K$ and $y\in B_{r}(x)$ we have $ \displaystyle |\slap \gamma_{s}(x,y)|\leq \mathcal{A}(n,s)\int_{\rn\setminus B_{r}(x)} \frac{a(n,s)}{r^{n-2s}}\frac{1}{|y-z|^{n+2s}}\ dz.$
    Thus 
    \begin{equation*}\label{eqn local uniform bound for H}
        |H_{\Omega}(x,y)|\leq C(n,s,K) \delta^{s}(y)+\frac{a(n,s)}{r^{n-2s}}\leq C(n,s,K,)(1+\delta^{s}(y)).\qedhere
    \end{equation*} 
\end{proof}
Next, we include some known properties of the fractional Green's kernel. 
\begin{proposition}\label{prop greens kernal continuity}
   Let $\Omega\subset \R^{n}$ be a domain such that $|\R^{n}\setminus\Omega|>0$, the Green's kernel $\gk{\Omega}(x,y)$ has the following properties
   \begin{itemize}
       \item[(i)] $\gk{\Omega}(\cdot,\cdot)\in C(\Omega\times\rn\setminus\{(x,x)~|~x\in \Omega\})$ and $\gk{\Omega}(x,y)=\gk{\Omega}(y,x)$ for all $x,y\in\Omega$, with $x\neq y$. 
       \smallskip 
       
       \item[(ii)] $\gk{\Omega}(x,y)\geq 0$ a.e. on $\Omega\times\rn$. 
       \smallskip 
       
       \item[(iii)] For any $x\in\Omega$ and $y\in \R^{n}\setminus\overline{\Omega}$, $\slap\gk{\Omega}(x,y)$ is well defined and is given by 
       \begin{equation*}
           \slap\gk{\Omega}(x,y)=-\mathcal{A}(n,s)\int_{\Omega}\frac{\gk{\Omega}(x,z)}{|y-z|^{n+2s}}\ dz.
       \end{equation*} 
       Thus, $\slap\gk{\Omega}(x,y)\leq 0$ for $x\in \Omega, y\in \R^{n}\setminus\overline{\Omega}$. In addition for $x\in \Omega$, $\slap \gk{\Omega}(x,\cdot)\in L^{1}(\R^{n}\setminus\overline{\Omega})$. 
       \smallskip
       
       \item[(iv)] The following formulas hold 
     \begin{equation*}
           \Gamma_{s}(x-y)=-\int_{\R^{n}\setminus\overline{\Omega}}\Gamma_{s}(y-z)\slap\gk{\Omega}(x,z)\ dz, \quad ~ \text{for }x\in \Omega\text{ and } y\in \R^{n}\setminus\overline{\Omega},
       \end{equation*} 
       and 
       \begin{equation*}
           H_{\Omega}(x,y)=-\int_{\R^{n}\setminus\overline{\Omega}}\Gamma_{s}(y-z)\slap\gk{\Omega}(x,z)\ dz, \quad ~\text{for }x\in \Omega\text{ and }y\in \Omega.
       \end{equation*}
       \smallskip 
       
       \item[(v)] If $\Omega$ is bounded and has $C^{1,1}$ boundary, then there exist  constants $C,c>0$ depending only on $\Omega$, $n$, and $s$ such that for $x,y\in \Omega $ the following holds:
       \begin{equation}\label{eqn behaviour of greens kernel}
           \frac{1}{C|x-y|^{n-2s}}\min\left\{\frac{\delta^{s}_{\Omega}(x)\delta^{s}_{\Omega}(y)}{|x-y|^{2s}},1\right\}\leq \gk{\Omega}(x,y)\leq \frac{C}{|x-y|^{n-2s}}\min\left\{\frac{\delta^{s}_{\Omega}(x)\delta^{s}_{\Omega}(y)}{|x-y|^{2s}},1\right\}.
       \end{equation}
For $x\in \Omega$ and $y\in \R^{n}\setminus\overline{\Omega}$ the following holds:
\begin{equation}\label{asym of poison on bounded domain}
    \frac{\delta_{\Omega}^{s}(x)}{c\delta_{\Omega}^{s}(y)\left(1+\delta_{\Omega}(y)\right)^{s}|x-y|^{n}}\leq -\slap\gk{\Omega}(x,y)\leq \frac{c\delta_{\Omega}^{s}(x)}{\delta_{\Omega}^{s}(y)\left(1+\delta_{\Omega}(y)\right)^{s}|x-y|^{n}}.
\end{equation}
\smallskip 

\item[(vi)] When $\Omega\subset \R^{n}$ is a domain with uniform exterior ball condition. Then for any $x\in \Omega$, there exists $a_x\in \R^{n}\setminus\Omega$ and a constant $C(n,s)$ such that for any $y\in \Omega$
\begin{equation}\label{eqn greens estimatte on unbounded domain}
    \gk{\Omega}(x,y)\leq C(n,s)|y-a_x|^{s}\frac{\delta_{\Omega}^{s}(x)}{|x-y|^{n-s}}.
\end{equation}
   \end{itemize}
\end{proposition}

The proofs of these results, except for the last two, can be found in  \cite[Lemma 3.2]{Abatangelo2015LargesHarmonic} (see also \cite{Bucur2016, Lankoff1972}), and the last two in \cite[Corollary 1.3, Theorem1.5, Proof of (1.4)]{ChenSong1998}. For completeness, we outline the proof of \eqref{eqn greens estimatte on unbounded domain}, as this inequality serves as an essential tool in several arguments used throughout this article.

\begin{proof}[Proof of Proposition \ref{prop greens kernal continuity} (vi)]
   Before proceeding with the proof, we consider an estimate of Green's kernel on the complement of a ball, given in \cite{ChenSong1998}. Let $B_{r}(a)\subset \R^{n}$, then by \cite[Lemma 2.5]{ChenSong1998} there exists a positive constant $c=c(n,s)$ such that the Green's kernel on the complement of the ball has the following bound:
   \begin{equation*}
       \gk{B_{r}^{C}(a)}(x,y)\leq c|y-a|^{s}\frac{\delta_{B_{r}(a)}^{s}(x)}{|x-y|^{n-s}}, \quad \text{ for all }x,y\in B^C_{r}(a).
   \end{equation*}
   For $x\in \Omega$, consider $\theta_x\in\partial\Omega$ such that $|x-\theta_x|=\delta_{\Omega}(x)$. Let $B:=B_{r}(a_x)$, with center at $a_x\in \R^{n}\setminus\overline\Omega$, be an exterior ball associated with $\theta_x$, i.e. $B\subset \R^{n}\setminus\Omega$ and $B\cap \partial\Omega=\{\theta_x\}$. As $\Omega\subset \R^{n}\setminus B$, $\gk{\Omega}(z,y)\leq \gk{B^{C}}(z,y)$ for all $z,y\in \Omega$; see for example \cite[Proposition 2.3]{BanuelosKulczyzki2008}. In particular 
   \begin{equation}\label{eqn behaviour of greens on unbounded domains}
           \gk{\Omega}(x,y)\leq \gk{B^{C}}(x,y) \leq c|y-a_x|^{s}\frac{\delta_{B}^{s}(x)}{|x-y|^{n-s}}=c|y-a_x|^{s}\frac{\delta_{\Omega}^{s}(x)}{|x-y|^{n-s}}. \qedhere
   \end{equation}
   \end{proof}

\begin{remark}
    When $\Omega$ is unbounded such that $|\R^{n}\setminus \Omega|>0$, then for $y \in B_1^C(x)\cap \Omega$  we get a constant $c = c(n,s)$ such that
    \begin{equation*}
        \gk{\Omega}(x,y)\leq c\frac{\delta_{\Omega}^{s}(x)}{|x-y|^{n-2s}}.
    \end{equation*}
\end{remark}

To define the operator $E_{\Omega}$, we first introduce the fractional normal derivative associated with the domain $\Omega$. The following definition can also be found in \cite{AbatangeloVazquez2023}.
\begin{definition}[s-normal derivative]\label{frac normal derivative}
    Let $\Omega$ be a $C^{1,1}$ domain in $\rn$. The \textit{s-normal derivative}  of a function $v:\Omega\to\R$, associated  with $\partial \Omega$ is 
    \begin{equation*}
        D_{s}v(\theta):=\lim_{y\in \Omega\to \theta}\frac{v(y)}{\delta^{s}_{\Omega}(y)},\quad \theta\in \partial\Omega.
    \end{equation*}
\end{definition}
Take $x\in \Omega$ and $\theta \in \partial \Omega$. Note that $ D_s \gk{\Omega}(x,\theta)$ is well-defined by \eqref{eqn behaviour of greens kernel} when $\Omega$ is bounded (see Lemma \ref{lemma equivalent E}) and by \cite[Theorem 1.2]{RosOtonSerra2017} otherwise. Indeed, if the domain is also has uniform exterior ball condition, then by using the symmetry of the Green's kernel and \eqref{eqn greens estimatte on unbounded domain} one also has  
\begin{equation}\label{eqn fractional normal derivative behaviour}
     \frac{\gk{\Omega}(x,y)}{\delta_{\Omega}^{s}(y)}\leq c(n,s)\frac{|x-a_y|^{s}}{|x-y|^{n-s}}.
\end{equation}
Let $\rho>0$, for any $y\in B_{\rho/2}(\theta)$, take an exterior ball $B_{r_y} (a_y)$, associated with corresponding $\theta_y \in \partial\Omega$ ( such that $|y-\theta_y|=\delta_{\Omega}(y)$), which is contained in $B_{\rho}(\theta)$. Choose a point $a_{\theta}\in  \partial B_{2\rho}(\theta)\cap (\R^{n}\setminus\Omega)$, then $|x-a_y|\leq |x-a_{\theta}|$, which implies 
\begin{equation}\label{eqn well definedness of frac normal derivative on unbounded domain}
    0\leq \Ds \gk{\Omega}(x,\theta)\leq c(n,s)\frac{|x-a_{\theta}|^{s}}{|x-\theta|^{n-s}}< +\infty.
\end{equation}
Moreover, using $\frac{|x-a_{\theta}|}{|x-\theta|}\leq 1+\frac{|\theta-a_{\theta}|}{|x-\theta|}$, $|x-\theta|\geq \delta_{\Omega}(x)$ and taking $\rho = \delta_{\Omega}(x)$ we obtain 

\begin{equation}\label{eqn martin behaviour}
    \Ds \gk{\Omega}(x,\theta)\leq\frac{C(n,s)}{\delta_{\Omega}^{n-2s}(x)}.
\end{equation}
Now, we define the operator $E_{\Omega}$ associated with $\partial\Omega$  as 
\begin{equation}\label{dfn E}
    E_{\Omega}v(\theta):=\lim _{\substack{x\in \Omega \\ x \to \theta}}\frac{v(x)}{\int_{\partial\Omega}\Ds\gk{\Omega}(x,\sigma)\ d\mathcal{H}(\sigma)},
\end{equation}
where $\mathcal{H}$ is the $n-1$ dimensional measure on the boundary.
For $h\in L^{1}(\partial\Omega)$ the limit is understood to be in the weak sense. If we denote 
\begin{equation*}
    v(x)=\int_{\partial\Omega}\Ds\gk{\Omega}(\theta,x)h(\theta)\ d\mathcal{H}(\theta).
\end{equation*} 
Then, for a bounded domain $\Omega$, we have  $E_{\Omega}v(\theta)=h(\theta)$ provided $h\in C(\partial \Omega)$. This can be found in \cite[Lemma 3.4]{Abatangelo2015LargesHarmonic}. The corresponding weak limit for $L^{1}$ data can be found in \cite{AbatangeloVazquez2023}.

\begin{remark}
    In general, if $h\in L^{1}(\partial\Omega)$ then the limit is replaced by weak limit \cite[Theorem 4.15]{AbatangeloVazquez2023}. 
\end{remark}
The operator $E_{\Omega}$ defined above can be seen as a singular trace on the boundary, as shown in the following lemma.
\begin{lemma}\label{lemma equivalent E}
    Let $\Omega$ be a bounded domain in $\R^{n}$, then $E_{\Omega}u(\theta)\asymp \displaystyle\lim_{\substack{x\in \Omega\\ x\to \theta}}\delta_{\Omega}^{1-s}(x)u(x)$.
\end{lemma}
\begin{proof}
   For $y\in \Omega$ close enough to $\partial\Omega$, we have by \eqref{eqn behaviour of greens kernel} that
\begin{equation*}
    \frac{\delta_{\Omega}(x)^{s}\delta_{\Omega}(y)^{s}}{c|x-y|^{n}}\leq \gk{\Omega}(x,y)\leq \frac{c\delta_{\Omega}(x)^{s}\delta_{\Omega}(y)^{s}}{|x-y|^{n}}, \quad \text{ hence} \quad  \frac{\delta_{\Omega}(x)^{s}}{c|x-y|^{n}}\leq \frac{\gk{\Omega}(x,y)}{\delta_{\Omega}^{s}(y)}\leq \frac{c\delta_{\Omega}(x)^{s}}{|x-y|^{n}}.
\end{equation*}
This implies $\displaystyle \frac{\delta_{\Omega}(x)^{s}}{c|x-\theta|^{n}}\leq \Ds \gk{\Omega}(x,\theta)\leq \frac{c\delta_{\Omega}(x)^{s}}{|x-\theta|^{n}}$. 
Therefore 
\begin{equation*}\label{eqn bdy asym of trace}
    \frac{\delta_{\Omega}(x)^{s}}{c}\int_{\partial\Omega}\frac{d\mathcal{H}(\theta)}{|x-\theta|^{n}}\leq \int_{\partial\Omega}\Ds \gk{\Omega}(x,\theta)\ d\mathcal{H}(\theta)\leq c\delta_{\Omega}(x)^{s}\int_{\partial\Omega}\frac{d\mathcal{H}(\theta)}{|x-\theta|^{n}}.
\end{equation*}
Further,  since $\displaystyle \int_{\partial\Omega}\frac{d\mathcal{H}(\theta)}{|x-\theta|^{n}} \asymp \delta^{-1}(x),$ 
we get 
\begin{equation*}
    \int_{\partial\Omega}\Ds \gk{\Omega}(x,\theta)\ d\mathcal{H}(\theta)\asymp \delta_{\Omega}^{s-1}(x).
\end{equation*}
The result follows by plugging this in the definition of $E_{\Omega}$.
\end{proof}
Consider the linear equation \eqref{eqn main linear omega} on a domain $\Omega$ in $\R^{n}$, not necessarily bounded. The Green's potential associated with $\Omega$, informally, the inverse of the fractional Laplacian on $\Omega$, is defined as follows:
 \begin{definition}\label{defining the operator G}
    Let $\Omega\subset \rn$ be a domain and $f\in L^{1}(\Omega,\delta_{\Omega}^{s})$. The Green's potential $\g{\Omega}(f)$ of $f$ is a solution of \eqref{eqn main linear omega} for the data $(f,0,0)$, and is given by
    \begin{align*}
        \g{\Omega}(f)(x)=
        \begin{cases}
             \displaystyle \int_{\Omega}\gk{\Omega}(x,y)f(y) dy & \quad \text{for } x\in\Omega \\
             0 &\quad \text{for }x\in \rn \setminus \Omega.
        \end{cases}           
    \end{align*}
\end{definition}

Green's potential $\g{\Omega}$ is a continuous operator between the following spaces:
\begin{proposition}[\cite{AbatangeloVazquez2023}]\label{prop-greens-potintial-continuity}
    For a bounded domain $\Omega$ with $C^{1,1}$ boundary, the potential $\g{\Omega}$ is continuous as a map
    \begin{equation}\label{eqn continuity of G 1}
            L^{\infty}(\Omega)\to C(\overline{\Omega}),
            \end{equation}
             \begin{equation}\label{eqn continuity of G 2}
                L^{\infty}_{c}(\Omega)\to \delta_{\Omega}^{s}C(\overline{\Omega}),
            \end{equation}
     \begin{equation}\label{eqn continuity of G 3}
         L^{1}(\Omega,\delta_{\Omega}^{s})\to L_{loc}^{1}(\Omega). 
     \end{equation}
     Further, for any $f\in L_{c}^{1}(\Omega)$, $\g{\Omega}(f)$ is defined pointwise on $\Omega\setminus\text{Supp}(f)$, and that
     \begin{equation}\label{bound near the boundary for f compactly supported}
         |\g{\Omega}(f)(x)|\leq C\delta_{\Omega}^{s}(x)\, \text{dist}(x,\text{Supp}(f))^{s-n}\int_{\Omega}|f(y)|\delta_{\Omega}^{s}(y) \ dy, \quad\text{for } x\in \Omega\setminus\text{Supp}(f).
     \end{equation}
\end{proposition}
The proof can be found in \cite[Theorem 2.1, Theorem 2.10]{AbatangeloVazquez2023}. 
Further, we show how additional local integrability of $f$ improves the pointwise definition of the Green's potential.
\begin{theorem}
    Let $\Omega$ be bounded and $f\in L^{1}(\Omega,\delta_{\Omega}^{s})\cap L^{q}_{loc}(\Omega)$ with $q>n/2s$, then $\g{\Omega}(f)\in L^{\infty}_{loc}(\Omega)$.
\end{theorem}
\begin{proof}
    Let $K$ be a compact subset of $\Omega$. For any $x\in K$, choose $r>0$ such that $B_{r}(x)\subset \Omega$, then by \eqref{eqn behaviour of greens kernel} followed by the H\"older's inequality,
    \begin{equation*}
    \begin{split}
        |\g{\Omega}(f)(x)|&\leq C\int_{B_{r}(x)}\frac{|f(y)|}{|x-y|^{n-2s}}\ dy+C\delta^{s}_{\Omega}(x)\int_{\Omega\setminus B_{r}(x)} \frac{\delta_{\Omega}^{s}(y)|f(y)|}{|x-y|^{n-2s}}\ dy\\
        &\leq C \|f\|_{L^{q}(B_{r}(x))}\left\{\int_{B_{r}(x)}\frac{dy}{|x-y|^{q'(n-2s)}}\right\}^{1/q'}+\frac{C\delta^{s}_{\Omega}(x)}{r^{n-2s}}\|f\|_{L^{1}(\Omega, \delta^{s}_{\Omega})}.
    \end{split}
    \end{equation*}
    The result follows since $q>n/2s$.
\end{proof}
\begin{theorem}\label{theorem greens poteantial continuity on unbounded domains}
    Let $\Omega$ be a domain in $\R^{n}$ with uniform exterior ball condition. The map $\g{\Omega}$ is continuous between
    \begin{equation*}
    \begin{split}
        L^{\infty}_{c}(\Omega)\to \delta_{\Omega}^{s}L^{\infty}(\Omega),\quad \text{and} \quad L^{1}(\Omega,\delta_{\Omega}^{s})\to L_{loc}^{1}(\Omega). 
    \end{split}    
    \end{equation*}
Further, for $q>n/2s$, it is a continuous map between 
\begin{equation*}
    L^{1}(\Omega,\delta_{\Omega}^{s})\cap L_{loc}^{q}(\Omega)\to L^{\infty}_{loc}(\Omega).
\end{equation*}
Moreover, for any $f\in L_{c}^{\infty}(\Omega)$, $\g{\Omega}(f)\in C(\overline{\Omega})$.
\end{theorem}
\begin{proof}
\textit{First part:}
Let $f\in L^{\infty}_{c}(\Omega)$. Denote $K=$ Supp$(f)$. Choose $\rho>0$ such that $K\subset\subset\Omega\setminus\Omega_{\rho}=\{x\in \Omega~|~\delta_{\Omega}(x)\geq\rho\}$. Then for any $x\in \Omega\setminus \Omega_{\rho}$ 
using $\gk{\Omega}(x,y)\leq \Gamma(x-y)$ that
\begin{equation*}
        \g{\Omega}(f)(x)\leq \|f\|_{L^{\infty}(\Omega)}\int_{K}\gk{\Omega}(x,y)\ dy
        \leq  c(n,s)\|f\|_{L^{\infty}(\Omega)}\int_{K}\frac{dy}{|x-y|^{n-2s}}\leq C(K,\rho,n,s)\|f\|_{L^{\infty}(\Omega)} \delta_{\Omega}^{s}(x).
    \end{equation*}
    Indeed, if $B_{1}(x)\cap K\neq \emptyset$ then the integral over $K$ is bounded by integral on a ball of radius $1+$diam$(K)$. Otherwise the integral is bounded by $|K|$.

    For $x\in \Omega_{\rho}$, using \eqref{eqn behaviour of greens on unbounded domains}
     \begin{equation}\label{eqn greens potential for supported f}
        \g{\Omega}(f)(x)\leq \|f\|_{L^{\infty}(\Omega)}\int_{K}\gk{\Omega}(x,y)\ dy
        \leq  c(n,s)\|f\|_{L^{\infty}(\Omega)}\int_{K}|y-a_x|^{s}\frac{\delta_{\Omega}^{s}(x)}{|x-y|^{n-s}}\ dy.
    \end{equation}
Further, by the triangle inequality
    \begin{equation}\label{eqn triangle inequality}
        \frac{|y-a_x|}{|x-y|}\leq 1+\frac{|x-a_x|}{|x-y|}\leq 1+\frac{|x-a_x|}{d(K,\Omega_{\rho})}.
    \end{equation}
     Due to the uniform ball condition, $|x-a_x|\leq \rho+R$. Using these in \eqref{eqn greens potential for supported f}
    \begin{equation*}
       \g{\Omega}(f)(x)\leq C(K,\rho,n,s)\|f\|_{L^{\infty}(\Omega)}\delta_{\Omega}^{s}(x)\int_{K}\frac{dy}{|x-y|^{n-2s}}\leq C(K,\rho,n,s)\|f\|_{L^{\infty}(\Omega)} \delta_{\Omega}^{s}(x).
    \end{equation*}
    
    \textit{Second part:}
    Let $f\in L^{1}_{loc}(\Omega,\delta_{\Omega}^{s})$ and let $K$ be any compact subset of $\Omega$. Choose $\rho>0$ such that $K\subset\subset \Omega\setminus\Omega_{\rho}$. Then by using $\gk{\Omega}(x,y)\leq \Gamma(x-y)$ and \eqref{eqn greens estimatte on unbounded domain}, we obtain
    \begin{equation*}
        \begin{split}
            \int_{K}|\g{\Omega}(f)(x)|\ dx&=\int_{K}\left\{\int_{\Omega\setminus\Omega_{\rho}}+\int_{\Omega_{\rho}}\right\}|f(y)|\gk{\Omega}(x,y)\ dy\ dx\\
            &\leq \int_{\Omega\setminus\Omega_{\rho}}\int_{K}\frac{|f(y)|}{|x-y|^{n-2s}}\ dx\ dy+\int_{\Omega_{\rho}}|f(y)|\delta_{\Omega}^{s}(y)\int_{K}\frac{|x-a_y|^{s}}{|x-y|^{n-s}}\ dx \ dy.
        \end{split}
    \end{equation*}
    By using the triangle inequality as in \eqref{eqn triangle inequality}, for the second inequality, we infer that
    \begin{equation*}
        \int_{K}|\g{\Omega}(f)(x)|\ dx\leq C(K,n,s,\rho)\|f\|_{L^{1}(\Omega,\delta_{\Omega}^{s})}+\Tilde{C}(K)\|f\|_{L^{1}(\Omega,\delta_{\Omega}^{s})}\leq C(K,n,s)\|f\|_{L^{1}(\Omega,\delta_{\Omega}^{s})}.
    \end{equation*}

    \textit{Third part:} Let $f\in L^{1}(\Omega,\delta_{\Omega}^{s})\cap L_{loc}^{q}(\Omega)$, for some $q>n/2s$, and $K$ be a compact subset of $\Omega$. For $\rho>0$ as in the previous case, then, as in the last part
    \begin{equation*}
        |\g{\Omega}(f)(x)| \leq \int_{\Omega\setminus\Omega_{\rho}}\frac{|f(y)|}{|x-y|^{n-2s}}\ dy+\int_{\Omega_{\rho}}|f(y)|\delta_{\Omega}^{s}(y)\frac{|x-a_y|^{s}}{|x-y|^{n-s}} \ dy,
    \end{equation*}
    Again, by the triangle inequality as in \eqref{eqn triangle inequality} 
    , we have
    \begin{equation*}
        \int_{\Omega_{\rho}}|f(y)|\delta_{\Omega}^{s}(y)\frac{|x-a_y|^{s}}{|x-y|^{n-s}} \ dy\leq C(K,\rho)\|f\|_{L^{1}(\Omega,\delta^{s}_{\Omega})}.
    \end{equation*}
    Choose $r>0$ such that for all $x\in K$, $B_{r}(x)\subset\Omega$, by the H\"older's inequality
    \begin{equation*}
        \int_{(\Omega\setminus\Omega_{\rho})\cap B_{r}(x)}\frac{|f(y)|}{|x-y|^{n-2s}}\ dy\leq \|f\|_{L^{q}(B_{r}(x))}\left\{\int_{B_{r}(x)}\frac{dy}{|x-y|^{q'(n-2s)}}\right\}^{1/q'},
    \end{equation*}
    and is finite owing to $q>n/2s$. The result follows by combining the above equations with the integral on $(\Omega\setminus\Omega_{\rho})\setminus B_{r}(x)$, which is bounded by $\left(1/(r^{n-2s}\rho^{s})\right)\|f\|_{L^{1}(\Omega,\delta_{\Omega}^{s})}$.

    Finally, for $f\in L_{c}^{\infty}(\Omega)$, denote $K=\text{Supp}(f)$. Take $x\in \overline{\Omega}\setminus K$ and $x_{i}\to x$ in $\in \Omega\setminus K$. For $i$ large since $|y-x_{i}|>d(x,K)/2$ for all $y\in K$, by $G_{\Omega}(x,y)\leq |x-y|^{-n+2s}$ and dominated convergence theorem $\g{\Omega}(f)(x_{i})\to \g{\Omega}(x)$.
    For $x\in K$, employing the similar argument of first part of \cite[Theorem 2.10]{AbatangeloVazquez2023} with $f\in L^{\infty}(K)$ we obtain that if $x_{i}\to x$ in $K$ then $\g{\Omega}(f)(x_{i})\to\g{\Omega}(f)(x)$.
   
\end{proof}
 \begin{remark}
    The last part of the proof can also be argued more directly. One can divide the integral representing $\g{\Omega}(f)(x_{i})-\g{\Omega}(f)(x)$ into two integrals, over a ball $B_{r}$ of radius $r$ at $x$ and in $K\setminus B_{r}$. By the dominated convergence theorem as above, the integral difference on $K\setminus B_{r}$ tends to zero. The integral on the ball can be made small by choosing $r$ small. This concludes the continuity.
\end{remark}  

We next construct a barrier for the Poisson kernel on a domain in $\R^{n}$.
\begin{theorem}\label{theorem poission behaviour}
    Let $\Omega$ be a domain with uniform exterior ball condition, then for some $C(x,n,s)>0$ we have
    \begin{equation}\label{eqn poison kernel estimate}
        -\slap \gk{\Omega}(x,y)\leq C(x,n,s) \delta^{-s}_{\Omega}(y).
    \end{equation}
    In particular, for any compact subset $K$ of $\Omega$, \eqref{eqn poison kernel estimate} holds for all $x\in K$, with a constant $C=C(K,n,s)$.
\end{theorem}
\begin{proof}
Let $r>0$ be such that $B_{2r}(x)\subset \subset \Omega$. Denote $\Omega_{r}=\{x\in\Omega~|~\delta_{\Omega}(x)<r\}$. For any $y\in \R^{n}\setminus\overline{\Omega}$, one can write
    \begin{equation*}
        -\slap\gk{\Omega}(x,y)=\int_{\Omega}\frac{\gk{\Omega}(x,z)}{|y-z|^{n+2s}}\ dz=\left\{\int_{B_{r}(x)}+\int_{(\Omega\setminus\Omega_{r})\setminus B_{r}(x))}+\int_{\Omega_{r}}\right\}\frac{\gk{\Omega}(x,z)}{|y-z|^{n+2s}}\ dz.
    \end{equation*}
    We estimate the first integral using $\gk{\Omega}(x,z)\leq \Gamma(x-z)$, $|z-y|^{n+s}>d(B_{r}(x),\partial\Omega)^{n+s}$, $|z-y|^{s}>\delta_{\Omega}^{s}(y)$ and by changing the variables to polar coordinate, to get
    \begin{equation*}
        \int_{B_{r}(x)}\frac{\gk{\Omega}(x,z)}{|y-z|^{n+2s}}\ dz\leq \frac{a(n,s)\omega_{n-1}}{d(B_{r}(x),\partial\Omega)^{n+s}\delta^{s}_{\Omega}(y)}\int_{0}^{r}\rho^{2s-1}\ d\rho=\frac{C(n,s)r^{2s}}{d(B_{r}(x),\partial\Omega)^{n+s}\delta^{s}_{\Omega}(y)}.
    \end{equation*}
    Further, by using $|x-z|>r$ for all $z\in \Omega\setminus B_{r}(x)$ and $|y-z|>r$ for all $z\in \Omega\setminus\Omega_{r}$, one can infer 
    \begin{equation*}
      \int_{(\Omega\setminus\Omega_{r})\setminus B_{r}(x))}  \frac{\gk{\Omega}(x,z)}{|y-z|^{n+2s}}\ dz\leq \frac{a(n,s)}{r^{n-s}}\int_{(\Omega\setminus\Omega_{r})\setminus B_{r}(x))}\frac{dz}{|y-z|^{n+s}}.
    \end{equation*}
    Next, by the triangle inequality as in \eqref{eqn triangle inequality} for $z\in \Omega_{r}$, we have $|x-a_{z}|\leq C(r,n,s) |x-z|$. Therefore using \eqref{eqn greens estimatte on unbounded domain} and $\delta_{\Omega}(z)<|y-z|$, we get
    \begin{equation*}
    \begin{split}
      \int_{\Omega_{r}} \frac{\gk{\Omega}(x,z)}{|y-z|^{n+2s}}\ dz&=C(r,n,s)\int_{\Omega_{r}}\frac{|x-a_{z}|^{s}\delta_{\Omega}^{s}(z)}{|x-z|^{n-s}}\frac{1}{|y-z|^{n+2s}}\ dz\\
      &\leq C(r,n,s)\int_{\Omega_{r}}\frac{1}{|x-z|^{n-2s}}\frac{1}{|y-z|^{n+s}}\ dz 
      \leq \frac{C(r,n,s)}{r^{n-2s}}\int_{\Omega_{r}}\frac{1}{|y-z|^{n+s}}\ dz.
      \end{split}
    \end{equation*}
   In the case of both preceding integrals, since $\Omega_{r}, (\Omega\setminus\Omega_{r})\setminus B_{r}(x) \subset B_{\delta_{\Omega}(y)}^{C}(y)$, by changing the variables to polar coordinate we get
    \begin{equation*}
        \begin{split}
        \left\{\int_{(\Omega\setminus\Omega_{r})\setminus B_{r}(x))}+\int_{\Omega_{r}}\right\}\frac{\gk{\Omega}(x,z)}{|y-z|^{n+2s}}\ dz&\leq 2C(r,n,s)\int_{B_{\delta_{\Omega}(y)}^{C}(y)}\frac{1}{|y-z|^{n+s}}\ dz 
         \leq C(r,n,s) \omega_{n-1} \delta_{\Omega}^{-s}(y).
        \end{split}
    \end{equation*}
    With this, one can infer the pointwise  estimate.

    For the local uniform estimate, Let $K  \subset  \Omega$ be compact, choose $r>0$ such that for all $x\in K$, $B_{r}(x)\subset \Omega$, then by the above calculations
    \begin{equation*}
         -\slap\gk{\Omega}(x,y)\leq \left\{\frac{C(n,s)r^{2s}}{d(K,\partial\Omega)^{n+s}}+C(r,n,s)\right\}\delta_{\Omega}^{-s}(y)=C(K,n,s)\delta_{\Omega}^{-s}(y).\qedhere
    \end{equation*}
\end{proof}
\begin{remark}
    The preceding theorem justifies the assumption \ref{G}.
\end{remark}

 Unlike local operators, the fractional Laplacian exhibits a different behavior. For example, when $h>0$, Lemma \ref{lemma equivalent E} implies that any solution of \eqref{eqn main linear omega} blows up at the boundary, even when $f=0$ and $g=0$, which accounts for the s-harmonic function $u_{1-s}$ on the unit ball as discussed in the introduction. In addition, it is shown in \cite{AbatangeloVazquez2023, Abaratgelo2017} that large solutions can be generated with the data $(\delta^{-\beta}_{\Omega},0,0)$ for some $2s<\beta<1+s$.
 \begin{proposition}[{\cite[Theorem 3.4]{AbatangeloVazquez2023}}]\label{prop boundary behaviour of the solution on bounded domain}
     Let $\Omega$ be a bounded domain and $u$ be the solution of \eqref{eqn main linear omega} with data $(1/\delta_{\Omega}^{\beta}(x), 0,0)$. 
     Then, there are positive constants $c_{1}$ and $c_{2}$ such that 
     \begin{equation*}
             c_{1}\delta_{\Omega}^{-\beta+2s}(x)\leq u(x)\leq c_{2}\delta_{\Omega}^{-\beta+2s}(x) \hspace{1.8cm}\text{for } s<\beta<1+s.
          \end{equation*}
\end{proposition}
 If the exterior data $g$ diverges at the boundary, the corresponding solution with data $(0, g, 0)$ likewise diverges at the boundary as stated below,
\begin{proposition}\cite{Abatangelo2015LargesHarmonic}
    Let $g$ satisfy \ref{G1} and $g(y)\to +\infty$ as $y\to \partial\Omega$ for $y\in \rn\setminus\overline{\Omega}$. If $u$ is a solution of \eqref{eqn main linear omega} with the data $(0,g,0)$, then $u(x)\to +\infty$ as $x\in \Omega$ and $x\to \partial\Omega$.
\end{proposition}
 For a proof we refer to \cite[Remark 3 and subsection 3.6.2]{Abatangelo2015LargesHarmonic}.

\subsection{Green's kernels in nested domains} When considering a sequence of finite cylinders, the domains are nested, necessitating a comparison of their respective Green's kernels. Accordingly, we analyze the behavior of the Green's kernel on two  bounded domains, where one domain is contained within the other. Here we provide an analytic proof using maximum principle.
\begin{lemma}\label{lemma greens seq is decreasing}
    Let $\Omega_{1}\subset \Omega_{2}$ be two bounded domains, then $\gk{\Omega_{1}}(x,y)\leq \gk{\Omega_{2}}(x,y)$ for all $x,y\in \Omega_{1}$. 
\end{lemma}
\begin{proof}
    Since $\gk{\Omega_{2}}(x,y)$ is nonnegative, for $x\in \Omega_{1}$ and $y\in \Omega_{2}\setminus\Omega_{1}$, $\Gamma_{s}(x-y)\geq H_{\Omega_{2}}(x,y)$. Therefore, $H_{\Omega_{2}}(x,\cdot)$ solves 
    \begin{equation*}
        \begin{cases}
          \slap H_{\Omega_{2}}(x,\cdot)=0 \quad &\text{on } \Omega_{1}\\
          H_{\Omega_{2}}(x-y)\leq \Gamma_{s}(x-y) &\text{on }\R^{n}\setminus\overline{\Omega}_{1}.
        \end{cases}
    \end{equation*}
  Then $H_{\Omega_{1}}(x,y)\geq H_{\Omega_{2}}(x,y)$ on $\rn$ by the maximum principle. From this we infer $\gk{\Omega_{1}}(x,y)\leq \gk{\Omega_{2}}(x,y)$ in $\rn$. 
\end{proof}
\begin{lemma}
    Let $\Omega_{1}\subset\Omega_{2}$ be two bounded domains, $\psi\in L_{c}^{\infty}(\Omega_{1})$ such that $\psi\geq 0$. Then for $z\in \R^{n}\setminus\overline{\Omega}_{2}$, $-\slap \g{\Omega_{1}}(\psi)(z) \leq -\slap\g{\Omega_{2}}(\psi)(z)$. 
\end{lemma}
\begin{proof}
    By Definition \ref{defining the operator G} and Lemma \ref{lemma greens seq is decreasing} we get $0\leq \g{\Omega_{1}}(\psi)(x)\leq \g{\Omega_{2}}(\psi)(x)$ for all  $x\in \rn.$  Then for $z\in \R^{n}\setminus\overline{\Omega}_{2}$,
    \begin{equation*}
        \begin{split}
            -\slap \g{\Omega_{1}}(\psi)(z)&= \mathcal{A}(n,s)\text{ p.v} \int_{\rn}\frac{\g{\Omega_{1}}(\psi)(y)}{|z-y|^{n+
            2s}}\ dy
            \leq \mathcal{A}(n,s)\text{ p.v} \int_{\rn}\frac{\g{\Omega_{2}}(\psi)(y)}{|z-y|^{n+2s}}\ dy =-\slap \g{\Omega_{2}}(\psi)(z) .
        \end{split}
    \end{equation*}
    Thus the result follows.
\end{proof}

\subsection{Green's Kernel and Potential on finite cylinders}
We examine properties of the Green’s kernel and potential for the fractional Laplacian on finite cylindrical domains, and study their behavior as the cylinder length approaches infinity. First, define  $\Omega_{\ell}$ as follows:
\begin{equation}\label{dfn of omegaell}
    \Omega_{\ell}:=\{x\in\rn:\sum_{1}^{m}x_{i}^{2}<1, -\ell<x_{j}<\ell, m<j\leq n \}\cup\{x\in\rn:\sum_{1}^{m}x_{i}^{4}+\sum_{m+1}^{n}(x_{j}\pm \ell)^{4}<1\}.
\end{equation}
This is a $C^{2}$ domain.
\begin{lemma}\label{lemma convergence of greens kernel}
  For any $x,y \in\Omega_{\infty}$, one has $\gk{\Omega_{\ell}}(x,y)\to \gk{\Omega_{\infty}}(x,y)$ as $\ell \to \infty$. Further, for $x\in \Omega_{\infty}$, the convergence is locally uniform.
\end{lemma}
\begin{proof}
 It suffices to show that the sequence of corresponding s-harmonic part of the Green's kernel converges to a function $H$ which is s-harmonic in $\Omega_\infty$. $H$ is s-harmonic by the mean value property \cite[Proposition 2.2]{Silvestre2007}. 
\smallskip
 
From the proof of Lemma \ref{lemma greens seq is decreasing}, the sequence $\{H_{\Omega_{\ell}}(x,y)\}_{\ell}$ is decreasing pointwise, and by the maximum principle $H_{\Omega_{\ell}}(x,y)\geq 0$. For any $x\in \Omega_{\infty}$ and $y\in \rn$, let 
    \begin{equation*}
        H(x,y):=\lim_{\ell\to\infty}H_{\Omega_{\ell}}(x,y).
    \end{equation*}
    Clearly,  $H(x,y)\geq 0$ and $H(x,y)=\Gamma_{s}(x-y)$ for $y\in \R^{n}\setminus\overline{\Omega}_{\infty}$. Let $x,y\in \Omega_{\infty}$ and $r>0$ such that $B_{r}(y)\in\Omega_{\infty}$. Choose $\ell_{0}$ such that $x\in \Omega_{\ell_{0}}$ and $B_{r}(y)\subset \Omega_{\ell_{0}}$. Then by the mean value formula for the fractional Laplacian, we obtain
    \begin{equation*}
        H_{\Omega_{\ell}}(x,y)=\int_{\R^{n}\setminus B_{r}(y)}\frac{c(n,s)r^{2s}}{|y-z|^{n}\left(|y-z|^{2}-r^{2}\right)^{s}}H_{\Omega_{\ell}}(x,z)\ dz.
    \end{equation*}
    As $\ell\to\infty$, the integral converges by the dominated convergence theorem, since Lemma \ref{lemma greens seq is decreasing} implies $H_{\Omega_{\ell}}(x,z)\leq H_{\Omega_{\ell_{0}}}(x,z)$. Thus $H(x,\cdot)$ satisfies the mean value formula. 
    \smallskip 

     We next prove that the convergence is locally uniform.
    For any compact set $K\subset \Omega_{\infty}$, without loss of generality let $\ell_{0}$ be such that $K\subset \Omega_{\ell_{0}}$. The inequality in the previous line and Lemma \ref{lemma uniform bound for H} implies that $\{H_{\Omega_{\ell}}(x,\cdot)\}_{\ell\geq \ell_{0}}$ is uniformly bounded on $K$. Also, for some $\alpha<2s$ if $2s\leq 1$ or $\alpha=1$ if $2s>1$, we have by \cite[Proposition 2.9]{Silvestre2007} that
    \begin{equation*}
       [H_{\Omega_{\ell}}(x,\cdot)]_{C^{\alpha}(\Omega_{\ell})}\leq [H_{\Omega_{\ell}}(x,\cdot)]_{C^{\alpha}(\rn)}\leq C ||H_{\Omega_{\ell}}(x,\cdot)||_{L^{\infty}(\rn)}\leq C ||H_{\Omega_{\ell_{0}}}(x,\cdot)||_{L^{\infty}(\rn)},
    \end{equation*}
Thus, the class is equicontinuous with respect to the second variable. Hence, up to a subsequence, the convergence of $\{H_{\Omega_{\ell}}(x,\cdot)\}_{\ell\geq \ell_{0}}$ to $H(x,\cdot)$ is uniform on $K$ and for fixed $x\in \Omega_{\infty}$. Due to the monotonicity of the sequence, the whole sequence converges. In particular, $H(x,\cdot)$ is continuous. This implies that $H(x,\cdot)$ is s-harmonic.  We infer that $H(x,y)=H_{\Omega_{\infty}}(x,y)$ in $\Omega_{\infty}$. 
    \end{proof}
    
    \begin{lemma}\label{Lemma local uniform convergence of solution operator}
        Let $\psi\in L^{\infty}_{c}(\Omega_{\infty})$, then $\g{\Omega_{\ell}}(\psi)(x)\to\g{\Omega_{\infty}}(\psi)(x)$ for all $x\in \Omega_{\infty}$. Moreover, the convergence is locally uniform  outside a compact set.
    \end{lemma}
\begin{proof}
By Lemma \ref{lemma greens seq is decreasing} and Lemma \ref{lemma convergence of greens kernel}, $\gk{\Omega_{\ell}}(x,\cdot)\uparrow \gk{\Omega_{\infty}}(x,\cdot)$. Thus, by the dominated convergence theorem, we obtain as $\ell \to \infty$, that
\begin{equation*}
    \g{\Omega_{\ell}}(\psi)(x)=\int_{\text{Supp}(\psi)}\gk{\Omega_{\ell}}(x,y)\psi(y)\ dy\to \int_{\text{Supp}(\psi)}\gk{\Omega_{\infty}}(x,y)\psi(y)\ dy=\g{\Omega_{\infty}}(\psi)(x).
\end{equation*}
 To obtain the uniform convergence, we show that the sequence $\{\g{\Omega_{\ell}}(\psi)\}_{\ell}$, for $\ell$ large, is equicontinuous and uniformly bounded outside a compact set. Let $x_{1},x_{2}\in \Omega_{\infty}$, choose $\ell_{0}>0$ such that Supp$(\psi)\subset \Omega_{\ell_{0}}$ and $x_{1},x_{2}\in \Omega_{\ell_{0}}$. For any $\ell\geq \ell_{0}$, by symmetry of $\gk{\Omega_{\ell}}$ and, for some $\alpha<2s$ if $2s\leq 1$ or $\alpha=1$ if $2s>1$ as in the Lemma \ref{lemma convergence of greens kernel}, we get
    \begin{equation*}
    \begin{split}
         \left|\g{\Omega_{\ell}}(\psi)(x_{1})-\g{\Omega_{\ell}}(\psi)(x_{2})\right| &\leq \int_{\Omega_{\ell}}|\gk{\Omega_{\ell}}(y,x_{1})-\gk{\Omega_{\ell}}(y,x_{2})| |\psi(y)|\ dy\\
       &\leq C\int_{\text{supp}(\psi)}\|H_{\Omega_{\ell_{0}}}(y,\cdot)\|_{L^{\infty}(\rn)}|x_{1}-x_{2}|^{\alpha}|\psi(y)|\ dy\\
       &\leq C(n,s,\text{supp}(\psi)) |x_{1}-x_{2}|^{\alpha} \|\psi\|_{L^{1}(\text{Supp}(\psi))},
    \end{split}
    \end{equation*}
    where the last inequality follows by taking the supremum with respect to both the variables in Lemma \ref{lemma uniform bound for H}.  Therefore, $\{\g{\Omega_{\ell}}(\psi)\}_{\ell>\ell_{0}}$ is equicontinuous. Next for $x\notin ~$Supp$(\psi)$ and 
    \begin{equation*}
        \begin{split}
          \left | \g{\Omega_{\ell}}(\psi)(x)\right|& \leq \int_{\text{Supp}(\psi)}\gk{\Omega_{\ell}}(x,y)|\psi(y)|\ dy\leq \int_{\text{Supp}(\psi)}\frac{a(n,s)}{|x-y|^{n-2s}}|\psi(y)|\ dy \leq \frac{a(n,s)}{d(x,\text{Supp}(\psi))^{n-2s}} ||\psi||_{L^{1}(\text{Supp}(\psi))}.  
           \end{split}
    \end{equation*}
    Choose $K\subset \Omega_{\ell_{0}}$ such that Supp$(\psi)\subset\subset K$, then for all $x\in \Omega_{\infty}\setminus K$, $\left | \g{\Omega_{\ell}}(\psi)(x)\right|\leq C(n,s,\psi, K)$.

    Thus the local uniform convergence outside $K$ follows by the Arzelà–Ascoli theorem.
     \end{proof}

\begin{lemma}\label{lemma-convergence-in-fractional-normal-derivative}
    Let $\psi\in L_{c}^{\infty}(\Omega_{\infty})$, then $\Ds \g{\Omega_{\ell}}(\psi)(\theta)\to \Ds \g{\Omega_{\infty}}(\psi)(\theta)$ for all $\theta\in \partial\Omega$.
\end{lemma}
\begin{proof}
    Consider $\psi\geq 0$, then by the Lemma \ref{Lemma local uniform convergence of solution operator} and \ref{lemma greens seq is decreasing} we get $\g{\Omega_{\ell}}(\psi)(x) \uparrow \g{\Omega_{\infty}}(\psi)(x)$. Let $\theta\in \partial\Omega_{\infty}$, we choose $\ell_{0}>0$ and a neighborhood $B$ of $\theta$ such that $B\cap \text{Supp}(\psi)=\emptyset$ and $\delta_{\Omega_{\ell}}=\delta_{\Omega_{\infty}}$ for all $\ell\geq \ell_{0}$ in $B$.
   
    Denote $u_{\ell}(x):=\g{\Omega_{\infty}}(\psi)(x)-\g{\Omega_{\ell}}(\psi)(x)$. Then $u$ is $s$-harmonic in $\Omega_{\ell}$ in weak dual sense (See Lemma \ref{lemma restriction of a solution to subdomains} bellow and \cite[Corollary 1]{Abatangelo2015LargesHarmonic}). Choose $x_{0}\in B\cap \Omega_{\infty}$ such that $u(x_{0})>0$; if such $x_{0}$ dose not exists in $B$, then the result holds trivially. Define $v(x)=u(x)\g{\Omega_{\infty}}(\psi)(x_{0})/u(x_{0})$. Then $v$ is $s$-harmonic in $B\cap \Omega_{\infty}$ and so is $\g{\Omega_{\infty}}(\psi)(x)$ by the choice of $B$. Further both $u$ and $v$ vanishes on $B\cap \Omega_{\infty}^{c}$ and $v(x_{0})=\g{\Omega_{\infty}}(\psi)(x_{0})>0$. By the boundary Harnack inequality \cite[Theorem 1]{Bogdan1997}, we get $v(x)\leq C \g{\Omega_{\infty}}(\psi)(x)$ for all $x\in \overline{B_{r}(\theta)}\cap \Omega_{\infty}$, where $x_{0}\in \overline{B_{r}(\theta)}\subset B$. Implies 
    \begin{equation*}
        \frac{u(x)}{\delta_{\Omega_{\infty}}^{s}(x)}\leq C \frac{u(x_{0})}{\g{\Omega_{\infty}}(\psi)(x_{0})}\frac{\g{\Omega_{\infty}}(\psi)(x)}{\delta_{\Omega_{\infty}}^{s}(x)}.
    \end{equation*}
    As the second ratio converges to $\Ds \g{\Omega_{\infty}}(\psi)(\theta)$ and $u(x_{0})\to 0$ by the preceding lemma, we have that 
     \begin{equation*}
        \left|\frac{\g{\Omega_{\infty}}(\psi)(x)}{\delta_{\Omega_{\infty}}^{s}(x)}-\frac{\g{\Omega_{\ell}}(\psi)(x)}{\delta_{\Omega_{\ell}}^{s}(x)}\right|= \frac{\g{\Omega_{\infty}}(\psi)(x)-\g{\Omega_{\ell}}(\psi)(x)}{\delta_{\Omega_{\infty}}^{s}(x)}\to 0
        \end{equation*}
as $\ell\to\infty$ uniformly on $\overline{B_{r}(\theta)}\cap \Omega_{\infty}$.

Thus, 
\begin{equation*}
    \Ds \g{\Omega_{\ell}}(\psi)(\theta)=\lim_{\ell\to \infty}\lim_{y\to \theta}\frac{\g{\Omega_{\ell}}(\psi)(x)}{\delta_{\Omega_{\ell}}^{s}(x)}=\lim_{y\to \theta}\lim_{\ell\to \infty}\frac{\g{\Omega_{\ell}}(\psi)(x)}{\delta_{\Omega_{\ell}}^{s}(x)}=\lim_{y\to \theta}\frac{\g{\Omega_{\infty}}(\psi)(x)}{\delta_{\Omega_{\infty}}^{s}(x)}= \Ds \g{\Omega_{\infty}}(\psi)(\theta).
\end{equation*}

The result follows for all $\psi\in L_{c}^{\infty}$ by writing $\psi=\psi^{+}-\psi^{-}$.
\end{proof}

\section{Linear Equation}\label{section linear}
This section addresses the analysis of the linear equation. Initially, we establish the existence and uniqueness of solutions on bounded domains, extend these results to domains with non-zero Lebesgue measure in their complement, and present a comparison principle for a fixed domain. Subsequently, we examine linear equations on sequences of finite cylinders, proving a comparison principle and demonstrating the convergence of solutions to a large solution on the infinite cylinder ,with different sets of data.

\subsection{On a Domain in $\rn$}\label{subsection_3.1}
In \cite{AbatangeloVazquez2023}, the existence and uniqueness for the linear equation with the data $(f,0,h)$ is established on bounded domains. In this regard, we begin by proving existence and uniqueness of a solution to \eqref{eqn main linear omega} on a bounded domain for the data $(0,g,0)$, and subsequently extend these results to a domain that is not necessarily bounded.

On a bounded domain, the optimal assumption of the complement data is
\begin{enumerate}
    \myitem{$\textbf{(G1)}$} \label{G1}\textit{ $g:\R^{n}\setminus\overline{\Omega} \to \R$ is a measurable function such that 
\begin{equation}\label{eqn condition on g}
    \int_{\R^{n}\setminus\overline{\Omega}}|g(y)|\min\{\delta_{\Omega}^{-s}(y),\delta_{\Omega}^{-n-2s}(y)\}\ dy <+\infty.
\end{equation}}
\end{enumerate}

 \begin{definition}\label{Dfn weak dual solution with f,g,h}
      Let $f\in L^{1}(\Omega,\delta^{s}_{\Omega})$, $h\in L^{1}(\partial\Omega)$ and let $g$ satisfy \ref{G1} if $\Omega$ is bounded, \ref{G} otherwise.  A function $u\in L^{1}_{loc}(\Omega)$ is said to be a weak dual solution of \eqref{eqn main linear omega}
        if for any $\psi\in L^{\infty}_{c}(\Omega)$ one has
    \begin{equation}\label{weak dual dfn with f}
        \int_{\Omega}u(x)\psi(x)\ dx=\int_{\Omega}\g{\Omega}(\psi)(x) f(x)\ dx-\int_{\R^{n}\setminus\overline{\Omega}
        }\slap \g{\Omega}(\psi)(y) g(y) \ dy+\int_{\partial\Omega
        } \Ds [\g{\Omega}(\psi)](\theta) h(\theta) \ d\mathcal{H}(\theta).
    \end{equation}
    Or a super-solution (sub-solution), if  for any $\psi\in L^{\infty}_{c}(\Omega)$ and  $\psi\geq 0$, one has
    \begin{equation*}
         \int_{\Omega}u(x)\psi(x)\ dx\geq (\leq)\int_{\Omega}\g{\Omega}(\psi)(x) f(x)\ dx-\int_{\R^{n}\setminus\overline{\Omega}
        }\slap \g{\Omega}(\psi)(y) g(y) \ dy+\int_{\partial\Omega
        } \Ds [\g{\Omega}(\psi)](\theta) h(\theta) \ d\mathcal{H}(\theta).
    \end{equation*}
    \end{definition}
    We first show that the quantity on the right-hand side of \eqref{weak dual dfn with f} is finite.
\medskip 

    \textit{The first integral:}
    By Theorem \ref{theorem greens poteantial continuity on unbounded domains}, $\g{\Omega}(\psi)=\Tilde{\psi} \, \delta_{\Omega}^{s}$ for some $\Tilde{\psi}\in L^{\infty}(\Omega)$. Hence the first integral is finite.
\smallskip 

    \textit{The second integral:} First we show that $\slap \g{\Omega}(\psi)$ is well defined. Let $\psi\in L_{c}^{\infty}(\Omega)$, the function $\phi=\g{\Omega}(\psi)$ solves \eqref{eqn main linear omega} with data $(\psi(x),0,0)$.
 For $x\in \R^{n}\setminus\overline{\Omega}$, consider
 \begin{equation*}
 \begin{split}
     \slap \phi(x) &=\mathcal{A}(n,s) \lim_{\epsilon\to 0} \int_{\Omega\setminus B_{\epsilon}(x)}\frac{-\phi(y)}{|x-y|^{n+2s}}\ dy 
     =-\mathcal{A}(n,s)  \int_{\Omega}\frac{1}{|x-y|^{n+2s}}\int_{\Omega}\psi(z) \gk{\Omega}(y,z)\ dz\ dy\\
     &= \mathcal{A}(n,s)  \int_{\Omega}\psi(z)  \int_{\Omega}\frac{-\gk{\Omega}(z,y)}{|x-y|^{n+2s}}\ dy \ dz
     = \int_{\Omega}\psi(z) \slap \gk{\Omega}(z,x)\ dz.
      \end{split}
 \end{equation*}
 We now consider the second integral on the right hand side of \eqref{weak dual dfn with f}.
 \begin{equation*}
     -\int_{\R^{n}\setminus\overline{\Omega}
        }\slap \g{\Omega}(\psi)(y) g(y) \ dy=\int_{\R^{n}\setminus\overline{\Omega}
        }\left\{\mathcal{A}(n,s)\int_{\Omega}\frac{\g{\Omega}(\psi)(z)}{|y-z|^{n+2s}}\ dz\right\}g(y)\ dy.
 \end{equation*}
 By Theorem \ref{theorem greens poteantial continuity on unbounded domains} and since $\delta_{\Omega}(z)\leq |y-z|$ whenever $z\in \Omega$ and $y\in \R^{n}\setminus\Omega$, we have
 \begin{equation*}
     \int_{\Omega}\frac{\g{\Omega}(\psi)(z)}{|y-z|^{n+2s}}\ dz\leq C\int_{\Omega}\frac{\delta_{\Omega}^{s}(z)}{|y-z|^{n+2s}}\ dz\leq C\int_{\Omega}\frac{dz}{|y-z|^{n+s}}\leq C\int_{\R^{n}\setminus B_{\delta_{\Omega}(y)}(y)}\frac{dz}{|y-z|^{n+s}}.
 \end{equation*}
 Here, the last integral is over the complement of the ball centered at $y$ and radius $\delta_{\Omega}(y)$. By polar coordinates we obtain
 \begin{equation*}
     \int_{\R^{n}\setminus B_{\delta_{\Omega}(y)}(y)}\frac{dz}{|y-z|^{n+s}}=\omega_{n-1}\int_{\delta_{\Omega}(y)}^{\infty}\frac{d\rho}{\rho^{1+s}}=\frac{\omega_{n-1}}{s}\delta_{\Omega}^{-s}(y).
 \end{equation*}
 That is 
 \begin{equation}\label{eqn complement integral behaviour}
    -\int_{\R^{n}\setminus\overline{\Omega}
        }\slap \g{\Omega}(\psi)(y) g(y) \ dy\leq  C \int_{\R^{n}\setminus\overline{\Omega}}|g(y)|\delta_{\Omega}^{-s}(y)\ dy.
        \end{equation}
   \smallskip 
   
 \textit{The third integral:} It is enough if $\Ds\g{\Omega}(\psi)\in L^{\infty}(\partial\Omega)$. This holds because
 \begin{equation*}
     |\Ds\g{\Omega}(\psi)(\theta)|=\lim_{\substack{y\to \theta\\ y\in \Omega}}\frac{|\g{\Omega}(\psi)(y)|}{\delta_{\Omega}^{s}(y)}=\lim_{\substack{y\to \theta\\ y\in \Omega}}|\Tilde{\psi}(y)|\leq \|\Tilde{\psi}\|_{L^{\infty}(\Omega)}.
     \end{equation*}
The above weak dual notion of a solution is studied in \cite{AbatangeloVazquez2023} for data $(f,0,h)$, and in \cite{BonforteFigalliVazquez2018} for data $(f,0,0)$ with $f\in L^{q}(\Omega)$ for some $ q>n/2s$.

\begin{remark}\label{remark-distrubutional-weak-dual-solution}
    The integrability assumption $L^{1}(\Omega,\delta_{\Omega}^{s})$ guarantee the finiteness of integrals with $\gk{\Omega}(x,\cdot)$ and $\g{\Omega}(\psi)$ respectively. Thus for any $f\geq 0$ such that $\int_{\Omega}\gk{\Omega}(x,y)f(y)\,dy<+\infty$ then for any $\psi\geq 0$ in $L^{\infty}_{c}(\Omega)$,we have by Tonelli that
    \begin{equation*}
        \int_{\Omega}\left(\int_{\Omega}\gk{\Omega}(x,y)f(y)\,dy\right)\psi(x)\,dx=\int_{\Omega}\left(\int_{\Omega}\gk{\Omega}(x,y)\psi(x)\,dx\right)f(y)\,dy=\int_{\Omega}f(y)\g{\Omega}(\psi)(y)\,dy.
    \end{equation*}
    Moreover, for any $\psi\in L_{c}^{1}(\Omega)$, writing $\psi=\psi^{+}-\psi^{-}$ and using $\g{\Omega}(\psi)=\g{\Omega}(\psi^{+})-\g{\Omega}(\psi^{-})$, the above equation holds for all $\psi\in L_{c}^{1}(\Omega)$.
    \end{remark}

\begin{theorem}\label{thrm linear existence with g}
   Let $\Omega$ be a domain with uniform exterior ball condition and $g$ satisfy \ref{G}. Then the function, given by
\begin{equation}\label{representation with g}
 u(x)=\begin{cases}
        -\int_{\R^{n}\setminus\overline{\Omega}}\slap \gk{\Omega}(x,y)g(y)\ dy. \quad &\text{for } x\in \Omega\\
        g(x) &\text{for }x\in \R^{n}\setminus\overline{\Omega},
    \end{cases}
\end{equation}
is the unique solution of the problem \eqref{eqn main linear omega} when $f=0$ and $h=0$.
\end{theorem}
\begin{proof}
For any $\psi\in L_{c}^{\infty}(\Omega)$, by the definition of Green's potential, using Theorem \ref{theorem poission behaviour} to employ Fubini's theorem, one can write
\begin{equation*}
    \begin{split}
        \int_{\Omega}u(x)\psi(x)\ dx 
        =-\int_{\R^{n}\setminus\overline{\Omega}}g(y)\left\{\int_{\Omega}\psi(x) \slap \gk{\Omega}(x,y)\ dx \right\}\ dy = -\int_{\R^{n}\setminus\overline{\Omega}}g(y) \slap \g{\Omega}(\psi)(y)\ dy.
    \end{split}
\end{equation*}
Therefore, $u$ is a weak dual solution as desired. 
For the uniqueness, if $u$ and $v$ are two solutions, then 
\begin{equation*}
    \int_{\Omega}u\psi=\int_{\Omega}v\psi \quad \psi\in L_{c}^{\infty}(\Omega).
\end{equation*}
By taking $\psi=$sign$(u(x)-v(x))\chi_{K}$ for any compact set $K$ of $\Omega$, $u=v$ a.e., in $\Omega$.
\end{proof}
Further, in view of Theorem \ref{theorem poission behaviour}, the solution obtained above is locally uniformly bounded. 

\begin{lemma}\label{lemma bound of the sln ug}
  Let $u$ be as in \eqref{representation with g} and $\Omega$ be a domain with uniform exterior ball condition. If $K\subset \Omega$ be compact, then for any $x\in K$  one has 
    \begin{equation*}\label{eqn l infty bound for solution with g}
        |u(x)|\leq C(n,s,K)\|g\delta_{\Omega}^{-s}\|_{L^{1}(\R^{n}\setminus\overline{\Omega})}.
    \end{equation*}  
\end{lemma}
\begin{remark}
    For a bounded domain $\Omega$, combining Theorem \ref{thrm linear existence with g} and the \cite[Theorem 2.5, Theorem 4.6]{AbatangeloVazquez2023} implies that the unique solution to \eqref{eqn main linear omega} is of the form
    \begin{equation}\label{eqn representation of a solution on omega}
    u(x)=\int_{\Omega}\gk{\Omega}(x,y)f(y)\ dy-\int_{\R^{n}\setminus\overline{\Omega}}\slap \gk{\Omega}(x,y)g(y)\ dy+\int_{\partial\Omega}\Ds\gk{\Omega}(\theta,x)h(\theta)\ d\mathcal{H}(\theta).  \end{equation}
    Given $f$ satisfying \ref{F1}, $g$ satisfying \ref{G1} and $h\in L^{1}(\partial\Omega)$.
\end{remark}
For a domain that is not necessarily bounded, we have the following.
\begin{theorem}\label{thrm rep of u on any domain}
    Let $\Omega$ be a $C^{1,1}$ domain with uniform exterior ball condition. Also, let $f$ satisfy \ref{F1}, $g$ satisfy \ref{G} and $h\in L^{1}(\partial\Omega)$, then the unique solution of \eqref{eqn main linear omega} is of the form \eqref{eqn representation of a solution on omega}.
\end{theorem}
\begin{proof}
    In view of the Theorem \ref{thrm linear existence with g}, we assume that $g=0$. By the third part of Theorem \ref{theorem greens poteantial continuity on unbounded domains} and \eqref{eqn martin behaviour}, using Fubini's theorem, the function $u$ defined in \eqref{eqn representation of a solution on omega} is a solution of the problem \eqref{eqn main linear omega}. For the uniqueness, the proof is the same as that in Theorem \ref{thrm linear existence with g}.
    \end{proof}
 
Next, we prove a comparison principle for \eqref{eqn main linear omega}.

\begin{lemma}(Comparison principle)\label{lemma comparison principle}
  Let $f_{i}\in L^{1}(\Omega,\delta_{\Omega}^{s})$, $g_{i}$ satisfies \ref{G1} and $h_{i}\in L^{1}(\partial\Omega)$ for $i=1,2$. Let $u_{1}, u_{2}\in L^{1}_{loc}(\Omega)$ be super-solution and sub-solution of \eqref{eqn main linear omega} respectively. If $f_{1}\geq f_{2}$, $g_{1}\geq g_{2}$ and $h_{1}\geq h_{2}$ then $u_{1}\geq u_{2}$ a.e. in $\Omega$. 
\end{lemma} 
\begin{proof}
Let $K$ be a compact set in $\Omega$, take $\psi=\chi_{K}\in L_{c}^{\infty}(\Omega)$. Using $\gk{\Omega}(x,y)\geq 0$ in Definition \ref{defining the operator G} with $f=\psi$, $\g{\Omega}(\psi)(x)\geq 0$ for all $x\in \Omega$. For $y\in \R^{n}\setminus\overline{\Omega}$, by the definition of $\slap$, $\slap \g{\Omega}(\psi) (y)\leq 0$.  Further, Definition \ref{frac normal derivative} implies $\Ds[\g{\Omega}(\psi)](\theta)\geq 0$ for all $\theta\in \partial\Omega$.  Using the definition of sub and super-solution,
       \begin{equation*}
        \begin{split}
            \int_{K}u_{1}\psi&=\int_{\Omega}u_{1}\psi \\
            &\geq \int_{\Omega}\g{\Omega}(\psi)(x)f_{1}(x)\ dx-\int_{\R^{n}\setminus\overline{\Omega}}\slap \g{\Omega}(\psi)(y)g_{1}(y)\ dy+\int_{\partial\Omega}\Ds[\g{\Omega}(\psi)](\theta)h_{1}(\theta)\ d\mathcal{H}(\theta)\\
            &\geq \int_{\Omega}\g{\Omega}(\psi)(x)f_{2}(x)\ dx-\int_{\R^{n}\setminus\overline{\Omega}}\slap\g{\Omega}(\psi)(y)g_{2}(y)\ dy+\int_{\partial\Omega}\Ds[\g{\Omega}(\psi)](\theta)h_{2}(\theta)\ d\mathcal{H}(\theta)\\
            &\geq \int_{\Omega}u_{2}\psi=\int_{K}u_{2}\psi
        \end{split}
    \end{equation*}
    Thus, $u_{1}\geq u_{2}$ a.e on $K$. Since $K$ is arbitrary, the result follows. 
\end{proof}

\begin{remark}
    When $f_{1}=f_{2}$, we get back the classical comparison principle.
\end{remark}

\subsubsection*{\textbf{Dimension reduction}}

A noteworthy phenomenon in the context of an infinite cylinder is the correspondence between the problem formulated on the infinite cylinder and that on its cross-section. Let $x \in \Omega_{\infty}$ be denoted by $x=(x', x'')$, and let $B_{r}^{k}(0)$ be the $k$ dimensional ball of radius $r$ centered at $0\in\R^{k}$ and denote $B_{r}^{k}(0) = B^{k}_r$. If $u: B^{m}_1 \to \mathbb{R}$ is such that $\slap_{(m)}u(x')$ is defined pointwise for each $x' \in B^{m}_1$, then the function $U(x) = u(x')$ for all $x \in \Omega_{\infty}$ satisfies $\slap_{(m)}u(x') = \slap_{(n)}U(x)$; see \cite[Lemma 3.1 and 3.2]{ChowdhuryRoy2017}. The same is true if $B^{m}_1$ is replaced by $\R^{m}$.

We examine this phenomenon within the framework of the dual weak formulation. Accordingly, we prove the following Theorems.

\begin{theorem}\label{thrm-dim-reduction-for-Greens}
 For $m > 2s$, it holds that $\gk{B_{1}^{m}(0)}(x', y') = \int_{\mathbb{R}^{n-m}} \gk{\Omega_{\infty}}(x, y)\, dy''$ for all $x', y' \in B_{1}^{m}(0)$.
\end{theorem}
\begin{proof}
The proof is divided into three steps. First, we show that the $\mathbb{R}^{n-m}$-integral of the fundamental solution over the unbounded component of $\Omega_{\infty}$ coincides with the $m$-dimensional fundamental solution. Second, we establish that the $\mathbb{R}^{n-m}$-integral of the Green’s kernel is a distributional solution with the Dirac delta as the source term. Finally, we show that the corresponding harmonic integral is an $s$-harmonic function on the cross-section.
\smallskip

\textit{Step 1.}
Following the same argument as in \cite[Lemma 3.2]{ChowdhuryRoy2017} we find 
\begin{equation*}
\begin{split}
    \int_{\mathbb R^{n-m}}
\frac{dy''}{|x-y|^{n-2s}}
&=\frac{|{S}^{n-m-1}|}{|x'-y'|^{m-2s}}\int_{0}^{\infty}\frac{t^{n-m-1}}{(1+t^{2})^{\frac{n-2s}{2}}}\,dt.
\end{split}  
\end{equation*}
Since $m > 2s$, the above integral is finite. By the substitution $r = t^2$, we obtain
\begin{equation*}
\int_0^\infty
\frac{t^{n-m-1}}
{(1+t^2)^{\frac{n-2s}{2}}}\,dt= \frac{1}{2}\int_{0}^{\infty}\frac{r^{\frac{n-m}{2}-1}}{(1+r)^{\frac{n-2s}{2}}}\,dr =\frac{\Gamma\left(\frac{n-m}{2}\right)\Gamma\left(\frac{m-2s}{2}\right)}{\Gamma\left(\frac{n-2s}{2}\right)}.
\end{equation*}
As $| S^{d-1}|=2\pi^{d/2}/\Gamma(d/2)$, we have
\begin{equation*}
    \int_{\R^{n-m}}\Gamma_{s}^{(n)}(x-y)\,dy''=\frac{a(m,s)}{|x'-y'|^{m-2s}}=\Gamma_{s}^{(m)}(x'-y').
\end{equation*}

\textit{Step 2.}
Let $G(x, y') := \int_{\mathbb{R}^{n-m}} \gk{\Omega_{\infty}}(x, (y', \xi))\, d\xi$. We assert that $(-\Delta)^{s}_{(m)} G(x, y') = \delta_{x'}$ in the distributional sense. Noting that $\gk{\Omega_{\infty}}(x,y)\leq \Gamma_{s}^{(n)}(x - y)$, by previous step it follows that $\int_{\mathbb{R}^{n-m}} \gk{\Omega_{\infty}}(x, (y', \xi))\, d\xi \leq \Gamma_{s}^{(m)}(x' - y')$.
Consider $\phi\in C_{c}^{\infty}(B_{1}^{m})$ and let $r>0$. Define $\Phi_{r}(x):=\phi(x')\zeta_{r}(x'')$, where $\zeta_{r}\in C_{c}^{\infty}(\R^{n-m})$, $0\leq \zeta_{r}\leq 1$, $\zeta_{r}=1$ on $B_{r}^{n-m}$ and $\zeta_{r}=0$ on $\R^{n-m}\setminus B_{r+1}^{n-m}$. Since $(-\Delta)_{(n)}^s\gk{\Omega_{\infty}}(x,\cdot) = \delta_x(\cdot)$ in the distributional sense, we have

\begin{equation*}
    \int_{\rn}\gk{\Omega_{\infty}}(x,y)(-\Delta)^{s}_{(n)}\Phi_{r}(y)\,dy=\Phi_{r}(x), \quad\text{for all }x\in \Omega_{\infty}.
\end{equation*}
For a fixed $x=(x',x'')\in \Omega_{\infty}$, choosing $r>0$ such that  $|x''|<r$ we have $\Phi_{r}(x)=\phi(x')$. Further,
\begin{equation*}
\begin{split}
  |(-\Delta)_{(n)}^s\phi(x')-(-\Delta)_{(n)}^s\Phi_r(x)| &=\left|\mathcal{A}(n,s)\operatorname{P.V.}
\int_{\mathbb R^m}
\int_{\mathbb R^{n-m}}
\frac{-\phi(y')(1-\zeta_r(y''))
}{
\bigl(
|x'-y'|^2+|x''-y''|^2
\bigr)^{\frac{n+2s}{2}}
}
\,dy''\,dy'\right|\\  
&\leq \mathcal{A}(n,s)\int_{\mathbb R^m}|\phi(y')|\,dy'\int_{\mathbb R^{n-m}\setminus B_{r}^{n-m}}\frac{dy''}{|x''-y''|^{n+2s}
}.
\end{split}
\end{equation*}
For $\delta=r-|x''|$, we have 
\begin{equation*}
    \begin{split}
      |(-\Delta)_{(n)}^s\phi(x')-(-\Delta)_{(n)}^s\Phi_r(x)| &\leq \mathcal{A}(n,s) \|\phi\|_{L^{1}(B_{1}^{m})}|S^{n-m-1}|\int_{\delta}^{\infty}t^{-m-2s-1}\,dt\\
      &=  C(n,m,s) \|\phi\|_{L^{1}(B_{1}^{m})} (r-|x''|)^{-m-2s}.
    \end{split}
\end{equation*}
For a compact set $K\subset \Omega_{\infty}$, choose $M>0$ such that $|x''|<M$ for all $x\in K$. Therefore, for all $x\in K$
\begin{equation*}
    |(-\Delta)_{(n)}^s\phi(x')-(-\Delta)_{(n)}^s\Phi_r(x)| \leq C(n,m,s) \|\phi\|_{L^{1}(B_{1}^{m})} (r-M)^{-m-2s}\to 0 \quad \text{as} \quad  r\to \infty. 
\end{equation*}

Since $\Phi_{r} \in C_{c}^{\infty}(\mathbb{R}^{n})$, it follows from the calculation in \cite[Section 2.1, p. 8–9]{BucurValdinocibook} that $|(-\Delta)_{(n)}^s \Phi_{r}(x)| \leq C (\|\Phi_{r}\|_{L^{\infty}(\mathbb{R}^{n})} + \|D^{2} \Phi_{r}\|_{L^{\infty}(\mathbb{R}^{n})})$, where $D^{2}$ denotes the second derivative and $C > 0$ is a constant independent of $r$. Since $\zeta_{r}$ is a cutoff function, we have $|(-\Delta)_{(n)}^s \Phi_{r}(x)| \leq C (\|\phi\|_{L^{\infty}(B_{1}^{m})} + \|D^{2} \phi\|_{L^{\infty}(B_{1}^{m})})$. Moreover, since $\gk{\Omega_{\infty}}(x, \cdot) \in L^{1}(\Omega_{\infty})$, the dominated convergence theorem implies
\begin{equation*}
   \int_{\rn}\gk{\Omega_{\infty}}(x,y)(-\Delta)^{s}_{(n)}\Phi_{r}(y)\,dy\to \int_{\rn}\gk{\Omega_{\infty}}(x,y)(-\Delta)^{s}_{(n)}\phi(y')\,dy.
\end{equation*}
Thus by \cite[Lemma 3.1 and 3.2]{ChowdhuryRoy2017}
\begin{equation*}
    \phi(x')=\int_{\rn}\gk{\Omega_{\infty}}(x,y)(-\Delta)^{s}_{(n)}\phi(y')\,dy=\int_{\rn}\gk{\Omega_{\infty}}(x,y)(-\Delta)^{s}_{(m)}\phi(y')\,dy.
\end{equation*}
By Fubini-Tonelli theorem
\begin{equation*}
    \phi(x')=\int_{\R^{m}}\left(\int_{\R^{n-m}}\gk{\Omega_{\infty}}(x,(y',y''))\,dy''\right)(-\Delta)^{s}_{(m)}\phi(y')\,dy'=\int_{\R^{m}}G(x,y')(-\Delta)^{s}_{(m)}\phi(y')\,dy'.
\end{equation*}
This leads to the claim.

\textit{Step 3.}
Finally, we claim that $H_{B_{1}^{m}}(x',y')=\int_{\R^{n-m}}H_{\Omega_{\infty}}(x,(y',\xi))\,d\xi$.

By Step 1, $\int_{\mathbb{R}^{n-m}} H_{\Omega_{\infty}}(x, (y', \xi))\, d\xi \leq \Gamma_{s}^{(m)}(x' - y')$, and by Step 2, this function serves as a distributional solution of $(-\Delta)^{s}_{(m)} u = 0$.

 As $H_{\Omega_{\infty}}(x, \cdot)$ is $s$-harmonic (cf. the proof of Lemma \ref{lemma convergence of greens kernel}), by \cite[Lemma 3.2]{BogdabKulczyckiNowak2002} we have $\partial_{y_{i}} H_{\Omega_{\infty}}(x, y) \leq \frac{n}{r} H_{\Omega_{\infty}}(x, y)$ for $y \in \Omega_{\infty}$ and for any $r > 0$ such that $B_{r}(y) \subset \Omega_{\infty}$. This implies that for any $y' \in K \subset\subset B_{1}^{(m)}$ and $y'' \in \mathbb{R}^{n-m}$
\begin{equation*}
    |\partial_{y_{i}}H_{\Omega_{\infty}}(x,(y',y''))|\leq \frac{n}{d(K,\partial\Omega_{\infty})} H_{\Omega_{\infty}}(x,(y',y''))\quad \text{for all }1\leq i\leq m.
\end{equation*}
Similar estimate holds for $\partial^{\alpha}_yH_{\Omega_{\infty}}(x,y)$ for any $\alpha \in(\mathbb{N}\cup\{0\})^n$. Therefore the map $y' \mapsto \int_{\mathbb{R}^{n-m}} H_{\Omega_{\infty}}(x, (y', \xi))\, d\xi$ is also smooth. 

The equality asserted in the claim is the result of uniqueness.
\end{proof}
This leads to the dimension reduction in dual weak formulation.

\begin{theorem}\label{them-dim-reduction-for-sol}
       Let $m > 2s$ and let $f : B_{1}^{m}(0) \to \mathbb{R}$ be a non-negative function satisfying \ref{F1}. Suppose that $v$ is a solution to \eqref{eqn main linear omega} with data $(f, 0, 0)$. Then the function $u(x) = v(x')$ for all $x \in \Omega_{\infty}$ is given by
    \begin{equation*}
        u(x)=\int_{\Omega_{\infty}}f(y')\gk{\Omega_{\infty}}(x,y)\,dy.
    \end{equation*}
    Further, for any $\psi\in L_{c}^{\infty}(\Omega_{\infty})$, the function $u$ satisfies \eqref{weak dual dfn with f} with data $(f(x'), 0, 0)$.
\end{theorem}
\begin{proof}
    For any $x\in \Omega_{\infty}$, by Theorem \ref{thrm-dim-reduction-for-Greens} and Tonelli one can write
    \begin{equation*}
        u(x)=v(x')=\int_{B^{(m)}}f(y')\gk{B_1^{(m)}}(x',y')\,dy'= \int_{\Omega_{\infty}}f(y')\gk{\Omega_{\infty}}(x,y)\,dy.
    \end{equation*}
    Further, for any $\psi\in L_{c}^{\infty}(\Omega_{\infty})$ (see Remark \ref{remark-distrubutional-weak-dual-solution})

    \begin{equation*}
    \int_{\Omega_{\infty}}u(x)\psi(x)\,dx = \int_{\text{supp}(\psi)} \psi(x)\int_{\Omega_{\infty}}f(y')\gk{\Omega_{\infty}}(x,y)\,dy\,dx=\int_{\Omega_{\infty}}f(y')\g{\Omega_{\infty}}(\psi)(y)\,dy. \qedhere
    \end{equation*}
\end{proof}

\subsection{On the Infinite Cylinder}\label{subsection_3.2}

We examine solutions $\{\ul\}_{\ell}$ of \eqref{solution on omega ell} on a sequence of finite cylinders. For each $x\in \Omega_{\infty}$ we demonstrate that the sequence $\{\ul(x)\}_{\ell}$ is increasing and bounded, and that $\sup_{\ell}\ul(x)$ is the solution on the infinite cylinder $\Omega_{\infty}$. 
Analogous to the bounded domain case, we consider three types of data
\begin{itemize}
    \item $(\delta_{\Omega_{\infty}}^{-\beta},0,0)$.
    \item $(0,g,0)$ such that $g$ blows up on the boundary of $\Omega_{\infty}$.
    \item $(0,0,h)$ with $h\in  L^{1}(\partial\Omega_{\infty})$ and $h>0$.
\end{itemize}
For each such data, we show that the corresponding solutions $\ul$ with any data $(f_{\ell},0,0),(0,g_{\ell},0)$ or $(0,0,h_{\ell})$ where $f_{\ell},g_{\ell}$ and $h_{\ell}$ are defined in \eqref{dfn of g ell and h ell}, converges to a large solution on the infinite cylinder.
\begin{lemma}\label{comparing u ell}
     Assume \ref{F1}, \ref{G} and \ref{H} on $\Omega=\Omega_{\infty}$. We further assume that $f,g\geq 0$. Let $f_{\ell},g_{\ell}$ and $h_{\ell}$ be as in \eqref{dfn of g ell and h ell} and $\ell_{1}<\ell_{2}$. If $u_{i}\in L^{1}_{loc}(\Omega_{\ell_{i}})$ are the solution of \eqref{solution on omega ell} with the data $(f_{\ell_{i}},g_{\ell_{i}},h_{\ell_{i}})$, for $i=1,2$ respectively. Then $u_{1}\leq u_{2}$ a.e. in $\Omega_{\ell_{1}}$.
\end{lemma}
\begin{proof}
    Let $\psi\in L_{c}^{\infty}(\Omega_{\ell_{1}})$ and $\psi\geq 0$, as $\Omega_{\ell_{1}}\subset\Omega_{\ell_{2}}$, by the definition of the weak dual solution and Lemma \ref{lemma restriction of a solution to subdomains}
    \begin{equation*}
    \begin{split}
    \int_{\Omega_{\ell_{1}}}u_{2}\psi &= \int_{\Omega_{\ell_{2}}}u_{2}\psi
    =\Big\{\int_{\Omega_{\ell_{1}}}+\int_{\Omega_{\ell_{2}}\setminus\Omega_{\ell_{1}}}\Big\} \, \g{\Omega_{\ell_{2}}} (\psi)(x)f_{\ell_{2}}(x)\ dx -\int_{\R^{n}\setminus\overline{\Omega}_{\infty}} \slap \g{\Omega_{\ell_{2}}}(\psi)(x) g(x) \ dx\\
&\hspace{3cm}+\Big\{\int_{\partial\Omega_{\ell_{1}}\cap\partial\Omega_{\infty}}+\int_{(\partial\Omega_{\ell_{2}}\setminus\partial\Omega_{\ell_{1}})\cap \partial\Omega_{\infty}}\Big\} \Ds[\g{\Omega_{\ell_{2}}}(\psi)](\theta)h(\theta)\ d\mathcal{H}(\theta)\\
       &\geq  \int_{\Omega_{\ell_{1}}}\g{\Omega_{\ell_{1}}}(\psi)(x) f_{\ell_{1}}(x)\ dx-\int_{\R^{n}\setminus\overline{\Omega}_{\ell_{1}}}\slap \g{\Omega_{\ell_{1}}}(\psi)(x) g_{\ell_{1}}(x) \ dx+\int_{\partial\Omega_{\ell_{1}}}\Ds[\g{\Omega_{\ell_{1}}}(\psi)](\theta)h_{\ell_{1}}(\theta)\ d\mathcal{H}(\theta)\\
       &=\int_{\Omega_{\ell_{1}}}u_{1}\psi.
        \end{split}
    \end{equation*}
    Now by choosing $\psi=\chi_{K}$ for arbitrary compact set $K$ in $\Omega_{\ell_{1}}$, it follows that $u_{1}\leq u_{2}$ a.e. in $\Omega_{\ell_{1}}$.
\end{proof}

\subsubsection*{\textbf{Blow up by the operator}}\label{subsection_by operator}\label{subsubsection_3.2.1}
 Here,  we focus on the large s-harmonic functions generated by the fractional Laplacian. These solutions are characterised by the operator $E_{\Omega}$. Accordingly, we restrict to the data to $(0,0,h)$. 

Let $u_{\ell}^{h}$ denote the solution to \eqref{solution on omega ell} with the data $(0,0,h_{\ell})$. Here, $h_{\ell}$ is defined in \eqref{dfn of g ell and h ell}. For any $x\in \Omega_{\infty}$, choose $\ell_{0}=\ell_{0}(x)>0$ such that $x\in \Omega_{\ell_{0}}$. Then by Lemma \ref{comparing u ell} the sequence $\{u^{h}_{\ell}(x)\}_{\ell\geq \ell_{0}}$ is non-decreasing.  Next, we prove that the sequence is locally uniformly bounded.
\begin{lemma}
    Assume \ref{H} on $\Omega=\Omega_{\infty}$. Let $K\subset\subset\Omega_{\infty}$. Then $|u_{\ell}^{h}(x)|\leq C(K,n,s)\|h\|_{L^{1}(\partial\Omega_{\infty})}$ for some positive constant $C(K,n,s)$ and for all $x\in K$. 
\end{lemma}
\begin{proof}
   For $x\in K$ we have by \eqref{eqn greens estimatte on unbounded domain} and \eqref{eqn fractional normal derivative behaviour}
   \begin{equation*}
   \begin{split}
      u^{h}_{\ell}(x)&\leq \int_{\partial\Omega_{\infty}\cap \partial\Omega_{\ell}}h(\theta) \Ds \gk{\Omega_{\ell}}(x,\theta)\ d\mathcal{H}(\theta) \\
      & \leq \int_{\partial\Omega_{\infty}\cap\partial\Omega_{\ell}}h(\theta)\lim_{\substack{y\to\theta \\ y\in\Omega_{\ell}}}\frac{\gk{\Omega_{\ell}}(x,y)}{\delta_{\ell}^{s}(y)}\ d\mathcal{H}(\theta)
      \leq \int_{\partial\Omega_{\infty}\cap\partial\Omega_{\ell}}h(\theta)\lim_{\substack{y\to\theta \\ y\in\Omega_{\ell}}}\frac{|x-a_y|^{s}}{|x-y|^{n-s}}\ d\mathcal{H}(\theta).
   \end{split}
 \end{equation*}
 Since $y$ close to the boundary and away from $K$, using \eqref{eqn triangle inequality} (notice that $\delta_{\infty}$ is bounded) in Theorem \ref{theorem greens poteantial continuity on unbounded domains}, we get 
 \begin{equation*}
     u^{h}_{\ell}(x) \leq \int_{\partial\Omega_{\infty}\cap\partial\Omega_{\ell}}\frac{h(\theta)}{|x-\theta|^{n-2s}}\ d\mathcal{H}(\theta)\leq \frac{\|h\|_{L^{1}(\partial\Omega_{\infty})}}{d(K,\partial\Omega_{\infty})^{n-2s}}.
     \qedhere
 \end{equation*}
\end{proof}
Define
\begin{equation}\label{dfn of uh in linear case}
    u^{h}(x)=\lim_{\ell\to\infty}\ul^{h}(x) \quad \text{for all }x\in \Omega_{\infty}.
\end{equation}
Then, we have the following
 \begin{theorem}\label{thrm for operator blowup}
    Assume \ref{H} on $\Omega=\Omega_{\infty}$. The function $u^{h}$ defined in \eqref{dfn of uh in linear case} is in $L^{1}_{loc}(\Omega_{\infty})$. Further, it is the solution of \eqref{linear eqn on infinite cylinder} with data $(0,0,h)$.
\end{theorem}
\begin{proof}
 Due to the representation, $\ul^{h}\geq0$ for all $\ell>0$ as $h\geq 0$ and hence $u^{h}\geq 0$. Let $K\subset \Omega_{\infty}$ be compact, choose $\ell_{0}$ such that $K\subset\subset \Omega_{\ell_{0}}$ and $d(K,\partial\Omega_{\ell_{0}})=d(K,\partial\Omega_{\infty})$, then for any $x\in K$ and $\ell\geq \ell_{0}$, by the previous theorem 
    \begin{equation*}
    \begin{split}
        \int_{K}\ul^{h}(x)\ dx&= \int_{K}\int_{\partial\Omega_{\ell}}\Ds\gk{\Omega_{\ell}}(\theta,x) h_{\ell}(\theta)\ d\mathcal{H}(\theta) \leq \frac{\|h\|_{L^{1}(\partial\Omega_{\infty})}|K|}{d(K,\partial\Omega_{\infty})^{n-2s}}.
        \end{split}
    \end{equation*}

        Applying the monotone convergence theorem, we have
        \begin{equation*}
            \int_{K}u^{h}(x)\ dx\leq \frac{\|h\|_{L^{1}(\partial\Omega_{\infty})}|K|}{d(K,\partial\Omega_{\infty})^{n-2s}}.
        \end{equation*}
Hence $u^{h}\in L^{1}_{loc}(\Omega_{\infty})$.

         For $\psi\in L^{\infty}_{c}(\Omega_{\infty})$ and $\ell_{0}>0$ such that Supp$(\psi)\subset \Omega_{\ell}$ for all $\ell\geq \ell_{0}$. Using the compact support of $\psi$, the dominated convergence theorem, 
         we obtain
\begin{equation*}
\begin{split}
    \int_{\Omega_{\infty}}u^{h}(x)\psi(x)\ dx=\lim_{\ell\to\infty}\int_{\Omega_{\ell}}\ul^{h}(x) \psi(x) 
=\lim_{\ell\to\infty}\int_{\partial\Omega_{\infty}} \chi_{\partial\Omega_{\ell}} \Ds [\g{\Omega_{\ell}}(\psi)](\theta)  h(\theta)  d\mathcal{H}(\theta).
    \end{split}
\end{equation*}
To use the dominated convergence theorem again, consider for $\ell$ large, choose a compact set $K\subset\Omega_{\ell_{0}}$ such that Supp$(\psi)\subset\subset K$, then by 
\eqref{eqn well definedness of frac normal derivative on unbounded domain} and \eqref{eqn triangle inequality}
\begin{equation*}
\begin{split}
     |\Ds [\g{\Omega_{\ell}}(\psi)](\theta)| h(\theta)
    &\leq c(n,s)\lim_{x\to\theta}\int_{\text{supp}(\psi)}\frac{|y-a_x|^{s}}{|x-y|^{n-s}}|\psi(y)|\ dy h(\theta) 
    \leq \lim_{ x\to\theta}\frac{C(K,n,s)}{d(K^{C},\text{supp}(\psi))^{n-2s}}||\psi||_{L^{1}(\Omega_{\infty})} h(\theta).
    \end{split}
\end{equation*}
The last term is in $L^{1}(\partial\Omega_{\infty})$ as $h\in L^{1}(\partial\Omega_{\infty})$.  Thus, by Lemma \ref{lemma-convergence-in-fractional-normal-derivative}
\begin{equation*}
   \int_{\Omega_{\infty}}u^{h}(x) \psi(x)\ dx =\int_{\partial\Omega_{\infty}}\Ds[\g{\Omega_{\infty}}(\psi)](\theta)h(\theta)\ d\mathcal{H}(\theta).\qedhere
\end{equation*}
\end{proof}

\subsubsection*{\textbf{Blow up by the force function}}\label{subsection_by force term}\label{subsubsection_3.2.2}
We construct a large solution on the infinite cylinder $\Omega_{\infty}$using data $(\delta_{\Omega_{\infty}}^{-\beta},0,0)$, analogous to Proposition \ref{prop boundary behaviour of the solution on bounded domain} on bounded domains. For completeness, we also demonstrate that for any data $(f,0,0)$, the corresponding solutions in the finite cylinders with data $(f_{\ell},0,0)$, where $f_{\ell}=f|_{\Omega_{\ell}}$, converge to a solution in the infinite cylinder. 

First, we establish the result for $f\geq 0$ satisfying \ref{F1} on $\Omega=\Omega_{\infty}$. Note that the corresponding solution with data $(f,0,0)$ may not necessarily blow up. Let $u^{f}_{\ell}$ denote the solution to \eqref{solution on omega ell} with data $(f_{\ell},0,0)$, where $f_{\ell}=f|_{\Omega_{\ell}}$. For any $x\in \Omega_{\infty}$, choose $\ell_{0}=\ell_{0}(x)>0$ such that $x\in \Omega_{\ell_{0}}$. Then by Lemma \ref{comparing u ell} the sequence $\{u^{f}_{\ell}(x)\}_{\ell\geq \ell_{0}}$ is increasing. We further show that the sequence is also locally uniformly bounded.
\begin{lemma}\label{lemma-local-bound-of-u-ell-f}
  Assume \ref{F1} for $\Omega=\Omega_{\infty}$. Let $K\subset\subset\Omega_{\infty}$, then for all $x\in K$, $|u_{\ell}^{f}(x)|\leq C(K,n,s)$.
\end{lemma}
\begin{proof}
    Let $K_{1}$ be a set such that $K\subset\subset K_{1}\subset\subset\Omega_{\infty}$, then, by \eqref{eqn greens estimatte on unbounded domain} and since $\delta_{\Omega_{\ell}}\leq \delta_{\Omega_{\infty}}$, we have
    \begin{equation*}
        |u^{f}_{\ell}(x)|\leq \int_{K_{1}}|f(y)|\frac{c(n,s)}{|x-y|^{n-2s}}\ dy+\int_{\Omega_{\infty}\setminus K_{1}}|f(y)|\delta_{\Omega_{\infty}}^{s}(y)\frac{|x-a_y|^{s}}{|x-y|^{n-s}} \ dy.
 \end{equation*}
 By the H\"older's inequality, and choosing $R>0$ such that for all $x\in K$, $K_{1}\subset B_{R}(x)$, and also since $q>n/2s$, we obtain
 \begin{equation*}
     \int_{K_{1}}|f(y)|\frac{c(n,s)}{|x-y|^{n-2s}}\ dy\leq \|f\|_{L^{q}(K_{1})}\int_{B_{R}(x)}\frac{dy}{|x-y|^{q'(n-2s)}}\leq C(K,n,s)\|f\|_{L^{q}(K_{1})},
 \end{equation*}
 and by the triangle inequality as in \eqref{eqn triangle inequality}, we get
 \begin{equation*}
     \int_{\Omega_{\infty}\setminus K_{1}}|f(y)|\delta_{\Omega_{\infty}}^{s}(y)\frac{|x-a_y|^{s}}{|x-y|^{n-s}} \ dy\leq \frac{C(K,n,s)}{d(K,\R^{n}\setminus K_{1})^{n-s}}\|f\|_{L^{1}(\Omega_{\infty},\delta_{\Omega_{\infty}}^{s})}.
 \end{equation*}
 By combining the above three equations, the result follows.
\end{proof}

We define
\begin{equation}\label{dfn of uf in linear case}
    u^{f}(x)=\lim_{\ell\to\infty}\ul^{f}(x) \quad \text{for all }x\in \Omega_{\infty}.
\end{equation}

\begin{lemma}
   Assume \ref{F1} on $\Omega=\Omega_{\infty}$ with $f\geq 0$, the limit $u^{f}$ defined above is in $L^{1}_{loc}(\Omega_{\infty})$.
\end{lemma}
\begin{proof}
    Due to the representation of the solution through the Green's kernel, $\ul^{f}\geq0$ for all $\ell>0$ and hence $u^{f}\geq 0$. Let $K\subset \Omega_{\infty}$ be compact, choose $\ell_{0}$ such that $K\subset \Omega_{\ell_{0}}$ and $\delta_{\Omega_{\ell_{0}}}=\delta_{\Omega_{\infty}}$ on $K$, then for any $\ell\geq \ell_{0}$ 
    \begin{equation*}
        \int_{K}|\ul^{f}(x)|\ dx=\int_{K}\ul^{f}(x)\ dx \leq \int_{K} \int_{\Omega_{\ell}}\gk{\Omega_{\ell}}(x,y) f_{\ell}(y)\ dy dx .
    \end{equation*}n
        Consider $\rho=\rho(K)>0$ such that $K_{\rho}:=\{x\in \Omega_{\ell_{0}} ~|~ 0\leq d(x,K)\leq \rho\} \subset \Omega_{\ell_{0}}$. Rewriting the above integral , we obtain 
        \begin{equation*}
            \begin{split}
                 \int_{K} \int_{\Omega_{\ell}}\gk{\Omega_{\ell}}(x,y) f_{\ell}(y) \ dydx 
                &=\int_{K}\left\{\int_{K_{\rho}}+\int_{\Omega_{\ell}\setminus K_{\rho}}\right\}\gk{\Omega_{\ell}}(x,y) f_{\ell}(y) \ dydx\\
            \end{split}
        \end{equation*}
        First using $\gk{\Omega_{\ell}(x,y)}\leq \Gamma_{s}(x-y)$ in the first integral; then multiplying and dividing by $\delta_{\Omega_{\infty}}^{s}(y)$ to obtain
        \begin{equation*}
            \begin{split}
                \int_{K}\int_{K_{\rho}}\gk{\Omega_{\ell}}(x,y) f_{\ell}(y)\ dxdy &\leq \frac{C(n,s, K)}{d(K_{\rho},\partial\Omega_{\infty})^{s}}\int_{K_{\rho}}f(y)\delta_{\Omega_{\infty}}^{s}(y)\int_{K}\frac{1}{|x-y|^{n-2s}}\ dx dy\\
                &\leq C(n,s,K)\int_{K_{\rho}}f(y)\delta_{\Omega_{\infty}}^{s}(y)\ dy.
            \end{split}
        \end{equation*}
        Using \eqref{eqn greens estimatte on unbounded domain} and \eqref{eqn triangle inequality} in the second integral to obtain
        \begin{equation*}
            \begin{split}
              \int_{K}\int_{\Omega\setminus K_{\rho}}\gk{\Omega_{\ell}}(x,y) f_{\ell}(y)\ dxdy &\leq \int_{K}\int_{\Omega\setminus K_{\rho}} |x-a_y|^{s}\frac{\delta_{\Omega_{\ell}}^{s}(y)f(y)}{|x-y|^{n-s}}\ dxdy\leq \frac{C(n,s,K)}{\rho^{n-2s}}\int_{\Omega\setminus K_{\rho}}f(y)\delta_{\Omega_{\infty}}^{s}(y)\ dy
            \end{split}
        \end{equation*}
       Thus,
       \begin{equation*}
           \int_{K}|\ul^{f}(x)|\ dx\leq C(n,s,K)\|f\|_{L^{1}(\Omega_{\infty},\delta^{s}_{\Omega_{\infty}})}
       \end{equation*}
         By the monotone convergence theorem, we deduce 
        \begin{equation*}
            \int_{K}|u^{f}(x)|\ dx \leq C(n,s,K)||f||_{L^{1}(\Omega_{\infty},\delta^{s}_{\Omega_{\infty}})}.\qedhere
        \end{equation*}
\end{proof}

\begin{theorem}
    Assume \ref{F1} on $\Omega=\Omega_{\infty}$ with $f\geq 0$, then $u^{f}$ defined in \eqref{dfn of uf in linear case} solves \eqref{linear eqn on infinite cylinder} with data $(f,0,0)$.
\end{theorem}
\begin{proof}
    For $\psi\in L^{\infty}_{c}(\Omega_{\infty})$ and $\ell_{0}>0$ such that Supp$(\psi)\subset \Omega_{\ell_{0}}$, by the dominated convergence theorem and the compact support of $\psi$, we have
\begin{equation*}
\begin{split}
    \int_{\Omega_{\infty}}u^{f}\psi=\lim_{\ell\to\infty}\int_{\Omega_{\ell}}\ul^{f} \psi=\lim_{\ell\to\infty}\int_{\Omega_{\infty}}\g{\Omega_{\ell}}(\psi)(x)f_{\ell}(x)\chi_{\Omega_{\ell}}\ dx .
    \end{split}
\end{equation*}
To use the dominated convergence theorem again, consider
\begin{equation*}
        |\g{\Omega_{\ell}}(\psi)(x) f_{\ell}(x)\chi_{\Omega_{\ell}}(x)|\leq f(x)\int_{\text{Supp}(\psi)}\gk{\Omega_{\ell}}(x,y) |\psi(y)|\ dy .
   \end{equation*}
   Choose a $ K$ such that Supp$(\psi)\subset\subset K\subset\subset \Omega_{\infty}$. For $x\in K$, since $\delta_{\Omega_{\infty}}(x)\geq d(K,\partial\Omega_{\infty})$, using $\gk{\Omega_{\ell}(x,y)}\leq \Gamma_{s}(x-y)$, we infer
   \begin{equation*}
       \int_{\text{Supp}(\psi)}\gk{\Omega_{\ell}}(x,y)\ |\psi(y)| dy \leq a(n,s)\|\psi\|_{L^{\infty}(\Omega_{\infty})}\frac{\delta_{\Omega_{\infty}}^{s}(x)}{\delta_{\Omega_{\infty}}^{s}(x)}\int_{\text{Supp}(\psi)}\frac{dy}{|x-y|^{n-2s}}\leq C(n,s,\psi)||\psi||_{L^{\infty}(\Omega_{\infty})} \delta_{\Omega_{\infty}}^{s}(x).
   \end{equation*}
   For $x\notin K$, using \eqref{eqn greens estimatte on unbounded domain}, we have
   \begin{equation*}
       \int_{\text{Supp}(\psi)}\gk{\Omega_{\ell}}(x,y)\ |\psi(y)| dy \leq \delta_{\Omega_{\ell}}^{s}(x)||\psi||_{L^{\infty}}\int_{\text{Supp}(\psi)}\frac{1}{|x-y|^{n-2s}}\ dy \leq C(n,s,\psi)||\psi||_{L^{\infty}}\delta_{\Omega_{\infty}}^{s}(x).
   \end{equation*}
   Thus, 
   \begin{equation*}
        |\g{\Omega_{\ell}}(\psi)(x) f_{\ell}(x)\chi_{\Omega_{\ell}}(x)|\leq C(n,s,\psi) f(x)\delta_{\Omega_{\infty}}^{s}(x).
   \end{equation*}
   Here, the right hand side is an $L^{1}$ function as $f\in L^{1}(\Omega_{\infty},\delta^{s})$.  Thus, by Lemma \ref{Lemma local uniform convergence of solution operator}
\begin{equation*}
   \int_{\Omega_{\infty}}u^{f}(x) \psi(x)\ dx =\int_{\Omega_{\infty}} \g{\Omega_{\infty}}(\psi)(x) f(x)\ dx.\qedhere
\end{equation*}
\end{proof}

Next we consider the case $f(x)=\delta_{\Omega_{\infty}}^{-\beta}=\delta_{B^{m}}^{-\beta}(x')$. As $f(x)=f(x')$, we consult to \textit{Dimension reduction} subsection.
\begin{theorem}
    Let $m>2s$ and $2s < \beta < 1 + s$. Suppose $u_{\ell}$ is the solution to \eqref{solution on omega ell} with data $(f_{\ell} = \delta_{\Omega_{\infty}}^{-\beta}|_{\Omega_{\ell}}, 0, 0)$. Then $u_{\ell}$ converges point wise to the solution $v$ of \eqref{linear eqn on infinite cylinder} on the cross-section $B^{m}$ with data $(\delta_{B^{m}}^{-\beta}, 0, 0)$. The function $u(x) = v(x')$ satisfies \eqref{weak dual dfn with f} with data $(\delta_{\Omega_{\infty}}^{-\beta}, 0, 0)$. Moreover,
    \begin{equation*}
        c_{1}\delta_{\Omega_{\infty}}^{-\beta+2s}(x)\leq u(x)\leq c_{2}\delta_{\Omega_{\infty}}^{-\beta+2s}(x),
    \end{equation*}
for some positive constants $c_{1}$ and $c_{2}$.
\end{theorem}
\begin{proof}
  For $x\in \Omega_{\infty}$, by Lemma \ref{comparing u ell}, monotone convergence theorem, and Theorem \ref{them-dim-reduction-for-sol} we have
  \begin{equation*}
      u^{f}(x)=\lim_{\ell\to +\infty}\int_{\Omega_{\ell}}f_{\ell}(y)\gk{\Omega_{\ell}}(x,y)\,dy=\int_{\Omega_{\infty}}\delta_{\Omega_{\infty}}^{-\beta}(y)\gk{\Omega_{\infty}}(x,y)\,dy=v(x').
  \end{equation*}
  Thus $u$ solves \eqref{linear eqn on infinite cylinder} with $(\delta_{\Omega_{\infty}}^{-\beta},0,0)$. 

  By Proposition \ref{prop boundary behaviour of the solution on bounded domain} applied to $v$, one can conclude the boundary behaviour of $u$.
\end{proof}

\subsubsection*{\textbf{Blow up by the exterior data}}\label{subsection_by exterior data}\label{subsubsection_3.2.3}

We construct large solutions on the infinite cylinder corresponding to the data $(0,g,0)$, where $g$ exhibits blow-up behaviour at the boundary. Prior to this, we establish that when $g$ blows up on an open connected subset of the boundary of a bounded domain $\Omega$, then the corresponding solution also blows up on the same subset. Accordingly we assume 

\begin{enumerate}
    \myitem{$\textbf{(G2)}$}\label{G2}  \textit{For an open connected subset $T$ of $\partial\Omega$
    \begin{equation*}\label{bdy blowup assumption on g}
    \lim_{\substack{x \to x_{0} \\ x\in\R^{n}\setminus\Omega}} g(x)=+\infty, \quad \text{for all } x_{0} \in T.
\end{equation*}}
\end{enumerate}

When $T=\partial\Omega$, the solution to \eqref{eqn main linear omega} with data $(0,g,0)$ belongs to $C(\overline{\Omega})$ as established in \cite[Theorem 1.2]{Abatangelo2015LargesHarmonic}. The argument given is local in nature, which yields the following:

\begin{theorem}\label{thrm continuity up to the boundary}
    Let $\Omega$ be a bounded domain, $g$ satisfies \ref{G1}. Assume that there exists an open set $\mathcal{O}\subset\R^{n}$ such that $\mathcal{O}\cap \partial\Omega=T$ and $g \in C(\mathcal{O}\cap(\R^{n}\setminus\Omega))$. If $u$ is a solution of \eqref{eqn main linear omega} with the data $(0,g,0)$, then $u\in C(\Omega\cup T)$.
\end{theorem}
\begin{proof}
    We show that as $x\to x_{0}\in T$, $u(x)\to g(x_{0})$, where $u$ is given in \eqref{representation with g}.
    Let $x\in \Omega$ and $x_{0}\in T$, given $\epsilon>0$, then there is a $\rho>0$ such that $B_{\rho}(x_{0})\subset \mathcal{O}$ and for all $y\in B_{\rho}(x_{0})\cap (\R^{n}\setminus \Omega)$ $|g(y)-g(x_{0})|<\epsilon/2$. Using the fact that the integral of the Poisson kernel is $1$, we can write
    \begin{equation*}
            |u(x)-g(x_{0})|\leq \int_{(\R^{n}\setminus\overline{\Omega}) \cap B_{\rho}(x_{0})} |g(y)-g(x_{0})| |\slap \gk{\Omega}(x,y)|\ dy+\int_{(\R^{n}\setminus\overline{\Omega})\setminus B_{\rho}(x_{0})} |g(y)-g(x_{0})| |\slap \gk{\Omega}(x,y)|\ dy.
    \end{equation*}
    The first integral is bounded by $\epsilon$ as
    \begin{equation*}
        \int_{(\R^{n}\setminus\overline{\Omega}) \cap B_{\rho}(x_{0})} |g(y)-g(x_{0})| |\slap \gk{\Omega}(x,y)|\ dy\leq \epsilon \int_{\R^{n}\setminus\overline{\Omega}}-\slap \gk{\Omega}(x,y)\ dy=\epsilon.
    \end{equation*}
    The above inequality holds for any $x\in \Omega$. Before dealing with the second integral choose $x\in B_{\rho}(x_{0})\cap \Omega$. Using \eqref{asym of poison on bounded domain} we write
    \begin{equation*}
        \int_{(\R^{n}\setminus\overline{\Omega})\setminus B_{\rho}(x_{0})} |g(y)-g(x_{0})| |\slap \gk{\Omega}(x,y)|\ dy\leq C\delta_{\Omega}^{s}(x)\int_{(\R^{n}\setminus\overline{\Omega})\setminus B_{\rho}(x_{0})}\frac{|g(y)-g(x_{0})|}{\delta_{\Omega}^{s}(y)(1+\delta_{\Omega}(y))^{s}|x-y|^{n}}\ dy \leq C\delta_{\Omega}^{s}(x).
    \end{equation*}
    Indeed, up to some constant we have $\delta_{\Omega}(y)\leq C |x-y|$, which implies
    \begin{equation*}
        \int_{(\R^{n}\setminus\overline{\Omega})\setminus B_{\rho}(x_{0})}\frac{|g(y)|}{\delta_{\Omega}^{s}(y)(1+\delta_{\Omega}(y))^{s}|x-y|^{n}}\ dy \leq \int_{(\R^{n}\setminus\overline{\Omega})\setminus B_{\rho}(x_{0})}|g(y)|\min\{\delta_{\Omega}^{-s},\delta_{\Omega}^{-n-2s}(y)\}\ dy <\infty,
    \end{equation*}
    and 
    \begin{equation*}
        \int_{(\R^{n}\setminus\overline{\Omega})\setminus B_{\rho}(x_{0})}\frac{|g(x_{0})|}{\delta_{\Omega}^{s}(y)(1+\delta_{\Omega}(y))^{s}|x-y|^{n}}\ dy\leq |g(x_{0})|\int_{(\R^{n}\setminus\overline{\Omega})\setminus B_{\rho}(x_{0})}\min\{\delta_{\Omega}^{-s},\delta_{\Omega}^{-n-2s}(y)\ dy<\infty.
    \end{equation*}
    Now we choose $r>0$ such that for all $x\in B_{r}(x_{0})\cap \Omega$, $\delta_{\Omega}^{s}(x)\leq \epsilon/2C$. Then choosing $\min\{\rho,r\}$ we have
    \begin{equation*}
        |u(x)-g(x_{0})|<\epsilon.
    \end{equation*}
    The conclusion follows as $\epsilon>0$ is arbitrary. 
\end{proof}

\begin{theorem}\label{thrm blow up of solution where g is}
    Let $\Omega$ be a bounded domain. Assume  \ref{G1} and \ref{G2}. 
    If $u$ is the solution of \eqref{eqn main linear omega} with the data $(0,g,0)$, then $u(x)\to \infty$ as $x\to x_{0}$, for all $x_{0}\in T$.
\end{theorem}
\begin{proof}
    Let $k$ be an integer and $u_{k}$ be the solution of \eqref{eqn main linear omega} with $(0,g_{k},0)$ where $g_{k}:=\min\{k,g\}$. By the above theorem, $u_{k}\in C(\Omega\cup T)$. Thus $u_{k}(x)=k$ for all $x\in T$. By Lemma \ref{lemma comparison principle}, $\{u_{k}(x)\}_{k}$ is an increasing sequence and $u_{k}(x)\leq u(x)$, for all $x\in \Omega$. By the monotone convergence theorem 
    \begin{equation*}
        \lim_{k\to \infty}u_{k}(x)=-\int_{\R^{n}\setminus\overline{\Omega}}\slap\gk{\Omega}(x,y) g(y)\ dy= u(x).
    \end{equation*}
    Further, for any $\theta\in \partial\Omega$, one has
    \begin{equation*}
        \liminf_{\substack {x\to \theta\\ x\in \Omega}} u_{k}(x)\leq \liminf_{\substack{x\to \theta\\ x\in \Omega}} u(x).
    \end{equation*}
    In particular, for $x_{0}\in T$, we get $u(x)\to \infty$ as $x\to x_{0}$.
\end{proof}
We turn to the construction of large solutions on an infinite cylinder generated by exterior data.
\begin{theorem}
    Let $g$ satisfies \ref{G} and \ref{G2} on $\Omega=\Omega_{\infty}$ and $T=\partial\Omega_{\infty}$. If $u$ is the solution of \eqref{linear eqn on infinite cylinder} with $(0,g,0)$, then $u(x)\to \infty$ as $x\to \theta\in \partial\Omega_{\infty}$. 
\end{theorem}
\begin{proof}
    Let $u_{\ell}$ be the solution of \eqref{solution on omega ell} with $(0,g_{\ell},0)$. Since $g_{\ell}(y)\to \infty$ as $y\to \theta\in \partial\Omega_{\ell}\cap \partial\Omega_{\infty}$ by Theorem \ref{thrm blow up of solution where g is}, $u_{\ell}(x)\to \infty$ as $x\to \theta\in \partial\Omega_{\ell}\cap \partial\Omega_{\infty}$. 
    Further, for any  $\psi
    \in L_{c}^{\infty} (\Omega_{\ell})$, one can write
    \begin{equation*}
        \begin{split}
    \int_{\Omega_{\ell}}u(x)\psi(x)\ dx=\int_{\Omega_{\infty}}u(x)\psi(x)\ dx&=- \! \int_{\rn\setminus\overline{\Omega}_{\infty}}g(y)\slap \g{\Omega_{\infty}}(\psi)(y)\ dy 
    \geq - \! \int_{\rn\setminus\overline{\Omega}_{\infty}}g(y)\slap \g{\Omega_{\ell}}(\psi)(y)\ dy.
    \end{split}
    \end{equation*}
    Since $g_{\ell}=0$ on $\Omega_{\infty}\setminus\Omega_{\ell}$, we infer that $u$ is a super solution of the equation \eqref{solution on omega ell}. By Lemma \ref{lemma comparison principle}, $\ul\leq u$ on $\Omega_{\ell}$, for all $\ell>0$. 

    Given $\theta\in \partial\Omega_{\infty}$, choose $\ell>0$ such that $\theta\in \partial\Omega_{\ell}$, then 
    \begin{equation*}
        \lim_{\substack{x\to\theta\\ x\in \Omega_{\ell}}}\ul(x)\leq \liminf_{\substack{x\to\theta\\ x\in \Omega_{\ell}}}u(x).
    \end{equation*}
    Hence the result follows.
\end{proof}

We now consider $g$ satisfying \ref{G1} on $\Omega=\Omega_{\infty}$. Let $\ul^{g}$ denote the solution of \eqref{solution on omega ell} with the data $(0,g,0)$. For any $x\in \Omega_{\infty}$, choose $\ell_{0}=\ell_{0}(x)>0$ such that $x\in \Omega_{\ell_{0}}$. By Lemma \ref{lemma comparison principle} the sequence $\{u^{g}_{\ell}(x)\}_{\ell\geq \ell_{0}}$ is increasing.  We further show that this sequence is locally uniformly bounded.
\begin{lemma}
    Assume \ref{G1} on $\Omega=\Omega_{\infty}$. Let $\ul^{g}$ be the solution of \eqref{solution on omega ell} with the data $(0,g,0)$. For any compact set $K\subset \Omega_{\infty}$, there exists $\ell_{0}$ such that for all $\ell\geq \ell_{0}$ one has 
    \begin{equation}\label{eqn l infty bound for solution with g ell}
        \|\ul^{g}\|_{L^{\infty}(K)}\leq C(n,s,K)\int_{\R^{n}\setminus\overline{\Omega}_{\infty}}|g(y)| \delta_{\Omega_{\infty}}^{-s}(y)\ dy
    \end{equation}
\end{lemma} 
\begin{proof}
    Choose $\ell_{0}>0$ such that $K\subset \Omega_{\ell_{0}}$ and also that $\delta_{\Omega_{\ell}}=\delta_{\Omega_{\infty}}$, then for any $x\in K$, using Theorem \ref{theorem poission behaviour}
\begin{equation*}
\begin{split}
   |\ul^{g}(x)|&= -\int_{\R^{n}\setminus\overline{\Omega}_{\ell}}\slap \gk{\Omega_{\ell}}(x,y) |g_{\ell}(y)|\ dy \leq C(K,n,s)\int_{\R^{n}\setminus\overline{\Omega}_{\ell}} \frac{|g_{\ell}(y)|}{\delta_{\Omega_{\ell}}^{s}(y)} \ dy= C(K,n,s)\int_{\R^{n}\setminus\overline{\Omega}_{\infty}}\frac{|g(y)|}{\delta_{\Omega_{\infty}}^{s}(y)}\ dy.\\
\end{split}
\end{equation*}
The last integral is finite in view of the assumption \ref{G}.
\end{proof}
Given the local uniform bound and the pointwise  monotonicity of the sequence $\{\ul^{g}\}$, we define
\begin{equation}\label{dfn of ug with g}
    u^{g}(x):=\lim_{\ell\to\infty} \ul^{g}(x) \quad \text{for all } x\in \rn.
\end{equation}
As a consequence of the local uniform bound, we obtain 
\begin{corollary}\label{prop v is l1 with g}
     Assume \ref{G} on $\Omega=\Omega_{\infty}$, the function $u^{g}$ is in $L^{1}_{loc}(\Omega_{\infty})$.
\end{corollary}
\begin{proof}
    Let $K$ be a compact subset of $\Omega_{\infty}$, from \eqref{eqn l infty bound for solution with g ell}, we get
    \begin{equation*}
        \int_{K}|\ul^{g}(x)|\ dx\leq C(K,n,s)\|g\|_{L^{1}(\Omega_{\infty},\delta_{\infty}^{-s})}|K|. \qedhere
    \end{equation*}
\end{proof}
Finally, we show that the function $u^{g}$ is the required solution.
\begin{theorem}\label{thrm gell convergeng to g}
    Assume \ref{G} on $\Omega=\Omega_{\infty}$. The function $u^{g}$ defined in \eqref{dfn of ug with g} solves \eqref{linear eqn on infinite cylinder} with the data $(0,g,0)$.
\end{theorem}
\begin{proof}
    Let $\psi\in L_{c}^{\infty}(\Omega_{\infty})$ and let $\ell_{0}>0$ such that, $K\subset \Omega_{\ell_{0}}$ and $\delta_{\Omega_{\ell}}=\delta_{\Omega_{\infty}}$. By Corollary \ref{prop v is l1 with g} and the dominated convergence theorem 
    \begin{equation*}
        \int_{\Omega_{\infty}}u\psi\ dx=\lim_{\ell\to \infty} \int_{\Omega_{\ell}}\ul \psi \ dx =-\lim_{\ell\to \infty} \int_{\R^{n}\setminus\overline{\Omega}_{\ell}}g_{\ell}(x) \slap \g{\Omega_{\ell}}(\psi)(x)\ dx.
    \end{equation*}
    Also, by \eqref{eqn complement integral behaviour} , one has 
    \begin{equation*}
        -\int_{\R^{n}\setminus\overline{\Omega}_{\ell}}g_{\ell}(x) \slap \g{\Omega_{\ell}}(\psi)(x)\ dx\leq \int_{\R^{n}\setminus\overline{\Omega}_{\ell}}|g_{\ell}(x)|~\delta^{-s}_{\Omega_{\ell}}(x)\ dx= \int_{\R^{n}\setminus\overline{\Omega}_{\infty}}|g(x)|~\delta^{-s}_{\Omega_{\infty}}(x)\ dx<\infty.
    \end{equation*}
    Therefore, by the dominated convergence theorem 
    \begin{equation*}
        \int_{\Omega_{\infty}}u\psi\ dx=-\int_{\R^{n}\setminus\overline{\Omega}_{\infty}}g(x)\lim_{\ell\to \infty}\slap \g{\Omega_{\ell}}(\psi)(x)\ dx=-\int_{\R^{n}\setminus\overline{\Omega}_{\infty}}g(x) \lim_{\ell\to\infty}\int_{\Omega_{\ell}}\psi(z)\slap\gk{\Omega_{\ell}}(z,x)\ dz\ dx,
    \end{equation*}
    where the last equality is due to \cite[Lemma 3.7]{Abatangelo2015LargesHarmonic}. By Lemma \ref{lemma greens seq is decreasing} and Lemma \ref{lemma convergence of greens kernel} and the compact support of $\psi$, using the dominated convergence theorem again, we finally have
    \begin{equation*}
        \int_{\Omega_{\infty}}u \psi \ dx=-\int_{\R^{n}\setminus\overline{\Omega}_{\infty}}g(x)\slap \g{\Omega_{\infty}}(\psi)(x)\ dx.\qedhere
    \end{equation*}
\end{proof}
\begin{remark}
    As in \cite[Remark 9, Remark 10]{Abatangelo2015LargesHarmonic}, we have for some $\eta>0$ small, that if
    \begin{equation*}
        \frac{m}{\delta^{\tau}_{\Omega_{\infty}}(y)}\leq g(y)\leq \frac{M}{\delta^{\sigma}_{\Omega_{\infty}}(y)}\quad \text{for } 0<\tau\leq \sigma<1-s,\text{on } 0<\delta_{\Omega_{\infty}}(y)\leq \eta    \end{equation*}
        where $m$ and $M$ are some positive constants, then 
        \begin{equation*}
            \frac{\Tilde{m}}{\delta^{\tau}_{\Omega_{\infty}}(y)}\leq u^{g}(y)\leq \frac{\Tilde{M}}{\delta_{\Omega_{\infty}}^{\sigma}(y)}
        \end{equation*}
        for some positive constants $\Tilde{m}$ and $\Tilde{M}$.
\end{remark}

\section{Semi-linear Equation}\label{section semilinear}
\subsection{On a Domain in $\rn$}
The primary challenge in working with semi-linear equations is comparing solutions on two finite cylinders. To address this, we examine the equation satisfied when restricting a solution to a subdomain. With this result,
we prove a comparison principle for a semi-linear equation in view of the monotonicity of $f$ in the second variable. We note that this is the first place where the monotonicity of $f$ is indispensable.  Further in this section, we prove the existence of a weak dual solution to \eqref{eqn main semilinear omega} on a bounded domain. 

Building on the weak dual notion of solution introduced in \cite{BonforteFigalliVazquez2018} for the semi-linear equation with data $(u^{p},0,0)$, $0<p<1$, as well as linear equations discussed above,  we define a weak dual notion of solution for the semi-linear equation \eqref{eqn main semilinear omega}.

 \begin{definition}\label{dfn weak dual for semilinear}
     Let $\Omega$ be a domain in $\rn$. Assume that $f$ satisfies \ref{F2}, $h\in L^{1}(\partial\Omega)$, $g$ satisfies \ref{G1} if $\Omega$ is bounded and \ref{G} if $\Omega$ is unbounded. A function $u\in L_{loc}^{1}(\Omega)$ is a weak dual solution to \eqref{eqn main semilinear omega}
  if 
 \begin{equation*}
 \begin{split}
     \int_{\Omega}u(x)\psi(x)\ dx=-& \int_{\Omega}f(x,u(x)) \g{\Omega}(\psi)(x)\ dx-\int_{\R^{n}\setminus\overline{\Omega}}g(y) \slap \g{\Omega}(\psi)(y)\ dy\\
      &\hspace{2cm}+\int_{\partial\Omega}h(\theta)\Ds \g{\Omega}(\psi)(\theta) \ d\mathcal{H}(\theta) \hspace{1.2cm}  \text{for all }\psi\in L_{c}^{\infty}(\Omega).
     \end{split}
 \end{equation*}
 Here, $\slap \g{\Omega}(\psi)(y)$ for $y\in \R^{n}\setminus\overline{\Omega}$ is defined pointwise. A function $u\in L^{1}_{loc}(\Omega)$ is a weak dual sub (or super) solution if 
 \begin{equation*}
 \begin{split}
     \int_{\Omega}u(x)\psi(x)\ dx\leq (\text{or } \geq) -&\int_{\Omega}f(x,u(x)) \g{\Omega}(\psi)(x)\ dx-\int_{\R^{n}\setminus\overline{\Omega}}g(y) \slap \g{\Omega}(\psi)(y)\ dy \\
     &\hspace{2cm} +\int_{\partial\Omega}h(\theta)\Ds \g{\Omega}(\psi)(\theta) \ d\mathcal{H}(\theta) \hspace{1.2cm} \text{for all }\psi\in L_{c}^{\infty}(\Omega), \psi\geq 0.
      \end{split}
 \end{equation*}
 
 \end{definition}
 \begin{remark}\label{Remark rep of solution for semilinear eqn}
  (i)  \,  If $u$ is a solution of \eqref{eqn main semilinear omega}, by Definition \ref{defining the operator G} and using Fubini's theorem (similar to the argument as in Theorem \ref{thrm rep of u on any domain}), for any $\psi\in L_{c}^{\infty}(\Omega)$ we write  
     \begin{equation*}
     \begin{split}
         \int_{\Omega}u(x)\psi(x)\ dx = &\int_{\Omega}\psi(x)\Big\{-\int_{\Omega}\gk{\Omega}(x,y)f(y,u(y))\ dy-\int_{\R^{n}\setminus\overline{\Omega}}\slap \gk{\Omega}(x,y)g(y)\ dy\Big\} \\
         & \hspace{2cm} +\int_{\Omega}\psi(x)\int_{\partial\Omega}\Ds\gk{\Omega}(\theta,x)h(\theta)\ d\mathcal{H}(\theta) \ dx. 
          \end{split}
     \end{equation*}
  With this, we infer that    
     \begin{equation*}
 u(x)=-\int_{\Omega}\gk{\Omega}(x,y)f(y,u(y))\ dy-\int_{\R^{n}\setminus\overline{\Omega}}\slap \gk{\Omega}(x,y)g(y)\ dy+\int_{\partial\Omega}\Ds\gk{\Omega}(\theta,x)h(\theta)\ d\mathcal{H}(\theta) \quad \text{a.e. in }\Omega.
     \end{equation*}
\smallskip 

\noindent (ii) \, Similarly, If $u$ is a sub (or, super) solution of \eqref{eqn main semilinear omega}, 
     \begin{equation*}
         u(x)\leq (\text{or,} \geq)-\int_{\Omega}\gk{\Omega}(x,y)f(y,u(y))\ dy-\int_{\R^{n}\setminus\overline{\Omega}}\slap \gk{\Omega}(x,y)g(y)\ dy+\int_{\partial\Omega}\Ds\gk{\Omega}(\theta,x)h(\theta)\ d\mathcal{H}(\theta), \, \, \text{a.e. in }\Omega.
     \end{equation*}
 \end{remark} 
Let $\Omega_{1}\subset\Omega_{2}$, and let $u_{i}$ be solutions of \eqref{eqn main semilinear omega} in $\Omega_{i}$ for $i=1,2$. To compare $u_{1}$ and $u_{2}$, we first need to determine the equation satisfied by $u_{2}$ in $\Omega_{1}$. The following lemma addresses this.
 \begin{lemma}\label{lemma restriction of a solution to subdomains}
 Assume \ref{F2} for $f$, and \ref{G1} for $g$. Let $\Omega'\subset \Omega$ be an open subset, and let $u$ is a solution of \eqref{eqn main semilinear omega}. For any $\psi\in L^{\infty}_{c}(\Omega')$, one has
 \begin{equation}\label{eqn ristriction to subdomain}
 \begin{split}
     \int_{\Omega'}u(x)\psi(x)\ dx=&-\int_{\Omega'}\g{\Omega'}(\psi)(y) f(y,u(y))\ dy-\int_{\Omega\setminus\Omega'}u(y)\slap \g{\Omega'}(\psi)(y)\ dy\\
& \hspace{2cm}+\int_{\partial\Omega'\cap\partial\Omega}h(\theta)\Ds[\g{\Omega'}(\psi)](\theta)\ d\mathcal{H}(\theta).
 \end{split}
    \end{equation}
 \end{lemma}
 \begin{proof}
     Let $v$ be the weak dual solution of 
     \begin{equation*}
         \begin{cases}
             \slap v= -f(x,u)\quad &\text{in }\Omega'\\
             v=u &\text{in }\R^{n}\setminus\overline{\Omega'}\\
             Ev=\Tilde{h} &\text{on }\partial\Omega'.
         \end{cases}
     \end{equation*}
     Here, $\Tilde{h}:=h$ on $\partial\Omega\cap\partial\Omega'$ and $\Tilde{h}=0$ on $\partial\Omega\setminus\partial\Omega'$. Whenever $\partial\Omega\cap\partial\Omega'=\phi$ then $\Tilde{h}\equiv0$. We divide the proof into three cases. 
     
     Case 1: Assume that $g=0$ and $h=0$. 
     For a.e. $x\in \Omega'$, by the representation of $u$, changing the order of integration and \eqref{eqn rep of G 1} and \eqref{eqn rep of G 2} below, we derive
     \begin{equation*}
     \begin{split}
    v(x)&=-\int_{\Omega'}f(y,u(y))\gk{\Omega'}(x,y)\ dy-\int_{\R^{n}\setminus\overline{\Omega'}}u(z)\slap \gk{\Omega'}(x,z)\ dz.\\
        &=-\int_{\Omega'}f(y,u(y))\gk{\Omega'}(x,y)\ dy-\int_{\Omega\setminus\Omega'}\slap \gk{\Omega'}(x,z) \left\{-\int_{\Omega}f(y,u(y))\gk{\Omega}(z,y)\ dy\right\}\ dz\\
        &=-\int_{\Omega'}f(y,u(y))\gk{\Omega'}(x,y)\ dy+\int_{\Omega}f(y,u(y))\int_{\Omega\setminus\Omega'}\slap \gk{\Omega'}(x,z)\gk{\Omega}(z,y)\ dz\ dy\\
        &=-\int_{\Omega'}f(y,u(y))\Big\{\gk{\Omega'}(x,y)-\int_{\Omega\setminus\Omega'}\slap \gk{\Omega'}(x,z)\gk{\Omega}(z,y)\ dz\Big\}\ dy\\
        &\hspace{3cm} + \int_{\Omega\setminus\Omega'}f(y,u(y))\int_{\Omega\setminus\Omega'}\slap \gk{\Omega'}(x,z)\gk{\Omega}(z,y)\ dz\ dy\\
        &=-\int_{\Omega}f(y,u(y))\gk{\Omega}(x,y)\ dy=u(x). 
     \end{split}
     \end{equation*}
     The last equality follows since $\gk{\Omega}(x,\cdot)$ is a distributional solution with Dirac delta at $x$, hence $\gk{\Omega}(x,\cdot)$ solves 
     \begin{equation*}
        \begin{cases}
            \slap \gk{\Omega}(x,\cdot)=\delta_{x}\quad &\text{in }\Omega'\\
            \gk{\Omega}(x,y)=\gk{\Omega}(x,y) &\text{in }\R^{n}\setminus\overline{\Omega'}.
        \end{cases}
     \end{equation*}
     Then by \cite[Theorem 1.4]{Abatangelo2015LargesHarmonic}, for $x,y\in \Omega'$, 
     \begin{equation}\label{eqn rep of G 1}
         \gk{\Omega}(x,y)=\gk{\Omega}(y,x)=\gk{\Omega'}(x,y)-\int_{\Omega\setminus\Omega'}\slap \gk{\Omega'}(x,z)\gk{\Omega}(y,z)\ dz.
     \end{equation}
     For $x\in \Omega'$ and $y\in \Omega\setminus\Omega'$, since $H_{\Omega}(x,\cdot)$ is a s-harmonic function, one can write
     \begin{equation*}
        H_{\Omega}(x,y)=H_{\Omega}(y,x)=-\int_{\R^{n}\setminus\overline{\Omega'}}H_{\Omega}(y,z) \slap \gk{\Omega'}(x,z)\ dz.
     \end{equation*}
     Also, by the properties of $\gk{\Omega'}$ and $H_{\Omega'}$, $\displaystyle \Gamma(x-y)=\Gamma(y-x)=-\int_{\R^{n}\setminus\overline{\Omega'}}\Gamma(y-z)\slap \gk{\Omega'}(x,z)\ dz.$
     Thus,
     \begin{equation}\label{eqn rep of G 2}
         \gk{\Omega}(x,y)=-\int_{\Omega\setminus\Omega'}\gk{\Omega}(y,z)\slap\gk{\Omega'}(x,z)\ dz.
     \end{equation}
    
     Case 2: Assume $f=0$ and $h=0$, then for a.e. $x\in \Omega'$, by the representation of $u$, \eqref{eqn rep of G 1}, and \eqref{eqn rep of G 2}, we obtain
     \begin{equation*}
         \begin{split}
             v(x)
             &=\int_{\Omega\setminus\Omega'}\left\{\int_{\R^{n}\setminus\overline{\Omega}}g(y)\slap \gk{\Omega}(z,y)\ dy\right\}\slap \gk{\Omega'}(x,z)\ dz-\int_{\R^{n}\setminus\overline{\Omega}}g(y)\slap \gk{\Omega'}(x,y)\ dy.\\
             &=-\int_{\R^{n}\setminus\overline{\Omega}}g(y)\left\{\slap\gk{\Omega'}(x,y)-\int_{\Omega\setminus\Omega'}\slap \gk{\Omega}(z,y) \slap \gk{\Omega'}(x,z)\ dz\right\}\ dy\\
             &=-\int_{\R^{n}\setminus\overline{\Omega}}g(y)\slap\gk{\Omega}(x,y)\ dy=u(x).
         \end{split}
     \end{equation*}
     For the last identity, since for $y\in \R^{n}\setminus\overline{\Omega}$ and $x\in \Omega'$, we note that
     \begin{equation*}
         \begin{split}
              -\slap \gk{\Omega}(x,y)& =\mathcal{A}(n,s)\text{ p.v} \int_{\Omega}\frac{\gk{\Omega}(x,z)}{|y-z|^{n+2s}}\ dz\\ 
             &=\mathcal{A}(n,s)\text{ p.v} \int_{\Omega'}\frac{\gk{\Omega'}(x,z)-\int_{\Omega\setminus\Omega'}\slap \gk{\Omega'}(x,\zeta)\gk{\Omega}(z,\zeta)\ d\zeta}{|y-z|^{n+2s}}\ dz \\
              &\hspace{4.5cm}  +\mathcal{A}(n,s)\int_{\Omega\setminus\Omega'}\frac{-\int_{\Omega\setminus\Omega'}\gk{\Omega}(z,\zeta)\slap\gk{\Omega'}(x,\zeta)\ d\zeta}{|y-z|^{n+2s}}\ dz\\
             &=-\slap\gk{\Omega'}(x,y)-\int_{\Omega\setminus\Omega'}\slap\gk{\Omega'}(x,\zeta) ~\mathcal{A}(n,s)\text{ p.v} \int_{\Omega}\frac{\gk{\Omega}(z,\zeta)}{|y-z|^{n+2s}}\ dz\ d\zeta\\
             &=-\slap\gk{\Omega'}(x,y)+\int_{\Omega\setminus\Omega'}\slap\gk{\Omega'}(x,z) \slap\gk{\Omega}(z,y)\ dz.
             \end{split}
     \end{equation*}
Case 3: Let $f=0$ and $g=0$, then for a.e. $x\in\Omega'$, using the representation of $u$,  we get
\begin{equation*}
    \begin{split}
        v(x)&=-\int_{\R^{n}\setminus\overline{\Omega'}}u(z)\slap\gk{\Omega'}(x,z)\ dz+\int_{\partial\Omega'\cap\partial\Omega}h(\theta)\Ds \gk{\Omega'}(x,\theta)\ d\mathcal{H}(\theta)\\
        &=-\int_{\Omega\setminus\Omega'}\left\{\int_{\partial\Omega}h(\theta)\Ds\gk{\Omega}(z,\theta)\ d\mathcal{H}(\theta)\right\}\slap\gk{\Omega'}(x,z)\ dz+\int_{\partial\Omega'\cap\partial\Omega}h(\theta)\Ds \gk{\Omega'}(x,\theta)\ d\mathcal{H}(\theta)\\
        &=\int_{\partial\Omega}h(\theta)\left\{-\int_{\Omega\setminus\Omega'}\lim_{y\to\theta}\frac{\gk{\Omega}(z,y)}{\delta^{s}(y)}\slap\gk{\Omega'}(x,z)\ dz\right\}\ d\mathcal{H}(\theta)+\int_{\partial\Omega'\cap\partial\Omega}h(\theta)\Ds \gk{\Omega'}(x,\theta)\ d\mathcal{H}(\theta)
    \end{split}
\end{equation*}
 By interchanging the limit and the integral, dividing the integral over $\partial\Omega=(\partial\Omega\setminus\partial\Omega')\cup (\partial\Omega\cap\partial\Omega')$, using \eqref{eqn rep of G 2} in the integral over $\partial\Omega\setminus\partial\Omega'$ and using \eqref{eqn rep of G 1} in the integral over $\partial\Omega\cap \partial\Omega'$, we obtain
\begin{equation*}
    \begin{split}
        v(x)&=\int_{\partial\Omega\setminus\partial\Omega'}h(\theta)\lim_{y\to\theta}\frac{\gk{\Omega}(x,y)}{\delta^{s}(y)}\ d\mathcal{H}(\theta)+\int_{\partial\Omega\cap\partial\Omega'}h(\theta)\left(\Ds \gk{\Omega}(x,\theta)-\Ds \gk{\Omega'}(x,\theta)\right)\ d\mathcal{H}(\theta)\\
        &\hspace{4cm} +\int_{\partial\Omega'\cap\partial\Omega}h(\theta)\Ds \gk{\Omega'}(x,\theta)\ d\mathcal{H}(\theta)\\
        &=\int_{\partial\Omega}h(\theta)\lim_{y\to\theta}\frac{\gk{\Omega}(x,y)}{\delta^{s}(y)}\ d\mathcal{H}(\theta)=u(x).
    \end{split}
\end{equation*}
      Finally, for any $\psi\in L_{c}^{\infty}(\Omega')$ $(\psi\geq 0)$, combining all cases, we get \eqref{eqn ristriction to subdomain}.
 \end{proof}
The existence of a solution is established using the sub- and super-solution method presented in \cite{Abatangelo2015LargesHarmonic} for $L^{1}$ weak solutions. We employ the same approach with the necessary modifications. The Main objective of the proof is to use the fixed-point theorem.
\begin{theorem}\label{thrm sub and super solution method}
    Let $\Omega$ be a bounded domain, $f$ satisfy \ref{F2}, and $g$ be a bounded measurable function. Let $\underline{u},\overline{u}\in L^{\infty}(\Omega)\cap C(\Omega)$. We assume that for every open subset $\Omega'\subseteq \Omega$ and $\psi\in L_{c}^{\infty}(\Omega')$ with $\psi\geq 0$ 
    \begin{equation}\label{eqn-local-sub-super-solution}
    \begin{split}
      \int_{\Omega'}\overline{u}(x)\psi(x)\ dx&\geq - \int_{\Omega'}f(x,\overline{u}(x))\g{\Omega'}(\psi)(x)\ dx-\int_{\Omega\setminus\Omega'}\overline{u}(x)\slap \g{\Omega'}(\psi)(x)\ dx\\ 
      &\hspace{2cm} -\int_{\rn\setminus\Omega}g(x)\slap \g{\Omega'}(\psi)(x)\ dx.
    \end{split}
    \end{equation}
    Similarly $\underline{u}$ solves the reversed inequality.
    If $\underline{u}\leq \overline{u}$, then the above equation admits a continuous weak dual solution $u$ such that $\underline{u}\leq u\leq \overline{u}$. 
\end{theorem}
The proof of the above theorem is present in the Appendix.
\begin{remark}
    The local notion of weak sub- and super-solution considered in the preceding theorem is indeed essential. Lemma \ref{lemma restriction of a solution to subdomains} does not necessarily apply to sub- or super-solutions as defined in Definition \ref{dfn weak dual for semilinear}; for instance, consider a non-negative $C_{c}^{\infty}$ function. 
\end{remark}
In fact, the stronger notions of sub-solution and supersolution allow for a comparison principle.
\begin{lemma}\label{lemma comparison principle for semilinear eqn}
     Let $\Omega$ be a domain with exterior ball condition. Assume that $f$ satisfies \ref{F2}, $g$ satisfies \ref{G1} if $\Omega$ is bounded or \ref{G} if $\Omega$ is unbounded and $h$ satisfies \ref{H}. If $u\in C(\Omega)$ is a sub-solution and $v\in C(\Omega)$ is a super-solution of \eqref{eqn main semilinear omega}. Furthermore,  $v$ satisfies \eqref{eqn-local-sub-super-solution} for every open subset $\Omega'\subset\Omega$  and similarly $u$ satisfies the reversed inequality, then $u\leq v$ in $\Omega$.

\end{lemma}
\begin{proof}
    On the contrary, let $\Omega_{1}:=\{x\in\Omega~|~v(x)<u(x)\}$ be an open set. For any $\psi\in L_{c}^{\infty}(\Omega_{1})$ with $\psi\geq 0$, using \eqref{eqn-local-sub-super-solution} for $v$ and the corresponding reverse inequality for $u$, subtracting the resulting inequalities,  we obtain  
\begin{equation*}
       \int_{\Omega_{1}}(u-v)(y)\psi(y)\ dy\geq \int_{\Omega_{1}}\{f(y,u(y))-f(y,v(y))\}\g{\Omega_{1}}(\psi)(y)\ dy-\int_{\mathbb{\R}^{n}\setminus\overline{\Omega}_{1}}(u-v)(y)\slap\g{\Omega_{1}}(\psi)(y)\ dy.
    \end{equation*}
     By the definition of $\Omega_{1}$, the left-hand side is negative.  By monotonicity of $f$ and the definition of $\Omega_{1}$, the right hand side is positive. This gives a contradiction. 
\end{proof}

Next, we prove the existence of a large solution, adapting the argument used in \cite{Abatangelo2015LargesHarmonic} for the existence of a weak $L^{1}$ large solution. 
\begin{theorem}[\textbf{Existence}]\label{theorem existence}
    Let $\Omega$ be a bounded domain with $C^{1,1}$ boundary, $f,g,h$ satisfy \ref{F2}, \ref{G1} and \ref{H}, respectively with $g\geq 0$. Then, there exists a weak dual solution of \eqref{eqn main linear omega}. Moreover, the solution can be chosen to be non-negative.
    \end{theorem}
We give the proof of this theorem in the Appendix as it is not the main focus of this paper.
\begin{remark}\label{remark dropping monotonicity}
    It is possible to drop the monotonicity assumption on $f$ for both Theorem \ref{thrm sub and super solution method} and Theorem \ref{theorem existence}. In Theorem \ref{thrm sub and super solution method}, it is possible to prove that the solution $u$ obtained is minimal in the sense that if $v$ is any super solution of \eqref{eqn main linear omega} for the data $(-f(\cdot,u),g,0)$ with $\underline{u}\leq v$, then $u\leq v$ (see \cite{Abatangelo2015LargesHarmonic} and \cite[Corollary 2.2]{DumontDupaigneOlivierRuadulescu2007}). Further, with this modification, one can use minimality instead of the comparison principle (see the argument in Lemma \ref{lemma comparison principle for semilinear eqn}) to prove Theorem \ref{theorem existence}. Since we finally need $f$ to be increasing in the second variable for the investigation on the infinite cylinder, we proved the existence result for the same.
\end{remark}
We observe that the solution obtained above is continuous:
By the representation of the solution given in Remark \ref{Remark rep of solution for semilinear eqn}, we denote $u_{1},u_{2},u_{3}$ as the corresponding integrals on the domain, its complement and boundary, respectively. 

In view of the assumption \ref{G1}, $u_{2}$ is bounded by \eqref{eqn condition on g}. By the dominated convergence theorem $u_{2}$ is continuous. 

Furthermore, for $h\in C(\partial\Omega)$, $u(x)\leq C\delta_{\Omega}^{s-1}$ in view of Lemma \ref{lemma equivalent E}. By \ref{F2} and $u(x)\leq C\delta_{\Omega}^{s-1}$ away from $x$; and using the bound of $u$ and the integrability of Green's kernel near $x$, one can apply  the dominated convergence theorem to deduce the continuity of $u_{1}$.

By the above arguments, we have a continuous representative of the solution obtained in Step 1 of Theorem \ref{theorem existence}. Finlay, in view of the local uniform convergence in Step 2 and Step 3, the solution obtained in Theorem \ref{theorem existence} is continuous.

\begin{theorem}[Uniqueness]
Let $f,g$ and $h$ be as in Theorem \ref{theorem existence}. Then the problem admits a unique continuous non-negative solution.
\end{theorem}
 \begin{proof}
    Let $u, v\in L^{1}_{loc}(\Omega)$ be two solutions. 
    If $\Omega_{1}:=\{x\in \Omega~|~u(x)<v(x)\}$ is non empty, then its open by the continuity.
     For $\psi\in L_{c}^{\infty}(\Omega_{1})$ with $\psi\geq 0$, by Lemma \ref{lemma restriction of a solution to subdomains} we obtain
    \begin{equation*}
    \begin{split}
        \int_{\Omega_{1}}(u-v)(x)\psi(x)\ dx=&\int_{\Omega_{1}}\left(f\left(x,v(x)\right)-f\left(x,u(x)\right)\right)\g{\Omega_{1}}(\psi)(x)\ dx-\int_{\Omega\setminus\Omega_{1}} (u-v)(x)\slap\g{\Omega_{1}}(\psi)(x)\ dx\\
        &\geq \int_{\Omega_{1}}\left(f\left(x,v(x)\right)-f\left(x,u(x)\right)\right)\g{\Omega_{1}}(\psi)(x)\ dx\geq 0.
        \end{split}
    \end{equation*}
    Therefore, $u-v$ is a super solution to 
    \begin{equation*}
        \begin{cases}
            \slap w=0\quad &\text{in }\Omega_{1}\\
            w=0 &\text{in }\R^{n}\setminus\overline{\Omega}_{1}\\
            E_{\Omega_{1}}(w)=0 &\text{on }\partial\Omega_{1}.
        \end{cases}
    \end{equation*}
    Since zero is a solution, by the comparison principle given in Lemma \ref{lemma comparison principle}, $u-v\geq 0$ a.e. in $\Omega_{1}$. By a similar argument, we derive that $u=v$ on $\Omega$.
\end{proof} 
It is known that for $g=h=0$, the corresponding `potentials' of the semi-linear problem vanish on the boundary. This guarantees the nonexistence of large solutions in such cases.  We state the result for completeness.  
\begin{theorem}\label{thrm no large solution}
    Assume \ref{F2} but without the monotonicity assumption of $f(x,t)$ in the second variable, and let $\Omega$ be a bounded domain. Then the problem \eqref{eqn main semilinear omega} does not admit a large solution with the data $(f(x,\cdot),0,0)$.
\end{theorem}
\begin{proof}
   As in the proof of Theorem \ref{theorem existence} (see the calculations regarding $I_1^j$ in the Appendix)
   \begin{equation*}
       \int_{\Omega}f(y,u(y))\gk{\Omega}(x,y)\ dy\leq a_{1}\int_{B_{r}(x)}\frac{1}{|x-y|^{n-2s}}\ dy+C\int_{\Omega\setminus B_{r}(x)}\delta_{\Omega}^{(s-1)p+s}(y)\ dy<+\infty.
   \end{equation*}
  The first integral can be bounded independently of $x$ by bounding the integral by an integral over a ball of sufficiently large radius.  Thus, by the dominated convergence theorem 
   \begin{equation*}
       \int_{\Omega}f(y,u(y))\gk{\Omega}(x,y)\ dy\to 0 \quad\text{as }x\to \partial\Omega.\qedhere
   \end{equation*} 
\end{proof}
\begin{remark}
   If in addition, $f(x,t)$ is monotone with respect to the second variable, then the above result follows immediately from the fact that $f(x,0)=0$ and the uniqueness of the solution.
\end{remark}

\begin{lemma}\label{local uniform bound for semi-linear eqn on bounded domain}
    Let $\Omega$ be a bounded domain, $f$ and $g$ be functions satisfying \ref{F2} and \ref{G1} respectively. Also let $h\in L^{1}(\partial\Omega)$. If $u$ a solution of \eqref{eqn main semilinear omega}, then for any compact subset $K$ of $\Omega$, there exists a positive constant $C$ such that 
    \begin{equation*}
        \|u\|_{L^{\infty}(K)}\leq C(n,s,K,a_{i},g,h)
    \end{equation*}
\end{lemma}
\begin{proof}
     By the representation of $u$ given in Remark \ref{Remark rep of solution for semilinear eqn}
    \begin{equation*}
        |u(x)|\leq I_{f}+I_{g}+I_{h}.
    \end{equation*}
    by Lemma \ref{lemma bound of the sln ug}
    \begin{equation*}
        I_{g}:=-\int_{\R^{n}\setminus\overline{\Omega}}|g(y)|\slap\gk{\Omega}(x,y)\ dy\leq C(n,s,K)\int_{\R^{n}\setminus\overline{\Omega}}|g(y)| \text{ min}\{\delta^{-s}(y),\delta^{-n-2s}(y)\}\ dy.
    \end{equation*}
    \begin{equation*}
        I_{h}:=\int_{\partial\Omega}h(\theta)\Ds\gk{\Omega}(x,\theta)\ d\mathcal{H}(\theta).
    \end{equation*}
    Applying \eqref{bound near the boundary for f compactly supported} to $u_{i}$ solving \eqref{eqn main linear omega} with the data $(f_{i},0,0)$, where $f_{i}$ is defined in \eqref{eqn dfn of fi} (Case 2 of the proof of Theorem \ref{theorem existence} in the Appendix), we get 
    \begin{equation*}
        |u_{i}(x)|\leq C \delta_{\Omega}^{s}(x)\text{dist}(x,\text{supp}(f_{i}))^{s-n}\int_{\Omega}f_{i}(y)\delta_{\Omega}^{s}(y)\ dy\quad \text{for all }x\in K.
    \end{equation*}
    Therefore, up to a subsequence $u_{i}\overset{*}{\rightharpoonup}I_{h}$ in $L^{\infty}(K)$ in view of \eqref{interier estimate for h} in the Appendix. Thus, using the weak $*$ lower semi-continuity of the norm
    \begin{equation*}
    \begin{split}
        I_{h}\leq \liminf_{i}\|u_{i}\|_{L^{\infty}(K)}&\leq C \liminf_{i}\delta_{\Omega}^{s}(x)\text{dist}(x,\text{supp}(f_{i}))^{s-n}\int_{\Omega}f_{i}(y)\delta_{\Omega}^{s}(y)\ dy
        \leq C \text{dist}(K,\partial\Omega)^{s-n}\|h\|_{L^{1}(\partial\Omega)}.
        \end{split}
    \end{equation*}
    This proof of the bound of $I_{h}$ can be found in \cite[Theorem 4.6]{AbatangeloVazquez2023}.
    Finally
    \begin{equation*}
        I_{f}:=\int_{\Omega}|f\left(y,u(y)\right)|\gk{\Omega}(x,y)\ dy\leq \int_{\Omega}\left(a_{1}+a_{2}\delta_{\Omega}^{(s-1)p}(y)\right)\gk{\Omega}(x,y)\ dy.
    \end{equation*}
    Consider a compact subset $K_{1}\subset\Omega$ such that dist$(K,K_{1}^{c})>\alpha>0$, then 
    \begin{equation*}
        \begin{split}
            I_{f}&\leq \left\{\int_{K_{1}}+\int_{\Omega\setminus K_{1}}\right\}\left(a_{1}+a_{2}\delta_{\Omega}^{(s-1)p}(y)\right)\gk{\Omega}(x,y)\ dy\\
            &\leq C(K,\Omega,a_{1},a_{2},\alpha)\left\{\int_{K_{1}}\frac{dy}{|x-y|^{n-2s}}+\int_{\Omega\setminus K_{1}}a_{1}+a_{2}\delta_{\Omega}^{(s-1)p+s}(y) \ dy\right\}.
        \end{split} 
    \end{equation*}
    Combining all three integrals, we get the bound for $\|u\|_{L^{\infty}(K)}$.
\end{proof}

\subsection{On the Infinite Cylinder}\label{subsection_4.2}
Here we examine the solution of the semi-linear equation in the infinite cylinder $\Omega_{\infty}$ approximated by a sequence of solutions on finite cylinders $\Omega_{\ell}$, where the solutions on finite cylinders exhibit blow-up behaviour on the lateral boundary $\partial\Omega_{\infty} \cap \partial\Omega_{\ell}$. The primary step involves comparing solutions on different cylinders -- by Lemma \ref{lemma restriction of a solution to subdomains} we obtain,
\begin{lemma}
    Let $\ell_{1}\leq \ell_{2}$ and $u_{\ell_{i}}$ are solutions of \eqref{semi-linear prob on finite cylinder} for $i=1,2$ respectively. If $f$ is increasing in the second variable, then $u_{\ell_{1}}\leq u_{\ell_{2}}$ on $\Omega_{\ell_{1}}$.
\end{lemma}
\begin{proof}
Define $w(x)=u_{\ell_{2}}-u_{\ell_{1}}$. 
If $\Omega':=\{x\in \Omega_{\ell_{1}}  ~|~u_{\ell_{2}}(x)<u_{\ell_{1}}(x)\}$ is non-empty, it is open by the continuity of solutions. Choose $\psi\in L_{c}^{\infty}(\Omega')$, $\psi\geq 0$. Then by Lemma \ref{lemma restriction of a solution to subdomains}
    \begin{equation*}
    \begin{split}
     \int_{\Omega'}w\psi\ dy&= \int_{\Omega'}\{f(y,u_{\ell_{1}}(y))-f(y,u_{\ell_{2}}(y))\}\g{\Omega'}(\psi)\ dy-\int_{\Omega_{\ell_{2}}\setminus\Omega_{\ell_{1}}}u_{\ell_{2}}(y)\slap\g{\Omega'}(\psi)(y)\ dy\\
     &  -\int_{\Omega_{\ell_{1}}\setminus\Omega'}(u_{\ell_{2}}(y)-u_{\ell_{1}}(y))\slap\g{\Omega'}(\psi)(y)\ dy +\int_{S}h(\theta)\Ds[\g{\Omega'}(\psi)](\theta)\ d\mathcal{H}(\theta),  
    \end{split}
    \end{equation*}
   where, $S=(\partial\Omega_{\ell_{2}}\setminus\partial\Omega_{\ell_{1}})\cap \partial\Omega'\cap\partial\Omega_{\infty}$. 
   
   Since in the proof of Theorem \ref{theorem existence} using Theorem \ref{thrm sub and super solution method}, we have $\underline{u}=0$, $u_{\ell_{2}}\geq 0$, thus by the definition of $\Omega'$ $LHS$ of the above equation is negative and $RHS$ is positive. This is a contradiction, hence $|\Omega'|=0$. 
\end{proof}
Next, let $x\in \Omega_{\infty}$, choose $\ell_{0}(x)>0$ such that for all $\ell\geq\ell_{0}(x)$, $x\in \Omega_{\ell}$. Since they are locally bounded, either by the proof of Lemma \ref{local uniform bound for semi-linear eqn on bounded domain} with the fact that $\delta_{\Omega_{\ell}}\leq \delta_{\Omega_{\infty}}$ or by comparing with a solution of the linear equation with the data $(0,g,h)$, we define
 \begin{equation*}
     u(x):=\sup_{\ell\geq\ell_{0}(x)}\ul(x) \quad \text{for all }x\in \Omega_{\infty}.
 \end{equation*}
 \begin{theorem}
     Assume \ref{F2}, \ref{G}, and \ref{H} on $\Omega=\Omega_{\infty}$. In addition $h\in  L^{\infty}(\partial\Omega_{\infty})$. 
      Then $u\in L_{loc}^{1}(\Omega)$ and $u$ is a solution of \eqref{semi-linear prob on infinite cylinder}.
 \end{theorem}
 \begin{proof}
     By Lemma \ref{local uniform bound for semi-linear eqn on bounded domain}, $u\in L_{loc}^{1}(\Omega)$.

     Let $\psi\in L_{c}^{\infty}(\Omega_{\infty})$, and let $K$ be a compact set such that Supp$(\psi)\subset K\subset \Omega_{\infty}$, denote $\ell(K)>0$ such that $K\subset \Omega_{\ell}$ for all $\ell\geq \ell(K)$. Then for all $\ell\geq \ell(K)$ we have 
 \begin{equation*}
 \begin{split}
     \int_{\Omega_{\ell}}u_{\ell}\psi\ dx =-&\int_{\Omega_{\ell}}f\left(x,u_{\ell}(x)\right)\g{\Omega_{\ell}}(\psi)(x)\ dx-\int_{\R^{n}\setminus\overline{\Omega}_{\infty}}g(y)\slap\g{\Omega_{\ell}}(\psi)(y)\ dy\\
     &\hspace{2cm}+\int_{\partial\Omega_{\ell}\cap \partial\Omega_{\infty}}h(\theta) \Ds[\g{\Omega_{\ell}}(\psi)](\theta)\ d\mathcal{H}(\theta).
     \end{split}
 \end{equation*}
 By the dominated convergence theorem 
 \begin{equation}\label{eqn operator part convergence}
     \lim_{\ell\to\infty}\int_{\Omega_{\ell}}\ul\psi\ dx=\lim_{\ell\to \infty}\int_{K}\ul\psi\ dx=\int_{K}u\psi\ dx=\int_{\Omega_{\infty}}u\psi\ dx.
 \end{equation}
 As in Theorem \ref{thrm gell convergeng to g} and Theorem \ref{thrm for operator blowup}, we have, respectively,
 \begin{equation}\label{eqn complementry part convergence}
    \lim_{\ell\to\infty} \int_{\R^{n}\setminus\overline{\Omega}_{\infty}}g(y)\slap\g{\Omega_{\ell}}(\psi)(y)\ dy= \int_{\R^{n}\setminus\overline{\Omega}_{\infty}}g(y)\slap\g{\Omega_{\infty}}(\psi)(y)\ dy
 \end{equation}
 and 
 \begin{equation}\label{eqn boundary part convergence}
     \lim_{\ell\to\infty} \int_{\partial\Omega_{\ell}\cap \partial\Omega_{\infty}}h(\theta) \Ds[\g{\Omega_{\ell}}(\psi)](\theta)\ d\mathcal{H}(\theta)=\int_{ \partial\Omega_{\infty}}h(\theta) \Ds[\g{\Omega_{\infty}}(\psi)](\theta)\ d\mathcal{H}(\theta).
 \end{equation}
 Finally, to deal with the first integral, for $\psi\geq 0$, the function $\chi_{\Omega_{\ell}}f\left(x,u_{\ell}(x)\right)\g{\Omega_{\ell}}(\psi)(x)$ is increasing in $\ell$ as $f(x,t)$ is increasing in $t$ variable and $\g{\Omega_{\ell}}(\psi)(x)$ is increasing in $\ell$. Then by the continuity of $f$ and Lemma \ref{Lemma local uniform convergence of solution operator} and the monotone convergence theorem
 
 \begin{equation}\label{eqn u is sub sln on infinite cylinder}
     \int_{\Omega_{\infty}} f\left(x,u(x)\right) \g{\Omega_{\infty}}(\psi)(x)\ dx = \lim_{\ell\to\infty}\int_{\Omega_{\ell}}f\left(x,\ul(x)\right) \g{\Omega_{\ell}}(\psi)(x)\ dx.
 \end{equation}
 Indeed, since $\ul\geq 0$, the integral on the left hand side is bounded by the sum of integral on the complement and on the boundary. 

 For $\psi\in L_{c}^{\infty}(\Omega_{\infty})$, separate the positive and negative parts $\psi=\psi^{+}-\psi^{-}$, applying the above arguments both $\psi^{\pm}$, as $\g{\Omega_{\infty}}$ is linear we obtain \eqref{eqn u is sub sln on infinite cylinder} for all $\psi\in L_{c}^{\infty}(\Omega_{\infty})$.
 
 Combining \eqref{eqn operator part convergence}, \eqref{eqn complementry part convergence}, \eqref{eqn boundary part convergence} and \eqref{eqn u is sub sln on infinite cylinder} we have
\begin{equation*}
\begin{split}
    \int_{\Omega_{\infty}}u(x)\psi(x)\ dx=-&\int_{\Omega_{\infty}}f\left(x,u(x)\right)\g{\Omega_{\infty}}(\psi)(x)\ dx-\int_{\R^{n}\setminus\overline{\Omega}_{\infty}}g(y)\slap\g{\Omega_{\infty}}(\psi)(y)\ dy\\
    &\hspace{2cm}+\int_{\partial\Omega_{\infty}}h(\theta) \Ds\gk{\Omega_{\infty}}(x,\theta)\ d\mathcal{H}(\theta).\qedhere
    \end{split}
\end{equation*}
\end{proof}
\subsection{ Uniqueness}\label{subsection_4.3}
Assume $u$ and $v$ are two continuous solutions. 
Take $\Omega':=\{x\in \Omega_{\infty}~|~v(x)<u(x)\}$, then by Lemma \ref{lemma restriction of a solution to subdomains}, for any $\psi\in L^{\infty}_{c}(\Omega')$ with $\psi\geq 0$, we have 
\begin{equation*}
\begin{split}
    \int_{\Omega'}(u-v)(x)\psi(x)\ dx&=\int_{\Omega'}\left(f(x,v(x))-f(x,u(x))\right)\g{\Omega'}(\psi)(x)\ dx -\int_{\Omega_{\infty}\setminus\Omega'}(u-v)(y)\slap\g{\Omega'}(\psi)(y)\ dy.
\end{split}
\end{equation*}
This is a contradiction, as LHS is strictly positive and RHS is negative. Thus $|\Omega'|=0$. By interchanging the role of $u$ and $v$, we obtain the uniqueness.

\section{Appendix}\label{section appendix}
\begin{proof}[\textbf{Proof of the Theorem \ref{thrm sub and super solution method}}]
Let $v$ be an s-harmonic function such that $v=g$ in $\R^{n}\setminus\overline{\Omega}$ and $Ev=0$ on $\partial\Omega$. Further, consider $w$ as a solution of \eqref{eqn main semilinear omega} with data $(-f(\cdot,w+v), 0,0)$. 
 Then, $u=w+v$ solves \eqref{eqn main semilinear omega} with the data $(-f(\cdot,u),g,0)$. Since $g$ is bounded, $v$ is continuous (Poisson representation) and thus the function $\Tilde{f}(x,t)=f(x,v(x)+t)\in C(\Omega\times \R)$. With an abuse of notation, $\underline{u}$ and $\overline{u}$ still denote the sub- and super-solution, respectively.
 Define $F:\Omega\times\R\to\R$ as 
 \begin{equation*}
     F(x,t)=\begin{cases}
         \Tilde{f}(x,\overline{u}(x)) \quad &\text{if }t >\overline{u}(x) \\
         \Tilde{f}(x,t) &\text{if } \underline{u}(x)\leq t\leq \overline{u}(x)\\
         \Tilde{f}(x,\underline{u}(x)) &\text{if } t<\underline{u}(x).
     \end{cases}
 \end{equation*}
 Clearly $F\in C(\Omega\times \R)\cap L^{\infty}(\Omega\times \R)$ by \ref{F2} and boundedness of $\underline{u},\overline{u}$ respectively.

Consider the map $\mathcal{T}:L^{\infty}(\Omega)\to L^{\infty}(\Omega)$ given by $\mathcal{T}u=\g{\Omega}\left(-F(x,u(x))\right)$, which is well-defined by \eqref{eqn continuity of G 1} as $F$ is bounded. 
 Further,  $\g{\Omega}\left(-F(x,u(x))\right)\in C^{s}(\rn)$. On the other hand, one can write $\mathcal{T}=\mathcal{T}_{2}\circ\mathcal{T}_{1}$ where $\mathcal{T}_{1}:L^{\infty}(\Omega)\to L^{\infty}(\Omega)$ and $\mathcal{T}_{2}:L^{\infty}(\Omega)\to L^{\infty}(\Omega)$, defined by 
 \begin{equation*}
     \mathcal{T}_{1}u=-F\left(x,u(x)\right) \quad\quad \text{and } \quad \quad\mathcal{T}_{2}v=\g{\Omega}(v).
 \end{equation*}
 Clearly, $\mathcal{T}_{1}$ takes a bounded set to a bounded set. In view of the regularity $C^{s}(\rn)$ the map $\mathcal{T}_{2}$ is compact to $C^{b}(\Omega)$, and so to $L^{\infty}(\Omega)$. Therefore, $\mathcal{T}$ is a compact map. Further, since $F(x,u(x))\leq \Tilde{f}(x,\overline{u}(x))$, the
existence follows from Schaefer's fixed point theorem (see \cite{EvansPDE2010}). In fact $u\in C(\Omega)$ by Proposition \ref{prop-greens-potintial-continuity}.

We next prove that $ u\leq \overline{u}$ on $\Omega$. Suppose $\Omega':=\{x\in \Omega ~|~u(x)>\overline{u}(x)\}$, which is an open set by the continuity of both $u$ and $\overline{u}$. Let $\psi\in L_{c}^{\infty}(\Omega')$, then by Lemma \ref{lemma restriction of a solution to subdomains}
\begin{equation*}
    \begin{split}
        \int_{\Omega'}u(x)\psi(x)\ dx&= -\int_{\Omega'}F\left(x,u(x)\right) \g{\Omega'}(\psi)(y) \ dy-\int_{\Omega\setminus\Omega'}u(y)\slap \g{\Omega'}(\psi)(y)\ dy\\
        &=-\int_{\Omega'}F\left(x,\overline{u}(x)\right) \g{\Omega'}(\psi)(y) \ dy-\int_{\Omega\setminus\Omega'}u(y)\slap \g{\Omega'}(\psi)(y)\ dy 
        \leq \int_{\Omega'}\overline{u}(x)\psi(x)\ dx.
    \end{split}
\end{equation*}
Therefore $u\leq \overline{u}$ on $\Omega'$. This is a contradiction. Thus, $\Omega'=\emptyset$.

Similarly, one can prove that $\underline{u}\leq u$.
\end{proof}

 \begin{proof}[\textbf{Proof of the Theorem \ref{theorem existence}}]
The existence will be proved in three steps: $h\equiv0$, $h\in C(\partial\Omega)$ and $h\in L^{1}(\partial\Omega)$. When $h\equiv 0$, the result is a consequence of Theorem \ref{thrm sub and super solution method}. The second case is by approximating the boundary data by a sequence of interior data whose weights concentrate towards the boundary. The last case is by an approximation of $L^{1}$ function from continuous functions.

    \textbf{Step 1:} Let $h\equiv 0$ and define $g_{n}(x):=min\{g(x),n\}$. Consider 
    \begin{equation}\label{eqn with gn}
        \begin{cases}
         \slap u=-f(x,u)\quad &\text{in } \Omega\\
         u=g_{n} &\text{in }\R^{n}\setminus\overline{\Omega}\\
         E_{\Omega}(u)=0 &\text{on }\partial\Omega,
    \end{cases}
    \end{equation}
   Clearly $v=n$ is a super-solution for all $n\in \mathbb{N}$ and $v=0$ is a sub-solution. Indeed, for any $\Omega'\subseteq \Omega$, $\psi\in L_{c}^{\infty}(\Omega')$ with $\psi\geq 0$, one has
    \begin{equation*}
        -\slap\g{\Omega'}(\psi)(y)=-\int_{\Omega'}\psi(x) \slap\gk{\Omega'}(x,y)\ dx \quad \text{for all }y\in \rn\setminus\Omega'.
    \end{equation*}
    Integrating the above equation on $\rn\setminus\Omega'$ and using the fact that $\int_{\rn\setminus\Omega'}-\slap\gk{\Omega'}(x,y)\ dy=1$ for $x\in \Omega'$, we can write
    \begin{equation*}
        \int_{\Omega'}\psi(x)\ dx=-\int_{\rn\setminus\Omega'}\slap\g{\Omega'}(\psi)(y))\ dy.
    \end{equation*}
    Since $g_{n}\leq n$ and $f(x,n)>0$, we get
    \begin{equation*}
        \int_{\Omega'}n\psi(x)\ dx\geq-\int_{\Omega'}f(x,n)\g{\Omega'}(\psi)(x)\ dx-\int_{\rn\setminus\Omega'}g_{n}(y)\slap\g{\Omega'}(\psi)(y))\ dy.
    \end{equation*}
    The corresponding verification for $v=0$ follows because the left hand side is always $0$ and the right hand side is non-negative as $f(x,0)=0$ and $g\geq 0$.
    By Theorem \ref{thrm sub and super solution method}, let $u_{1}\in C(\Omega)$ be the solution corresponding to the data $g_{1}$ with $v=0$ and $v=1$ as sub-solution and super-solution, respectively.  Then $0\leq u_{1}\leq 1$. As $g_{1}\leq g_{2}$, and for any $\psi\in L_{c}^{\infty}(\Omega)$ $\psi\geq 0$, since $\slap \g{\Omega}(\psi)\leq 0$, $u_{1}$ is a sub-solution to \eqref{eqn with gn} for $n=2$. Employing Lemma \ref{lemma restriction of a solution to subdomains} and the above theorem, let $u_{2}$ be a solution corresponding to the data $g_{2}$ such that $u_{1}\leq u_{2}\leq 2$. We construct an increasing sequence $\{u_{n}\}$ such that $u_{n}$ is a solution of \eqref{eqn with gn} and $u_{n-1}\leq u_{n}\leq n$. 
    
    By Theorem \ref{thrm linear existence with g}, let $u^{0}$ be a solution of \eqref{eqn main linear omega} with data $(0,g,0)$. As $f(x,t)\geq 0$ and $g_{n}\leq g$, $u_{n}$ is a sub solution of \eqref{eqn main linear omega} with the data $(0,g,0)$; this holds for all $n$, and so by the comparison principle given in Lemma \ref{lemma comparison principle}, the sequence $\{u_{n}\}$ is bounded by $u^{0}$. Define $u(x):=\sup_{k}u_{k}(x)$. Then $0\leq u\leq u^{0}$, and so $u\in L_{loc}^{1}(\Omega)$. Also, $u=g$ in $\R^{n}\setminus\overline{\Omega}$. 
    
    Next, we show that $u$ is a solution of \eqref{eqn main linear omega} with $h=0$.
    Let $\psi\in L^{\infty}_{c}(\Omega)$, then, by the dominated convergence theorem 
    \begin{equation*}
        \begin{split}
            \int_{\Omega}u(x)\psi(x)\ dx
            &=\lim_{k\to \infty}\int_{\Omega}u_{k}(x)\psi(x)\ dx \\
           & =\lim_{k\to\infty}\left\{-\int_{\Omega}f\left(x,u_{k}(x)\right)\g{\Omega}(\psi)(x)\ dx-\int_{\R^{n}\setminus\overline{\Omega}}g_{k}(y) \slap \g{\Omega}(\psi)(y)\ dy\right\}.
        \end{split}
    \end{equation*}
    Given \ref{F2}, $f(x,u_{k}(x))\leq a_{1}+a_{2}(u^{0})^{p}$. Further, $Eu^{0}=0$ and Lemma \ref{lemma equivalent E} along with Lemma \ref{lemma bound of the sln ug} implies that $u^{0}(x)\leq C \delta_{\Omega}^{s-1}(x)$ in $\Omega$. Using this with \eqref{eqn continuity of G 2}, we have 
    \begin{equation*}
        f(x,u_{k}(x))\g{\Omega}(\psi)(x)\leq C(a_{1}+a_{2}\delta_{\Omega}^{(s-1)p}(x))\delta^{s}_{\Omega}(x),
    \end{equation*}
    where the first term is integrable since $\Omega$ is bounded, and second term since $p<(1+s)/(1-s)$. Therefore, by the dominated convergence theorem, we obtain
    \begin{equation*}
         \int_{\Omega}u(x)\psi(x)\ dx= -\int_{\Omega}f\left(x,u(x)\right) \g{\Omega}(\psi)(x)\ dx -\int_{\R^{n}\setminus\overline{\Omega}}g(y) \slap \g{\Omega}(\psi)(y)\ dy.
    \end{equation*}
      The function $u$ is the required solution in this case.
\smallskip 

    \textbf{Step 2:}
    Let $h\in C(\partial\Omega)$. In view of the $C^{1,1}$ boundary of $\Omega$, we consider a continuous extension $H$ of $h$ to a neighbourhood of the boundary in the interior of $\Omega$, as given in \cite[Theorem 4.6 and Remark 3.1]{AbatangeloVazquez2023}. Let $\epsilon_{0}>0$ be sufficiently small, define $\Omega_{\epsilon_{0}}:=\{x\in \Omega~|~ \delta_{\Omega}(x)<\epsilon_{0}\}$, and consider 
    \begin{equation*}
        H(x):=h(z(x))\quad  \text{ where, } z(x)\in \partial\Omega \text{ such that } \delta_{\Omega}(x)=|x-z(x)|, \text{ for all } x\in \Omega_{\epsilon_{0}}.
    \end{equation*}
    Note that $H\in C(\Omega_{\epsilon_{0}}\cup \partial\Omega)$. For all $i>2/\epsilon_{0}$, define $\Omega_{i}:=\{x\in \Omega~|~ 1/i<\delta_{\Omega}(x)<2/i\}$ and 
    \begin{equation}\label{eqn dfn of fi}
        f_{i}(x):=\frac{\mathcal{H}(\partial\Omega)H(x)\chi_{\Omega_{i}}}{|\Omega_{i}|\delta_{\Omega}^{s}(x)}, \quad \text{for all } x\in \Omega.
    \end{equation}
    Clearly, $f_{i}\in L^{1}(\Omega,\delta_{\Omega}^{s})$. Consider the problem 
    \begin{equation}\label{eqn interior perturbation}
         \begin{cases}
         \slap u_{i}=-f(x,u_{i})+f_{i}(x)\quad &\text{in } \Omega\\
         u_{i}=g &\text{in }\R^{n}\setminus\overline{\Omega}\\
         E_{\Omega}(u_{i})=0 &\text{on }\partial\Omega.
    \end{cases}
    \end{equation}
    As $f_{i}\geq 0$, $v=0$ is a sub-solution of \eqref{eqn interior perturbation}. Let $u_{i}$ be the solution obtained in the first case for each $i$. 
    
     To obtain a convergent subsequence, it suffices to achieve local equicontinuity and a local uniform bound. Similar to the previous case, let $v_{i}$ be the solution of \eqref{eqn main linear omega} with data $(f_i,g,0)$.
    Clearly $u_{i}$ is a sub-solution of \eqref{eqn main linear omega}. Therefore, by Lemma \ref{lemma comparison principle}, $0\leq u_{i}\leq v_{i}$. Further, for any $x\in K$
    \begin{equation*}
        \begin{split}
            |v_{i}(x)|\leq \int_{\Omega}|f_{i}(y)|\gk{\Omega}(x,y)\ dy+\Big|\int_{\R^{n}\setminus\overline{\Omega}}g(y)\slap \gk{\Omega}(x,y)\ dy\Big|. 
        \end{split}
    \end{equation*}
    We note that 
    \begin{equation*}
         \int_{\Omega}|f_{i}(y)|\gk{\Omega}(x,y)\ dy\leq \frac{\displaystyle\sup_{\theta\in \partial\Omega}h(\theta) \mathcal{H}(\partial\Omega)}{|\Omega_{i}|}\int_{\Omega_{i}}\frac{\gk{\Omega}(x,y)}{\delta_{\Omega}^{s}(y)}\ dy
    \end{equation*}
    For all $i$ such that $K\cap \Omega_{i}=\emptyset$, and $d(K,\Omega_{i})>\alpha>0$,  one can get 
      \begin{equation*}
      \begin{split}
                \int_{\Omega}|f_{i}(y)|\gk{\Omega}(x,y)\ dy &\leq \frac{\displaystyle\sup_{\theta\in \partial\Omega}h(\theta) \mathcal{H}(\partial\Omega)}{|\Omega_{i}|}\int_{\Omega_{i}}\frac{\delta_{\Omega}^{s}(x)}{|x-y|^{n}}\ dy 
                \leq \frac{1}{\alpha^{n}}\displaystyle\sup_{\theta\in \partial\Omega}h(\theta) \mathcal{H}(\partial\Omega)\delta_{\Omega}^{s}(x).
      \end{split}
 \end{equation*}
 Using Lemma \ref{lemma bound of the sln ug} with the above calculation, the sequence $\{v_{i}\}$ is uniformly bounded on $K$. See also \cite[Lemma 2.9]{Abatangelo2015LargesHarmonic}. Therefore, $\{u_{i}\}$ is a bounded sequence in $L^{1}(K)$.

 For equicontinuity, let $z\in K$, choose $\rho>0$ such that $B_{\rho}(z)\subset \Omega$, then for every $x\in B_{\rho}(z)\cap K$, by the second part of \ref{F2} 
 \begin{equation*}\label{eqn equacontinuity}
 \begin{split}
     |u_{i}(x)-u_{i}(z)|&\leq \int_{\Omega}|\gk{\Omega}(x,y)-\gk{\Omega}(z,y)|\left(f(y,u_{i}(y))+f_{i}(y)\right)\ dy\\
     &\leq \left\{\int_{B_{\rho}(z)}+\int_{\Omega\setminus B_{\rho}(z)}\right\} |\gk{\Omega}(x,y)-\gk{\Omega}(z,y)|\left((a_{1}+a_{2}v_{i}^{p})+\frac{C\chi_{\Omega_{i}}(y)}{|\Omega_{i}|} \delta_{\Omega}^{-s}(y)\right)\ dy.
 \end{split}
 \end{equation*}
  Using the boundedness of $v_{i}$ in $B_{\rho}(z)$ (c.f. Lemma \ref{lemma bound of the sln ug}) and \eqref{eqn behaviour of greens kernel}, for $i$ large such that $B_{\rho}(z)\cap \Omega_{i}=\emptyset$ we obtain
 \begin{equation*}
     \int_{B_{\rho}(z)}|\gk{\Omega}(x,y)-\gk{\Omega}(z,y)|(a_{1}+a_{2}v_{i}^{p})\ dy\leq C(\Omega, K,n,s)\int_{B_{\rho}(z)}\left\{\frac{1}{|x-y|^{n-2s}}+\frac{1}{|z-y|^{n-2s}}\right\} dy <+\infty.
 \end{equation*}
 Therefore, given $\epsilon>0$, choose $\rho>0$ independently of $i$ such that 
 \begin{equation*}
     \int_{B_{\rho}(z)}|\gk{\Omega}(x,y)-\gk{\Omega}(z,y)|(a_{1}+a_{2}v_{i}^{p})\ dy\leq \epsilon/2.
 \end{equation*}
 For $y\in \Omega\setminus B_{\rho}(z)$, $|x-y|>\alpha$ and $|z-y|>\alpha$ for some $\alpha>0$, using \eqref{eqn behaviour of greens kernel}, since $\Omega$ is bounded we have
 \begin{equation*}
        \int_{\Omega\setminus B_{\rho}(z)}|\gk{\Omega}(x,y)-\gk{\Omega}(z,y)|(a_{1}+a_{2}v_{i}^{p})\ dy\leq \frac{C(\Omega,K,n,s)}{\alpha^{n}}\int_{\Omega\setminus B_{\rho}(z)}\delta_{\Omega}^{s}(y)\left((a_{1}+a_{2}v_{i}^{p})\right)\ dy.
 \end{equation*}
By Lemma \ref{lemma equivalent E}, $Ev_{i}=0$ implies 
$v_{i}(x)\leq C \delta^{s-1}(x)$ in  $\Omega$. Thus 
\begin{equation*}
    I_{i}:=\int_{\Omega\setminus B_{\rho}(z)}|\gk{\Omega}(x,y)-\gk{\Omega}(z,y)|(a_{1}+a_{2}v_{i}^{p})\ dy\leq C\int_{\Omega\setminus B_{\rho}(z)} a_{1}\delta_{\Omega}^{s}(y)+a_{2}\delta_{\Omega}^{p(s-1)+s}(y)\ dy.
\end{equation*}
 The integral on the right-hand side is finite as $\Omega$ is bounded and $p<(1+s)/(1-s)$. By the dominated convergence theorem
 \begin{equation*}
   \sup_{i} I_{i}\leq C\int_{\Omega\setminus B_{\rho}} |\gk{\Omega}(x,y)-\gk{\Omega}(z,y)|(a_{1}+a_{2}\delta_{\Omega}^{p(s-1)}(y))\ dy\to 0 \quad\text{as }|x-y|\to 0.
 \end{equation*}
 By choosing $\delta_{1}>0$ such that whenever $|x-z|<\delta_1$ we obtain $\sup_{i} I_{i}<\epsilon/4$. Next, since $|\Omega_{i}|=\mathcal{H}(\partial\Omega)/i$ and by the definition of $\Omega_{i}$ we have $i<2/\delta_{\Omega}(y)$ for all $y\in \Omega_{i}$, we obtain using \eqref{eqn behaviour of greens kernel} that
 \begin{equation*}
    J_{i}:= \int_{\Omega\setminus B_{\rho}(z)}|\gk{\Omega}(x,y)-\gk{\Omega}(z,y)|\frac{C\chi_{\Omega_{i}}}{|\Omega_{i}|\delta_{\Omega}^{s}(y)}\ dy< C\int_{\Omega_{i}} \frac{dy}{\delta_{\Omega}(y)}<+\infty.
 \end{equation*}
 Indeed, we have $n\geq 2$. By the dominated convergence theorem, we obtain
 \begin{equation*}
     \sup_{i}J_{i}\leq C\int_{\Omega\setminus B_{\rho}} |\gk{\Omega}(x,y)-\gk{\Omega}(z,y)| \delta_{\Omega}^{-1-s}(y)\ dy\to 0\quad \text{as } |x-z|\to 0.
 \end{equation*}
By choosing $\delta_{2}>0$ such that whenever $|x-z|<\delta_2$ we obtain $\sup_{i} J_{i}<\epsilon/4$.
 Finally, given $\epsilon>0$ choose $\rho>0$ small enough $(\rho=\min\{\rho,\delta_{i}\})$ such that $|u_{i}(x)-u_{i}(z)|\leq \epsilon$ for all $i$ large whenever $|x-z|<\rho$.

 Thus, by compact exhaustion of $\Omega$, and the fact that $u_{i}\leq C\delta_{\Omega}^{s-1}$, there exists a function $u\in L_{loc}^{\infty}(\Omega)$ such that up to a subsequence  $u_{i_{j}}\to u$ in $L^{\infty}(K)$ for every $K$ compact in $\Omega$. Then, $u\geq 0$ a.e. and for every $\psi\in L_{c}^{\infty}(\Omega)$, by uniform convergence $ \displaystyle \int_{\Omega}u_{i_{j}}\psi\to \int_{\Omega}u\psi, \text{as } j\to \infty.$ 
Indeed, by the definition of the weak solution
 \begin{align*}
     \int_{\Omega}u_{i_{j}}\psi =  \int_{\Omega}f_{i_{j}}(y)\g{\Omega}(\psi)(y)\ dy - \int_{\R^{n}\setminus\overline{\Omega}}g(y) \slap \g{\Omega}(\psi)(y)\ dy - \int_{\Omega}f\left(y,u_{i_{j}}(y)\right)\g{\Omega}(\psi)(y)\ dy
      : =I_{1}^{j}+I_{2}+I_{3}^{j}.
 \end{align*}
Using \cite[Proposition 3.18]{Abatangelo2015LargesHarmonic} we obtain
 \begin{equation}\label{interier estimate for h}
     I_{1}^{j}:=\int_{\Omega}f_{i_{j}}(y)\g{\Omega}(\psi)(y)\ dy\to  \int_{\partial\Omega}h(\theta) \Ds[\g{\Omega}(\psi)](\theta) \ d\mathcal{H}(\theta).
 \end{equation}
 Since $u_{i}\leq v_{i}$ and $\psi\in L^{\infty}_{c}(\Omega)$, we have by Proposition \ref{prop greens kernal continuity}, $\g{\Omega}(\psi)=\delta_{\Omega}^{s}\phi$ for some $\phi\in L^{\infty}(\Omega)$. Then
 \begin{equation*}
    |f\left(y,u_{i_{j}}(y)\right)\g{\Omega}(\psi)(y) |\leq \left(a_{1}+a_{2} v_{i_{j}(y)}^{p}\right)C \delta^{s}(y).\quad \text{for }y\in \Omega. 
 \end{equation*}
As $E_{\Omega}v_{i}=0$, we have $v_{i}(x)\leq C \delta^{s-1}(x)$ in $\Omega$, which implies $ v_{i}(x)^{p}\delta_{\Omega}^{s}(x)\leq C\delta^{(s-1)p}(x)$. 

The majorant is integrable if $(s-1)p+s> -1$, which holds as $p< (1+s)/(1-s)$. Therefore, we will have by the dominated convergence theorem 
\begin{equation*}
    I_{3}^{j}\to -\int_{\Omega}f\left(y,u(y)\right)\g{\Omega}(\psi)(y)\ dy.
\end{equation*}
Thus 
\begin{equation*}
    \int_{\Omega}u\psi=-\int_{\Omega}f\left(y,u(y)\right)\g{\Omega}(\psi)(y)\ dy-\int_{\R^{n}\setminus\overline{\Omega}}g(y) \slap \g{\Omega}(\psi)(y)\ dy +\int_{\partial\Omega}h(\theta) \Ds[\g{\Omega}(\psi)](\theta) \ d\theta.
\end{equation*}
which completes the proof for $h\in C(\partial\Omega)$. 
\smallskip 

\textbf{Step 3:} Let $h\in L^{1}(\partial\Omega)$. Choose $h_{k}\in C(\partial\Omega)$ such that $h_{k}\to h$ in $L^{1}(\partial\Omega)$. Let $u^{k}$ be the corresponding solution to data $h_{k}$ obtained in the previous case. Since the approximating sequence $u_{i}^{k}$ of $u^{k}$ is such that $u_{i}^{k}\leq v_{i}^{k}$ and that $v_{i}^{k}$ is locally uniformly bounded, for any compact subset $K$ of $\Omega$, $u^{k}_{i}(x)\leq C(K)$ for all $x\in K$. Thus, $u^{k}\leq C(K)$ on $K$. Further, $\{u^{k}\}_{k}$ is equicontinuous for $k$ large. Thus, there exists a function $u\in L^{1}_{loc}(\Omega)$  such that, up to a subsequence $u^{k}\to u$ uniformly in any compact set in $\Omega$. Then $u\geq 0$ a.e..
The result follows from the dominated convergence theorem.
\end{proof} 
\section{Acknowledgment}
I.C. acknowledges the grant of  ANRF PMECRG (ANRF/ECRG/2024/004743/PMS).
N.N.D. is supported by PMRF grant (2302262). 

\printbibliography
\end{document}